\documentclass{acta_2017}
\usepackage{harvard_acta}
\usepackage{amssymb}
\usepackage{amsmath}
\usepackage{amsfonts, amssymb}
\usepackage{graphicx}
\usepackage{bm, color}
\usepackage{booktabs, threeparttable, multirow}

\newcommand{\newcom}{\newcommand}
\newcom{\al}{\alpha}
\newcom{\be}{\beta}
\newcom{\eps}{\epsilon}
\newcom{\ve}{\varepsilon}
\newcom{\e}{\varepsilon}
\newcom{\ga}{\gamma}
\newcom{\Ga}{\Gamma}
\newcom{\ka}{\kappa}
\newcom{\Lam}{\Lambda}
\newcom{\lam}{\lambda}
\newcom{\Om}{\Omega}
\newcom{\om}{\omega}
\newcom{\Si}{\Sigma}
\newcom{\si}{\sigma}
\newcom{\tht}{\theta}
\newcom{\dtri}{\nabla}
\newcom{\tri}{\triangle}
\newcom{\oo}{\infty}
\newcom{\vphi}{\varphi}
\newcom{\CA}{{\mathcal A}}
\newcom{\cB}{{\mathcal B}}
\newcom{\cC}{{\mathcal C}}
\newcom{\cD}{{\mathcal D}}
\newcom{\cF}{{\mathcal F}}
\newcom{\CG}{{\mathcal G}}
\newcom{\CH}{{\mathcal H}}
\newcom{\cL}{{\mathcal L}}
\newcom{\cM}{{\mathcal M}}
\newcom{\cN}{{\mathcal N}}
\newcom{\CP}{{\mathcal P}}
\newcom{\CR}{{\mathcal R}}
\newcom{\cS}{{\mathcal S}}
\newcom{\cQ}{{\mathcal Q}}
\newcom{\cT}{{\mathcal T}}
\newcom{\CU}{{\mathcal U}}
\newcom{\cY}{{\mathcal Y}}
\newcom{\cZ}{{\mathcal Z}}
\newcom{\T}{\mathbb T}
\newcom{\BT}{{\mathbb{T}^2}}
\newcom{\Z}{\mathbb Z}
\newcom{\C}{\mathbb C}
\newcom{\E}{\mathbb E}
\newcom{\dd}{\mathrm d}
\newcom{\x}{\mathbf x}
\newcommand{\vc}[1]{{\bf #1}}
\newcom{\vN}{\vc{N}}
\newcom{\vn}{\vc{n}}
\newcom{\vG}{\vc{G}}
\newcom{\vF}{\vc{F}}
\newcom{\vf}{\vc{f}}
\newcom{\vg}{\vc{g}}
\newcom{\vq}{\vc{q}}
\newcom{\vu}{\vc{u}}
\newcom{\vw}{\vc{w}}
\newcom{\vb}{\vc{b}}
\newcom{\vh}{\vc{h}}
\newcom{\vz}{\vc{z}}
\newcom{\vup}{\vu^{+}}
\newcom{\vum}{\vu^{-}}
\newcom{\vvp}{\vv^{+}}
\newcom{\vvm}{\vv^{-}}
\newcom{\vbp}{\vb^{+}}
\newcom{\vbm}{\vb^{-}}
\newcom{\vhp}{\vh^{+}}
\newcom{\vhm}{\vh^{-}}
\newcom{\Omp}{{\Om^+}}
\newcom{\Omm}{{\Om^-}}
\newcom{\vupm}{{\vu^{\pm}}}
\newcom{\vvpm}{{\vv^{\pm}}}
\newcom{\vbpm}{{\vb^{\pm}}}
\newcom{\vhpm}{{\vh^{\pm}}}
\newcom{\vwp}{{\vc{w}^+}}
\newcom{\vwm}{{\vc{w}^-}}
\newcom{\vwpm}{{\vc{w}^{\pm}}}
\newcom{\Ompm}{{\Omega^{\pm}}}
\newcom{\vom}{\boldsymbol{\omega}}
\newcom{\vvap}{\boldsymbol{\varpi}}
\newcom{\vop}{\vom^{+}}
\newcom{\vnu}{\boldsymbol{\nu}}
\newcom{\vopm}{\vom^{\pm}}
\newcom{\vjp}{\vj^+}
\newcom{\vjm}{\vj^-}
\newcom{\vjpm}{\vj^{\pm}}
\newcom{\vj}{\boldsymbol{\xi}}
\newcom{\Ds} {\langle\nabla\rangle^{s-\f12}}
\newcom{\Qvec}{\mathbf{Q}}
\newcom{\xx}{\mathbf{x}}
\newcom{\ee}{\mathbf{e}}
\newcom{\yy}{\mathbf{y}}
\newcom{\ii}{\mathbf{i}}
\newcom{\jj}{\mathbf{j}}
\newcom{\kk}{\mathbf{k}}
\newcom{\nn}{\mathbf{n}}
\newcom{\mm}{\mathbf{m}}
\newcom{\pp}{\mathbf{p}}
\newcom{\hh}{\mathbf{h}}
\newcom{\uu}{\mathbf{u}}
\newcom{\vv}{\mathbf{v}}
\newcom{\ww}{\mathbf{w}}
\newcom{\DD}{\mathbf{D}}
\newcom{\FF}{\mathbf{F}}
\newcom{\BB}{\mathbf{B}}
\renewcommand{\AA}{\mathbf{A}}
\newcom{\MM}{\mathbf{M}}
\newcom{\NN}{\mathbf{N}}
\newcom{\II}{\mathbf{I}}
\newcom{\PP}{\mathbf{P}}
\newcom{\QQ}{\mathbf{Q}}
\newcom{\WW}{\mathbf{W}}
\newcom{\Ue}{\mathcal{U}_\ve}
\newcom{\Be}{\mathcal{B}_\ve}
\newcom{\mue}{\mathcal{\mu}_\ve}
\newcom{\sqe}{\sqrt{\ve}}
\newcom{\BR}{{\mathbb{R}^3}}
\newcom{\BS}{{\mathbb{S}^2}}
\newcom{\BOm}{\mathbf{\Omega}}
\newcom{\tr}{\mathrm{tr}}
\newcommand{\ud}{{\rm d}}

\newcom{\ds}{{\rm d} s}
\newcom{\f}{\frac}
\newcom{\di}{\displaystyle\int}
\newcom{\dl}{\displaystyle\lim}
\newcom{\ov}{\overline}
\newcom{\sset}{\subset}
\newcom{\wt}{\widetilde}
\newcom{\pa}{\partial}
\newcom\na{\nabla}
\newcom{\suml}{\sum\limits}
\newcom{\supl}{\sup\limits}
\newcom{\intl}{\int\limits}
\newcom{\infl}{\inf\limits}
\newcom{\disp}{\displaystyle}
\newcom{\non}{\nonumber}
\newcom{\no}{\noindent}
\newcom{\QED}{$\square$}

\def\div{\mathop{\rm div}\nolimits}

\def\ef{\hphantom{MM}\hfill\llap{$\square$}\goodbreak}

\def\eqdefa{\buildrel\hbox{\footnotesize def}\over =}

\newtheorem{Proposition}{Proposition}[section]
\newtheorem{Lemma}{Lemma}[section]

\newtheorem{theorem}{Theorem}[section]
\newtheorem{lemma}{Lemma}[section]

\newtheorem{corollary}{Corollary}[section]

\title[]{Modeling and Computation of Liquid Crystals}

\author[{Acta Numerica}]{%
 Wei Wang\\
Department of Mathematics, Zhejiang University, Hangzhou 310027, China\\
{\tt wangw07@zju.edu.cn}\\
\and
Lei Zhang\\
Beijing International Center for Mathematical Research,\\
Center for Quantitative Biology, Peking University, Beijing 100871, China\\
{\tt zhangl@math.pku.edu.cn}\\
\and
Pingwen Zhang\\
School of Mathematical Sciences, Peking University, Beijing 100871, China\\
{\tt pzhang@pku.edu.cn}
}

\pagerange{\pageref{firstpage}--\pageref{lastpage}}
\doi{XXXXXX}

\begin{document}

\label{firstpage}
\maketitle

\begin{abstract}
\vspace{3cm}
Liquid crystal is a typical kind of soft matter that is intermediate between crystalline solids and isotropic fluids. The study of liquid crystals has made tremendous progress over the last four decades, which is of great importance on both fundamental scientific researches and widespread applications in industry. In this paper, we review the mathematical models and their connections of liquid crystals, and survey the developments of numerical methods for finding the rich configurations of liquid crystals.
\end{abstract}

\tableofcontents
\vspace{3mm}

\section{Introduction}\label{sec:1}

Liquid crystals (LCs) are classical examples of partially ordered materials that translate freely as liquid and exhibit some long-range order above a critical concentration or below a critical temperature. The anisotropic properties lead to anisotropic mechanical, optical and rheological properties \cite{de1993physics,stewart}, and make LCs suitable for a wide range of commercial applications, among which the best known one is in LC display industry \cite{kitson_geisow_apl,majumdar_pre_2007}. LCs also have substantial applications in nanoscience, biophysics, materials design, etc. Furthermore, LC is a typical system of complex fluids, hence the theoretical approaches or technical tools of LC systems can be applied in the study beyond the specific field of LCs, such as surface/interfacial phenomena, active matter, polymers, elastomers and colloid science \cite{Teramoto2010Morphological,Marcus2012Transition,Cai2017Liquid}.

LCs are mesophases between anisotropic crystalline (Fig. \ref{fig:LC} (a)) and isotropic liquid (Fig. \ref{fig:LC} (e)).
There are three major classes of LCs-- the nematic, the cholesteric, and the smectic \cite{Friedel1922Les}.
The simplest phase is the nematic phase (Fig. \ref{fig:LC} (d)), where there is a long-range orientational order, i.e., the molecules almost align parallel to each other, but no long-range correlation to the molecular center of mass positions.
On a local scale, the cholesteric (Fig. \ref{fig:LC} (c)) and nematic orders are similar, while, on a larger scale the director of cholesteric molecules follows a helix with a spatial period. The nematic liquid crystal is a special cholesteric liquid crystal with no helix.
Smectics (Fig. \ref{fig:LC} (b)) have one degree of translational ordering, resulting in a layered structure. As a consequence of this partial translational ordering, the smectic phases are much more viscous and more close to crystalline than either nematic phase or cholesteric phase.

\begin{figure}[hbt]
    \begin{center}
    \includegraphics[width=\columnwidth]{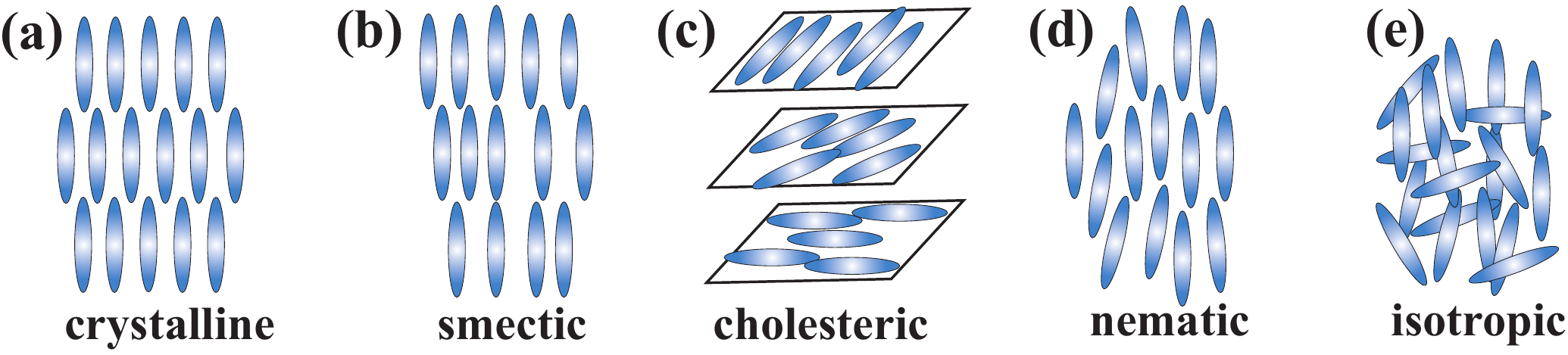}
    \caption{Schematic representation of (a) crystalline, (b) smectic, (c) cholesteric, (d) nematic, and (e) isotropic phases with rod-like molecules.}
    \label{fig:LC}
    \end{center}
\end{figure}

The most widely studied system of LCs is rod-like nematic liquid crystal (NLC), of which molecules are rod-shape and rigid.
In such a NLC, the molecules may move freely like a liquid, but its molecules in a local area may tend to align along a certain direction,  which makes the liquid being anisotropic.
In order to describe the anisotropic behaviour of NLC,
one has to choose appropriate functions, called {\it order parameters} in the community of physics. There are various ways for choosing order parameters, which lead to mathematical theories at different levels, ranging from microscopic molecular theories to macroscopic continuum theories.

The first type of model is the vector model, including the Oseen-Frank theory \cite{oseen1933theory,frank1958liquid} and the Ericksen's theorem \cite{Ericksen1990}. In these models, it is assumed that there exists a locally preferred direction
 $\nn(\xx)\in\mathbb{S}^2$ (the unit sphere in the three dimensional space)
 for the alignment of LC molecules at each material point $\xx$.
This setting is rough but works very well in many situations, so the vector theory has been widely used in LC community for its simplicity. However, {vector theories have that drawback that it does not respect the head-to-tail symmetry of rod-like molecular, in which $-\nn$ should be equivalent to $\nn$ \cite{ball2008orientable}.  This drawback may lead to a incorrect description of some systems, especially when defects are present.}

The second one is the molecular model, which was proposed by  \citeasnoun{onsager1949effects} to characterize
the nematic-isotropic phase transition and then developed by  \citeasnoun{doi1981molecular} to study the LC flow.
In this theory, the alignment behavior is described by an orientational distribution function $f(\xx,\mm)$ which represents the number density of molecules with orientation $\mm \in \mathbb{S}^2$ at a material point $\xx$. Since the distribution function $f$ contains much more information on the molecular alignment, the molecular models can provide more
accurate description. However, the computational cost is usually very expensive as it often involves solving high dimensional problems.

The third type of model is the $\QQ$-tensor model, including the Landau-de Gennes (LdG) theory \cite{de1993physics}, which uses
a traceless symmetric $3\times3$ matrix $\QQ(\xx)$ to describe the alignment of LC molecules at the position $\xx$. In a physical viewpoint, the order parameter $\QQ$-tensor, is related to the second moment of the orientational distribution function $f(\xx,\mm)$.
It does not assume that the molecular alignment has a preferred direction and thus can describe the biaxiality.

The vector theory and tensor theory are called macroscopic theories, which are based on continuum mechanics, while the molecular theory is a microscopic one that is derived from the viewpoint of statistical mechanics.
Although they were proposed from different physical viewpoints, all of them play important roles and have been widely used in studies of LCs. Understanding these models and their relationships becomes an important issue for LC studies. Moreover, the coefficients in macroscopic theories are phenomenologically determined and their interpretations in terms of basic physical measurements remain unclear. By exploring their relationships, one can determine these coefficients in terms of the molecular parameters, which provide a clear physical interpretation rather than phenomenological determination.
From the aspect of mathematical modeling, many efforts have been addressed on their relationships especially on microscopic foundations of macroscopic theories.
However, little work has been done from the analytical side until some progresses have been made during the past ten years.
New experimental works and theoretical paradigms call for major modeling and analysis efforts.

A particularly intriguing feature of LCs is the topologically induced defects \cite{kleman1989defects,ball2017liquid}. Defects are discontinuities in the alignment of LCs. They are classified as point defects, disclination lines, and surface defects, and further classified by the topological degree of defects, including  two-dimensional (2D) $\pm1$ and $\pm1/2$ point defects (Fig. \ref{fig:defect}). Defects are energetically unfavorable because the existence of defects will increase the elastic energy of the nearby LCs. However, defects are unavoidable due to the environment such as external fields (electric field and magnetic field) \cite{oh1995electro}, geometric constraints (boundary condition and domain) \cite{de2007point}, unsmooth boundary (polygon corner) \cite{han2020pol}, etc.
NLC is a typical system to study defects, rich static structures and dynamic processes relevant to defects. Defects have the property of being isotropic and are surrounded by NLC molecules, and hence are able to induce the phase transition between NLC phases and isotropic \cite{mottram2000defect,mottram1997disclination}. The multiplicity of defect patterns also guide the design of new multi-stable LC display device \cite{willman2008switching}.

\begin{figure}
    \begin{center}
    \includegraphics[width=\columnwidth]{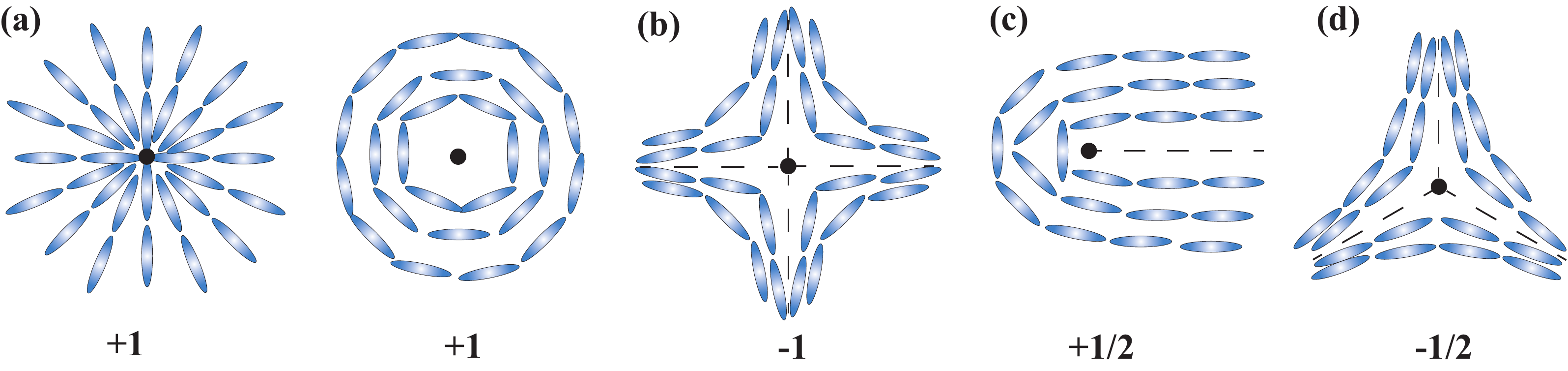}
    \caption{(a-d) are schematic diagrams of $\pm1$ and $\pm1/2$ 2D point defects, respectively.}
    \label{fig:defect}
    \end{center}
\end{figure}

A topologically confined NLC system can admit multiple stable equilibria, which usually correspond to different defect patterns~\cite{kralj1999biaxial,wang2017topological,robinson2017molecular}. The energy landscape of the NLC system, upon which the equilibrium states are located, is determined by the properties of the LC material as well as the environment, such as temperature, size and shape of confining space, external field, etc. Tremendous experimental and theoretical studies have been made to investigate the defect patterns in NLCs \cite{muvsevivc2006two,lubensky1998topological,onsager1949effects,bajc2016mesh}.

From a numerical perspective, there are two approaches to compute stable defect patterns. One is the energy-minimization based approach \cite{cohen1987minimum,alouges1997new,adler2015energy,nochetto2017finite,majumdar2018remarks,gartland1991numerical}, which is often numerically solved by Newton-type or quasi-Newton method.
The other approach is to follow the gradient flow dynamics driven by the free energy corresponding to individual model of LCs \cite{fukuda2004interaction,ravnik2009landau,canevari2017order,wang2019order,macdonald2020moving}. Various efficient numerical methods have been developed to solve the gradient flow equations, including energy stable numerical schemes such as convex splitting method \cite{Elliott1993global}, invariant energy quadratization method \cite{Yang2016linear}, scalar auxiliary variable method \cite{Shen2018sav}, etc.
Furthermore, machine learning recently becomes an emerging approach in the field of soft matter including LCs \cite{walters2019machine}.

There are also extensive numerical developments for the LC hydrodynamics to simulate LC flows, LC droplets,  colloid LC composites, etc \cite{badia2011overview,foffano2014dynamics}.
Various numerical studies of the NLC dynamics are performed by applying the Ericksen--Leslie equations \cite{liu2000approximation,becker2008finite}, the hydrodynamic $\mathbf{Q}$-tensor models \cite{BerisEdwards,zhao2016semi}, and the molecular models based on the extended Doi kinetic theory  \cite{doi1988theory,ji2008kinetic}.

With the existence of multiple stable or metastable states in the LC systems, it may transit from one stable equilibrium to another under thermal fluctuation or external perturbation, causing the position and topology of defect to change drastically~\cite{kusumaatmaja2015free,gupta2009size}.
Thus it is important, both experimentally and theoretically, to determine when and how such phase transition occurs.
In the zero-temperature limit, the phase transition connecting two stable defects follows the so called minimal energy path (MEP), which has the lowest energy barrier among all possible paths. The transition state corresponds to the state with the highest energy along the MEP, i.e., index-1 saddle point~\cite{zhang2016optimization}.
Finding accurate critical nuclei and transition pathways is a challenging problem due to the anisotropic nature of the problems and the existence of a number of length scales.
There are two typical approaches to compute transition states and transition pathways. One is the surface-walking method, such as the gentlest ascent dynamics \cite{weinan2011gentlest} and the dimer type method \cite{henkelman1999dimer}, the other is path-finding method, such as the string method \cite{weinan2002string} and the nudged elastic band method \cite{jonsson1998nudged}.

Besides local minimizers and transition states, there is substantial recent interest in high-index saddle points with multiple unstable directions on the LC energy landscape, which are stationary solutions of the Euler-Lagrange equation corresponding to the LC free energy, and the Morse index is the number of negative eigenvalues of the corresponding Hessian of the free energy \cite{milnor1969morse}.
In recent years, a number of numerical algorithms have been developed to find multiple solutions of nonlinear equations, including the minimax method \cite{li2001minimax}, the deflation technique \cite{farrell2015deflation}, the eigenvector-following method \cite{doye2002saddle}, and the homotopy method \cite{mehta2011finding}.
Despite substantial progress in this direction, the relationships between different solutions are unclear.
In a recent work \cite{yin2019high}, the high-index saddle dynamics was proposed to efficiently compute any-index saddle points. By applying the high-index optimization-based shrinking dimer method, a solution landscape, which is a pathway map of all connected solutions, can be constructed for the NLCs confined on a square domain \cite{yin2020construction}.

The rest of the paper is organized as follows. The mathematical models of LC, including molecular models, vector models, and tensor models, will be introduced in Section 2. Mathematical analysis and connections between different LC models will be discussed in Section 3. For the numerical computation of LCs, we will review numerical methods for computing stable defects of LC in Section 4 and LC hydrodynamics in Section 5. In Section 6, we will introduce the numerical algorithms to compute the transition pathways between difference LCs and the solution landscapes of LC systems. Section 7 will conclude with an outlook for trends and future developments of LCs.

\section{Mathematical models of liquid crystals}\label{sec:2}
In this section, we review the three typical theories of LCs, i. e., molecular models, vector models, and tensor models separately.

\subsection{Molecular model}
\subsubsection{The static Onsager theory}
\citeasnoun{onsager1949effects} proposed a classical model which can predict the isotropic-nematic phase transition for rod-like LCs.
The theory is based on a orientational distribution function $f(\mm)$ which represents the number density of molecules with orientation $\mm$.
The free energy can be written as
\begin{equation}\label{eq:free energy-homo}
A[f]=\int_{\BS}\Big\{f(\mm)\ln{f(\mm)}+\frac12f(\mm)U(\mm)\Big\}\ud\mm.
\end{equation}
The first term comes from the Brownian motion of the rod-like molecules, and the second term is the interaction term where $U(\mm)$ is the mean-field interaction potential
\begin{align}\label{potential}
 U(\mm)=(U[f])(\mm):=\int_{\BS}B(\mm, \mm')f(\mm')d\mm',
\end{align}
where $B(\mm,\mm')$ is the interaction potential between two molecules with orientation $\mm$ and $\mm'$.
Onsager introduced the potential $B(\mm,\mm')$ with the form
\begin{align}\label{potential-Ons}
B(\mm, \mm')=\alpha~|\mm\times\mm'|,
\end{align}
which is calculated based on the excluded volume potential.
 \citeasnoun{maier1958einfache} proposed a similar interaction potential, now known as the
Maier-Saupe potential:
\begin{align}\label{potential-MS}
B(\mm, \mm')=\alpha~|\mm\times\mm'|^2.
\end{align}
The parameter $\alpha$ represents the density (for lyotropic LCs)
or the inverse of absolute temperature (for thermotropic LCs).

The energy functional (\ref{eq:free energy-homo})-(\ref{potential}) with (\ref{potential-Ons})
 or (\ref{potential-MS}) are called the Onsager energy or the Maier-Saupe energy respectively.
 Both of them characterize the competition between the entropy and interaction energy and
 can effectively describe the nematic-isotropic phase transition.
 If the temperature is high or the density is dilute, the entropy
 term dominates the energy and the minimizer is the constant distribution $f(\mm)=1/4\pi$, which describes the isotropic phase.
On the contrary, if the temperature is low or the density is large, the  energy is dominated by the interaction term
and will be minimized by an axially symmetric distribution $f(\mm)=f_0(\mm\cdot\nn)$. This case corresponds to the nematic phase in which
molecules prefer a uniform alignment.
{The rigorous proof for the Maier-Saupe energy  is given independently by \citeasnoun{LiuZhangZhang2005CMS} and \citeasnoun{fatkullin2005critical}.
Another  proof is given in \cite{Zhou2005}.
Precisely, in these papers, the following theorem was proved:}
\begin{theorem}\label{thm:critical point}
All the critical points of the Maier-Saupe energy ((\ref{eq:free energy-homo})-(\ref{potential}) with (\ref{potential-MS})) are given by
$$h_{\eta,\nn}(\mm)=\frac{\mathrm{e}^{\eta(\mm\cdot\nn)^2}}{\int_\BS\mathrm{e}^{\eta(\mm\cdot\nn)^2}\ud\mm}, $$
where $\nn$ is an arbitrary unit vector, and $\eta$ equals to $0$ or satisfies
\begin{eqnarray}\label{eta-alpha}
\frac{\int_0^1\mathrm{e}^{\eta z^2}\ud z}{\int_0^1z^2(1-z^2)\mathrm{e}^{\eta{z^2}}\ud{z}}=\alpha.
\end{eqnarray}
Furthermore, there exists
$$\alpha^*=\min_{\eta\in\mathbb{R}}\frac{\int_0^1\mathrm{e}^{\eta z^2}\ud z}{\int_0^1z^2(1-z^2)\mathrm{e}^{\eta{z^2}}\ud{z}}\approx6.731393,$$
 such that
\begin{enumerate}

\item For $\alpha<\alpha^*$, (\ref{eta-alpha}) has no solution. While for $\alpha=\alpha^*$, it has a unique solution
$\eta=\eta^*$.

\item For $\alpha>\alpha^*$, (\ref{eta-alpha}) has exactly two solutions $\eta=\eta_1(\alpha),\eta_2(\alpha)$
satisfying
\begin{itemize}
\item $\eta_1(\alpha)>\eta^*>\eta_2(\alpha)$, $\lim_{\alpha\to\alpha^*}\eta_1(\alpha)=\lim_{\alpha\to\alpha^*}\eta_2(\alpha)=\eta^*$;
\item $\eta_1(\alpha)$ is an increasing function of $\alpha$, while $\eta_2(\alpha)$ is a decreasing function;
\item $\eta_2(7.5)=0$.
\end{itemize}
\end{enumerate}
\end{theorem}
The above theorem gives a  complete classification on all critical points of the Maier-Saupe energy functional. The three kinds of solutions for $\eta=\eta_1, \eta_2, 0$ are named as prolate, oblate, isotropic solutions.  Their stabilities are summarized in the following proposition. The proof can be found in \cite{zhang2007stable,WangZZCPAM}:
 \begin{Proposition}\label{prop:energy stability}
$h=\frac{1}{4\pi}$ ($\eta=0$) is a stable critical point of $A[f]$ if and only
if $\alpha<7.5$; If $\alpha>\alpha^*$, for any $\nn\in\BS$, $h_{\eta_1,\nn}$ is stable,
while $h_{\eta_2,\nn}$ is unstable. Therefore, for $\alpha>7.5$, $h_{\eta_1,\nn}$ are the only minimizers.
\end{Proposition}

Theorem \ref{thm:critical point} and Proposition \ref{prop:energy stability} inform us that:
the oblate soltion is always unstable; if $\alpha>7.5$, the prolate solution is the only stable solution, which corresponds to the nematic phase; if $\alpha<\alpha^*$, the isotropic solution is the only solution; for $\alpha\in(\alpha^*, 7.5)$, the prolate and isotropic solutions are both stable, which
indicates that the isotropic phase and nematic phase can coexist in this parameter region.

Classification of minimizers of the Onsager energy functional is much more difficult, since the interaction
potential is irregular and all even order moments of the orientation distribution function are involved
in the interaction part of the energy.  The axial symmetry of all solutions to the 2D problem is proved by \citeasnoun{ChenLiWang2010}. 
We refer to \cite{vollmer2017critical} and \cite{Ball2020} for some recent progresses on the 3D problem.

\subsubsection{Dynamic Doi theory and its inhomogeneous extension}

The molecular theory has been developed by \citeasnoun{doi1981molecular} to study  the homogeneous LC flow.
Under a given velocity gradient $\nabla \vv$, the evolution of the distribution function is given by the following
 equation:
\begin{eqnarray}\label{eq:homo-Doi}
\frac{\partial{f}(t,\mm)}{\partial{t}}=\f 1 {De}\CR\cdot(\CR{f}+f\CR
U)-\CR\cdot\big(\mm\times\kappa\cdot\mm{f}\big).
\end{eqnarray}
Here $De$ is the Deborah number which characterizes the average time tending to local equilibrium state,
$\CR=\mm\times\nabla_{\BS}$ is the rotational gradient operator on the unit sphere
and $\kappa=(\nabla \vv)^T$ is the transpose of the velocity gradient.
The molecular alignment field in turn induces an extra stress tensor to the bulk fluids which is given by
\begin{eqnarray}\label{eq:tensor-hom}
\sigma^{De}=\frac{1}{2}\DD:\langle\mm\mm\mm\mm\rangle_f-\frac{1}{De}\langle\mm\mm\times\CR\mu\rangle_f,
\end{eqnarray}
where $\DD=(\nabla \vv+(\nabla \vv)^T)/2$  is the strain rate tensor,  $\langle(\cdot)\rangle_f=\int_\BS(\cdot)f\ud\mm $ denotes the average under the distribution $f$, and $\mu$ is the chemical potential:
\begin{align*}
  \mu = \frac{\delta A[f]}{\delta f}=\ln f+U(\mm).
\end{align*}
The equation (\ref{eq:homo-Doi}) has been very successful in describing the
properties of LC polymers in a solvent. This model takes
into account the effects of hydrodynamic flow, Brownian motion, and
intermolecular forces on the molecular orientation distribution.
However, it does not include effects such as distortional elasticity and thus valid only in the limit of spatially homogeneous flows.

\citeasnoun{marrucci1991elastic} extended Doi's theory to the inhomogeneous case by incorporating
 the long-range interaction into the theory.
By using a truncated Taylor series expansion to approximate the nonlocal potential, the elastic energy
is then described by gradients of the second moments of the distribution function.
This method was subsequently developed by many people \cite{fengSL2000theory,wang2002hydrodynamic} to study the inhomogeneous LC flow.
However, instead of using the distribution as the sole order parameter,
these works used a combination of the tensorial order parameter and the distribution
function, and spatial variations are described by the spatial gradients of the tensorial order
parameter, which departs from the original motivation of the kinetic formulation.

\citeasnoun{wang2002kinetic} set up a formalism in which the
interaction between molecules is treated more directly by using the
position-orientation distribution function via interaction
potentials. They extended the free energy (\ref{eq:free energy-homo})
to include the effects of nonlocal intermolecular interactions
through an interaction potential as follows:
\begin{align}\label{energy:nonlocal-onsager}
\begin{aligned}
&A_\ve[f]=\int_{\Omega}\int_{\BS}f(\xx,\mm,t)(\ln{f(\xx,\mm,t)}-1)+\frac{1}{2}U_\ve(\xx,\mm,t)f(\xx,\mm,t)\ud\mm\ud\xx,\\
&U_\ve(\xx,\mm,t)=\CU_\ve{f}:=\int_{\Omega}\int_{\BS}\Be(\xx,\mm; \xx',\mm')f(\xx',\mm',t)\ud\mm'\ud\xx',
\end{aligned}
\end{align}
where $\Be(\xx,\mm; \xx',\mm')$ is the interaction potential between two molecules in the configurations $(\xx, \mm)$ and
$(\xx',\mm')$, which depends on the non-dimensional small parameter ${\ve}=\f {L^2} {L_0^2}$ (here $L$ is the length of the rods and $L_0$ is
the typical size of the flow region). There are two typical choices:
\begin{enumerate}
\item Hard-core excluded volume potential:
\begin{align}\label{potential:nonlocal-Onsager}
\Be(\xx,\mm; \xx',\mm')=
\left\{
\begin{array}{ll}
1, \quad{\small \text{molecule } (\mathbf{x},\mathbf{m}) \text{ is joint with
molecule } (\mathbf{x}',\mathbf{m}')},\\
0, \quad\text{\small disjoint with each other}.
\end{array}
\right.
\end{align}

\item Long-range Maier-Saupe interaction potential:
\begin{align}\label{potential:nonlocal-MS}
\Be(\xx,\mm; \xx',\mm')=\frac{1}{\ve^{3/2}}g\Big(\frac{\xx-\xx'}{\sqrt{\ve}}\Big)\al|\mm\times\mm'|^2,
\end{align}
where $g(\xx)$ is a smooth function on $\mathbb{R}^3$  with $\int_\BR g(\xx)\ud{\xx}=1$,
and the small parameter $\sqe$ represents the typical interaction distance.
\end{enumerate}

Both the potentials are capable to capture the nonlocal interaction between molecules, and thus can describe
distortion effects of the molecular alignment.
The hard-core potential indeed coincides with Onsager's choice adopted in \cite{onsager1949effects}. Note that
different geometric shapes of molecules will lead to different energy forms. For NLC,
the molecules are commonly assumed to be prolate ellipsoids or spherocylinders.
The long range Maier-Saupe interaction potential (\ref{potential:nonlocal-MS}), proposed in \citeasnoun{yu2007kinetic}, can be viewed as a smooth
approximation for the hard-core potential, which is easier to analyze and simulate.

Based on the nonlocal energy, \citeasnoun{wang2002kinetic} presented a inhomogeneous model for the LC flow.
Define the chemical potential as
\begin{align*}
  \mue=\ln{f(\xx,\mm,t)+U_\ve(\xx,\mm,t)}.
\end{align*}
Then the inhomogeneous (non-dimensional) system reads as:
\begin{eqnarray}\label{eq:LCP-non}
&&\frac{\pa{f}}{\pa{t}}+\vv\cdot\nabla{f}=\frac{\ve}{De}\nabla\cdot\big\{
\big(\gamma_{\|}\mm\mm+\gamma_{\bot}(\II-\mm\mm)\big)\cdot f\nabla\mue\big\}\nonumber\\
&&\qquad\qquad+\frac{1}{De}\CR\cdot(f\CR \mue)
-\CR\cdot(\mm\times\kappa\cdot\mm{f}),\\\nonumber
&&\frac{\pa{\vv}}{\pa{t}}+\vv\cdot\nabla\vv=-\nabla{p}+\frac{1}{Re}\Big\{\nabla\cdot(\tau^s+\tau^e)+\FF^e_\ve \Big\}
\end{eqnarray}
Here $\vv$ is the fluid velocity, $p$ is the pressure, and   $\gamma_{\|}, \gamma_{\bot}$ are, respectively, the translational diffusion coefficients parallel to and
normal to the orientation of the LCP molecule. $Re$ is the Reynolds number. The viscous stress $\tau^s$, the elastic stress $\tau^e$ and the body force $\FF^e$ are
given by \begin{align*}
&\tau^s=2\gamma\DD+\frac{1-\gamma}{2}\DD:\langle\mm\mm\mm\mm\rangle_f,\\
& \tau^e=-\frac{1-\gamma}{De}\langle\mm\mm\times\CR\mue\rangle_f,\quad
\FF^e=-\frac{1-\gamma}{De}\langle\nabla\mue\rangle_f.
\end{align*}
System (\ref{eq:LCP-non}) has the following energy-dissipation relation:
\begin{align}
&\frac{\ud}{\ud{t}}\Big(\int_{\Omega}\frac{Re}{2(1-\gamma)}|\vv|^2\ud\xx+\frac{1}{De}A_\ve[f]\Big)\nonumber\\
&= - \int_{\Omega}\Big(\frac{\gamma}{1-\gamma}|\DD|^2+\frac{1}{2}\big\langle(\mm\mm:\DD)^2\big\rangle_f
+\frac{1}{De^2}\big\langle|\CR\mue|^2\big\rangle_f \quad\nonumber \\
&\qquad+\frac{\ve}{De^2(1-\gamma)}\big\langle\nabla\mue\cdot(\gamma_{\|}\mm\mm+\gamma_\bot(\II-\mm\mm))\cdot\nabla\mue\big\rangle_f
\Big)\ud\xx,
\end{align}
We refer to \citeasnoun{yu2007kinetic} for the numerical study and \citeasnoun{ZhangZhang-SIAM} for the well-posedness of the system (\ref{eq:LCP-non}).

\subsection{Vector Theories}

\subsubsection{Static vector models: the Oseen-Frank theory}

The molecular theories provide a detailed description for LCs, however, it is not convenient to use.
The simplest model to study the equilibrium configuration for NLCs is the Oseen-Frank model, which is proposed by
 \citeasnoun{oseen1933theory} and \citeasnoun{frank1958liquid}. It neglects the molecular details and use  a unit vector $\nn(\xx)$ to describe the
average orientation of LCs molecules at position $\xx$.
Then the distortion energy, which is called as Oseen-Frank energy, takes the following form:
\begin{align}\nonumber
E_{OF}(\nn,\nabla\nn)=&~\frac{k_1} 2(\na\cdot\nn)^2+\frac{k_2}2|\nn\cdot(\na\times\nn)|^2+\frac{k_3} 2|\nn\times(\na\times \nn)|^2\\
&+\frac{(k_2+k_4)}2\big(\textrm{tr}(\na\nn)^2-(\na\cdot\nn)^2\big). \label{energy:Oseen-Frank}
\end{align}
The constants $k_1, k_2, k_3$ represent modules  for  three different kinds of pure deformation respectively: {\it splay}, {\it twist} and {\it bending},
which are illustrated in Figure \ref{fig:distortion}. For prolate nematics, one often has
$$k_3>k_1>k_2>0.$$

\begin{figure}\label{fig:distortion}
    \begin{center}
    \includegraphics[width=0.8\columnwidth]{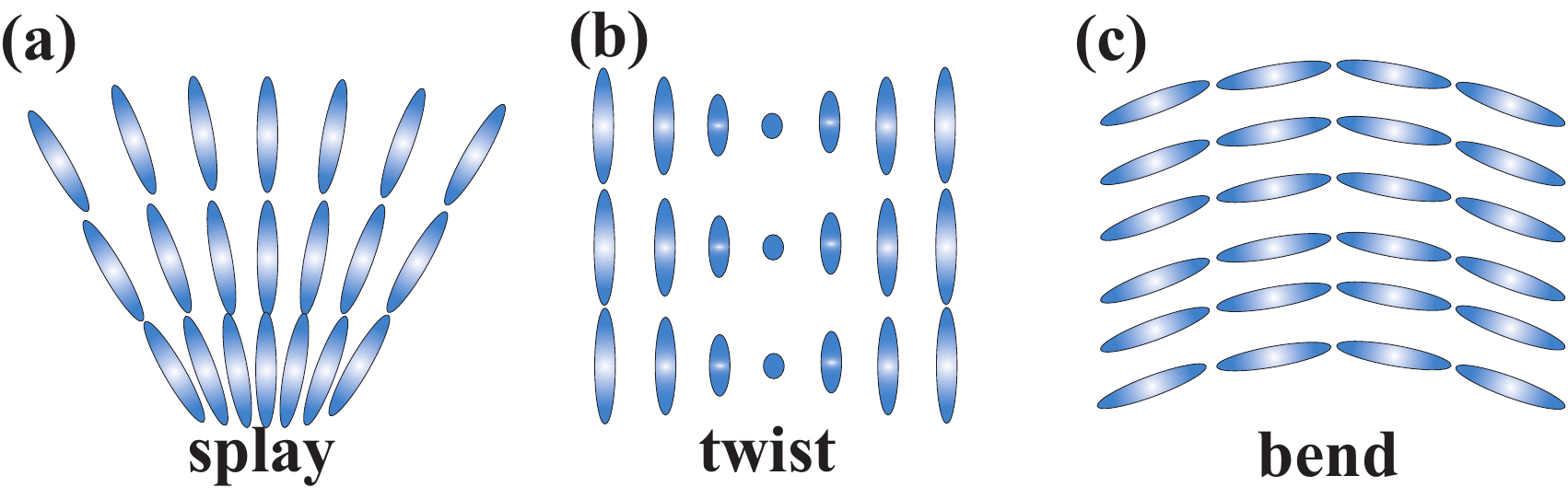}
    \caption{Three kinds of distortion: (a) splay, (b) twist, and (c) bend.}
    \label{splay_twist_bend}
    \end{center}
\end{figure}

The last term is actually a null Lagrangian which can be reduced to boundary terms. A simplest reduction for the Oseen-Frank energy
is the case $k_1=k_2=k_3=k$ and  $k_4=0$, referred as one-constant approximation, which leads to the Dirichlet energy
\begin{align*}
E(\nabla \nn)= \frac{k}{2}|\nabla \nn|^2.
\end{align*}

For given boundary data on a bounded domain, the observed configuration usually corresponds to a minimizer of the Oseen-Frank energy.
Applying the method of calculus of variation, a minimizer should satisfies, at least formally, the following Euler-Lagrange equations:
\begin{align*}
\nn\times\div\Big(\frac{\partial}{\partial\nabla\nn}E_{OF}(\nn,\nabla\nn)-\frac{\partial}{\partial \nn}E_{OF}(\nn,\nabla\nn)\Big)=0,
\end{align*}
or equivalently
\begin{align}\label{EL:Oseen-Frank}
\div\Big(\frac{\partial}{\partial\nabla\nn}E_{OF}(\nn,\nabla\nn)-\frac{\partial}{\partial \nn}E_{OF}(\nn,\nabla\nn)\Big)=\lambda \nn.
\end{align}
Note that  $\nn$ need to satisfy the unit-norm condition: $|\nn|=1$, which gives a nonlinear constraint and induces a Lagrange multiplier  $\lambda\in\mathbb{R}$,
in the above equation. In the one-constant approximation case, the equation reduces to the harmonic map equation:
\begin{align*}
\Delta \nn =-|\nabla\nn|^2\nn.
\end{align*}

Defects in vector theories are described by singularities in $\nabla\nn$. For instance, the configuration
$\nn(\xx)=\frac{\xx}{|\xx|}$ ($\xx\in\mathbb{R}^3$) is a solution to (\ref{EL:Oseen-Frank}), which is called the hedgehog solution.
It is a typical  and important example of point defects. In a 2D region, $\nn(\xx)=\frac{\xx}{|\xx|}$ is formally a solution to (\ref{EL:Oseen-Frank}). However, the energy blows up near the
singular point $\xx=0$. That is, the energy of a 2D point defect in vectorial description is infinite.
Moreover, the following theorem has been proved by \citeasnoun{HardtKL1986CMP}. 
\begin{theorem}
If $\nn\in H^1(\Omega, \mathbb{S}^2)$ is a minimizer of the Oseen–Frank energy $E_{OF}$, then $\nn$ is analytic on $\Omega\setminus Z$, where $Z$ is a relatively closed subset of $\Omega$ which has one dimensional Hausdorff measure zero.
\end{theorem} 

This fundamental result excludes the possibility of
line defects, which have dimension one, under the framework of the Oseen-Frank theory.
To resolve this problem, \citeasnoun{Ericksen1990} proposed a vector model with an extra scalar order  parameter $s
\in [-\frac{1}{2}, 1]$, which represents the degree of orientation. Under the one-constant approximation, the Ericksen free energy takes the form
\begin{equation}
  E(s, \nn) = \int_{\Omega} \psi(s) + k |\nabla s|^2 + s^2 |\nabla \nn|^2 \dd \x,
\end{equation}
where $k > 0$ is a parameter,  $\psi(s)$ is the potential function satisfying:
\begin{itemize}
\item $\lim_{s \rightarrow 1} \psi(s) =  \lim_{s \rightarrow - 1/2} \psi(s) = + \infty,$
\item there exists $s^{*} \in (0, 1)$ such $\psi(0) > \psi(s^{*}) = \mathrm{min}_{s \in [-1/2, 1]} \psi(s),$
\item $\psi'(0) = 0.$
\end{itemize}
In Ericksen's theory, the defects are defined as zero set of $s$, which permits the line defect or a 2D point defect \cite{lin1991nematic,ambrosio1990existence}. We refer to \citeasnoun{Ericksen1990}
for details.

\subsubsection{Dynamical vector models: the Ericksen-Leslie theory}

The dynamic continuum theory for LC flows was  established by \cite{Leslie1968} and \cite{ericksen1961conservation}.
The full system, which is called Ericksen-Leslie system, takes the form
\begin{eqnarray}\label{eq:EL}
\left\{
\begin{split}
&\vv_t+\vv\cdot\nabla\vv=-\nabla{p}+\nabla\cdot\sigma,\\
&\na\cdot\vv=0,\\
&\nn\times\big(\hh-\gamma_1\NN-\gamma_2\DD\cdot\nn\big)=0.
\end{split}\right.
\end{eqnarray}
Here $\vv$ is the fluid velocity, $p$ is the pressure,
and the stress $\sigma$ is given by the phenomenological constitutive relation
\begin{equation*}
\sigma=\sigma^L+\sigma^E,
\end{equation*}
where $\sigma^L$ is the viscous (Leslie) stress
\begin{align}\label{eq:Leslie stress}
\sigma^L=\alpha_1(\nn\nn:\DD)\nn\nn+\alpha_2\nn\NN+\alpha_3\NN\nn+\alpha_4\DD
+\alpha_5\nn\nn\cdot\DD+\alpha_6\DD\cdot\nn\nn
\end{align}
with $\DD=\frac{1}{2}(\kappa^T+\kappa), \kappa=(\na \vv)^T$, and
\begin{equation*}
\NN=\nn_t+\vv\cdot\nabla\nn+\BOm\cdot\nn,\quad\BOm=\frac{1}{2}(\kappa^T-\kappa).
\end{equation*}
The constants $\al_1, \cdots, \al_6$ in (\ref{eq:Leslie stress}) are called the Leslie coefficients.  While, $\sigma^E$ is the elastic (Ericksen) stress
which is given by
\begin{eqnarray}\label{eq:Ericksen}
\sigma^E=-\frac{\partial{E_{OF}}}{\partial(\nabla\nn)}\cdot(\nabla\nn)^T,
\end{eqnarray}
and the molecular field $\hh$ is given by
\begin{equation*}
\hh=-\frac{\delta{E_{OF}}}{\delta{\nn}}=
\nabla\cdot\frac{\partial{E_{OF}}}{\partial(\nabla\nn)}-\frac{\partial{E_{OF}}}{\partial\nn},
\end{equation*}
The Leslie coefficients and $\gamma_1, \gamma_2$ satisfy the following relations
\begin{align}
&\alpha_2+\alpha_3=\alpha_6-\alpha_5,\label{Leslie relation}\\
&\gamma_1=\alpha_3-\alpha_2,\quad \gamma_2=\alpha_6-\alpha_5,\label{Leslie-coeff}
\end{align}
where (\ref{Leslie relation}) is called Parodi's relation derived from the Onsager reciprocal relation \cite{parodi1970stress}. These two relations
ensure that the system has a basic energy dissipation law:
\begin{align}
\frac{\ud}{\ud{t}}\Big(\int_{\BR}\frac12|\vv|^2\ud\xx+E_{OF}\Big)
=&- \int_{\BR}\Big((\alpha_1+\frac{\gamma_2^2}{\gamma_1})(\DD:\nn\nn)^2
+\alpha_4|\DD|^2\qquad\nonumber\\
&\quad+(\alpha_5+\alpha_6-\frac{\gamma_2^2}{\gamma_1})|\DD\cdot\nn|^2
+\frac{1}{\gamma_1}|\nn\times\hh|^2\Big)\ud\xx.\label{EL_energy_law}
\end{align}

Besides the lack of ability to describe line defects, vector models have some other drawbacks.
For example, from the physical viewpoint, $\nn(\xx)$ is not distinguishable to $-\nn(\xx)$. This is referred as
the head-to-tail symmetry of LCs, which can not be inherently revealed by the vectorial description.
Indeed, there are some configurations cannot be described by a vector field.
For example, consider the point defect in 2D with degree $1/2$
(see Fig. \ref{fig:defect}(c)) and a circle near the defect point. The alignment on the circle is a smooth line field.
However, one can not define a continuous vector field $\nn\in \mathbb{S}^1$ on this circle.
This problem is carefully discussed in \cite{Ball2011} for more general domains.

In addition, in vector theories, it is assumed that the orientation of the LC molecules has a preferred
alignment at a material point. In most cases this assumption is reasonable. However, there are some situations that
a preferred alignment can not be defined, for example, near the core of a defect.
So vector models fail to give accurate descriptions for molecular alignments near the core of defects.

\subsection{Tensor Theories}
%

Despite its successes in  predicting the phase transition and rheological parameters for LCs,
the molecular theory is not convenient in practice since it always leads to a high-dimensional problem with expensive costs.
Therefore, it is natural to explore alternative models to simulate LC flow or complex patterns.
A common method to reduce the molecular theory is to consider
the second order moment of the probability distribution function $f(\xx, \mm)$:
\begin{align}\label{def:Q}
\QQ(\xx)=&\int_\BS (\mm\mm-\frac13\II)f(\xx, \mm)\ud \mm \in \mathbb{Q}=\Big\{\QQ \in \mathbb{M}^{3\times 3}, \QQ=\QQ^T, \mathrm{tr}~\QQ=0\Big\}, \nonumber
\end{align}
which is called $\QQ$-tensor, as the concerned order parameters.
Apparently, $\QQ$ has five independent components. Let $\lambda_1, \lambda_2, \lambda_3$ be the three eigenvalues of $\QQ$, then
we can write
\begin{align*}
  \QQ=\lambda_1\nn_1\nn_1+\lambda_2\nn_2\nn_2+\lambda_3\nn_3\nn_3,
\end{align*}
where $\nn_1, \nn_2, \nn_3$ are unit norm vectors with $\nn_i\cdot\nn_j=0\ (1\le i<j\le 3)$. The definition (\ref{def:Q}) gives
a constraint for the eigenvalues of $\QQ$:
\begin{align*}
\QQ\in \mathbb{Q}_{phy}=\Big\{\QQ\in\mathbb{Q}: \lambda_1, \lambda_2,\lambda_3\in\Big(-\frac13,\frac23\Big)\Big\}.
\end{align*}
One can classify $\QQ$ into three classes:
\begin{itemize}
\item If $\QQ$ has three equal eigenvalues, i.e. $\QQ=0$, we say $\QQ$ is {\it isotropic}.
\item If the eigenvalues have two distinct value, then there exist $s\in\mathbb{R},~\nn\in\BS$, such that
$$\QQ=s(\nn\nn-\frac13\II).$$
In this case, we say $\QQ$ is  {\it uniaxial};
\item If the three eigenvalues are distinct, we say $\QQ$ is {\it biaxial}; In this case we can find $s,\lambda\in\mathbb{R},~\nn,\mm\in\BS$, such that
$s\neq\lambda$, $\mm\bot\nn$ and
$$\QQ=s(\nn\nn-\frac13\II)+\lambda(\mm\mm-\frac13\II).$$
\end{itemize}
If the LC material retains at the liquid state, i.e., the alignment are disorder, then
the distribution function $f(\xx,\mm)$ is the uniform distribution on $\BS$, which implies $\QQ=0$, or equivalently  $\QQ$ is isotropic;
If the material retains at the LC state, i.e., $f(\xx,\mm)$ is axially symmetric function on $\BS$, that is $f(\xx,\mm)=f(\nn(\xx)\cdot\mm)$, then
\begin{align*}
  \QQ=s(\nn\nn-\frac13\II)\quad\text{ with }s=\int_\BS\frac{3(\nn\cdot\mm)^2-1}{2}f(\nn\cdot\mm)\ud\mm ,
\end{align*} which means that $\QQ$ is uniaxial.

\subsubsection{Static $\QQ$-tensor models}

For LC materials, the total energy $F^{LG}(\QQ,\nabla\QQ)$ consists of  two parts: {\it the bulk energy} $F_b(\QQ)$ which dictates the preferred state of the material, and {\it the elastic energy} $F_e(\QQ, \nabla \QQ)$
which comes from the distortion of LCs:
\begin{align*}
F^{LG}(\QQ,\nabla\QQ)=F_b(\QQ)+F_e(\QQ,\nabla\QQ).
\end{align*}
The energy should be frame indifference, that is, for any $\PP\in O(3)$,
\begin{align*}
F^{LG}(\tilde \QQ, \tilde{\mathbf{\Psi}})=F^{LG}(\QQ, \tilde{\mathbf{\Psi}}),
\end{align*}
where $\mathbf{\Psi}=\nabla\QQ$ with $\Psi_{ijk}=Q_{ij,k}$ and $\tilde \QQ =\PP\QQ\PP^T$, $\tilde{\mathbf{\Psi}}_{ijk}=P_{ii'}P_{jj'}P_{kk'}\Psi_{ijk}$.


The bulk energy is a function of the tensor $\QQ$ which should predict the isotropic-nematic phase transition for LCs. Therefore, at high temperatures,
the bulk energy should arrive its minimum at the isotropic state, while at low temperature, its minimizers should be uniaxial which represents the nematic phase.
In addition, due to the frame indifference, the bulk energy should depend only on the eigenvalues of $\QQ$.


A simplest form meets these requirements takes the following polynomial form:
\begin{align}
  F_b(\QQ)=\frac{a}{2}|\QQ|^2-\frac{b}{3}\tr \QQ^3+\frac{c}{4}|\QQ|^4,
\end{align}
where $$|\QQ|^2=\sum_{i,j=1}^3Q_{ij}^2,~ \tr \QQ^3=\sum_{i,j,k}Q_{ij}Q_{jk}Q_{ki},~|\QQ|^4=(|\QQ|^2)^2.$$
Here $a, b, c$ are constants depending on materials and temperature in general with $b, c>0$. In particular, the parameter $a$
plays key roles to the isotropic-nematic phase transition, which is usually assumed by $a=A(T-T_*)$,
where $T$ is the temperature and $T_{*}$ is the critical temperature at which the isotropic phase loses stability.

All the (possible) critical points of the bulk energy $f_b$ are given by
$$\QQ=s\left(\nn\nn-\frac13\II\right),\quad s\in\left\{s_0=0, s_{\pm}=\frac{-b\pm\sqrt{b^2-24ac}}{4c}\right\}.$$
Stability or instability of the critical points are shown in Table \ref{table:stability}, where the three critical temperatures are given by:
\begin{align}
T_c=\frac{b^2}{27Ac}+T_*, \quad T_{II}=\frac{b^2}{24Ac}+T_*.
\end{align}
These assertions can be proved straightforwardly. When the temperature $T<T_*$, the isotropic state loses stability and the bulk energy is minimized
by tensors in the minimal manifold
\begin{align}\label{set:minimal-manifold}
  \mathcal{M}=\Big\{s_+\Big(\nn\nn-\frac13\II\Big):~  s_{+}=\frac{-b+\sqrt{b^2-24ac}}{4c}, \nn\in\mathbb{S}^2\Big\}.
\end{align}

\begin{table}[tp] \centering \caption{Stability/instability of critical points for LdG energy}
  \begin{threeparttable}
  \label{table:stability}
    \begin{tabular}{ccccccc}
    \toprule
    \multirow{2}{*}{Temperature}&
    \multicolumn{3}{c}{ Critical points}\cr
    \cmidrule(lr){2-4}
    &$s_0$&$s_-$&$s_+$\cr
    \midrule
     $T<T_{*}$ & unstable & unstable &global minimizer \\
     $T_{*}<T<T_c$ & local minimizer  & unstable & global minimizer\\
     $T_c<T<T_{II}$ & global minimizer & unstable & local minimizer\\
     $T_{II}<T$ & global minimizer & not exist & not exist\\
    \bottomrule
    \end{tabular}
    \end{threeparttable}
\end{table}

One should note that the polynomial energy is phenomenological, that is, the precise physical meaning of coefficients $a, b, c$ are not very clear and
not easy to be a prior determined. Moreover, there is no term force the tensor being in the physical space $\mathbb{Q}_{phy}$, and thus
$\QQ$ may have eigenvalues greater than $2/3$ or less than $-1/3$.
Another point should to be minded is that, the polynomial bulk energy is a finitely truncated Taylor expansion of a
real bulk energy around $\QQ=0$.  Therefore, it is only valid near the isotropic-nematic  transition temperature $T_*$.


To obtain a reliable tensorial model for low temperature materials (far from the transition point), a natural method is by taking the minimal entropy approximation from the molecular energy, which has been applied by \citeasnoun{ball2010nematic}. More precisely, one can define the energy as
\begin{align}\begin{aligned}
F_b(\QQ)& =\min_{f\in \mathcal{A}_\QQ} \int_{\BS}f(\mm)\ln f(\mm)-\alpha |\QQ|^2,\\
\text{where }&f\in \mathcal{A}_\QQ=\Big\{f: f\ge 0, \int_\BS f =1, \int_\BS (\mm\mm-\frac13\II)f =\QQ\Big\}.
\end{aligned}\label{energy:maxentropy}
\end{align}
The above energy can also be obtained by replacing the orientation distribution function by the Bingham distribution of given second momentum $\QQ$. In other words,
for given $\QQ\in\mathbb{Q}_{phy}$, let
$\BB_\QQ$ be the unique trace-free symmetric matrix (see \citeasnoun{LiWangZhang2015DCDSB} for a proof) which satisfies
\begin{align}\label{relation:BQ}
  \frac{\int_\BS (\mm\mm-\frac13\II){\exp(\BB_\QQ:\mm\mm)}\ud\mm}{\int_\BS \exp(\BB_\QQ:\mm\mm)\ud\mm} =\QQ,
\end{align}
and the Bingham distribution $f_{\QQ}$ be given as:
\begin{align}\label{relation:fQ}
f_{\QQ}(\mm)=\frac{\exp(\BB_\QQ:\mm\mm)}{Z_\QQ},\quad Z_\QQ=\int_\BS \exp(\BB_\QQ:\mm\mm).
\end{align}
Then the minimum in (\ref{energy:maxentropy}) is attained by $f_\QQ$, i.e.,
\begin{align*}
F_b(\QQ)= \int_{\BS}f_\QQ(\mm)\ln f_\QQ(\mm)-\alpha |\QQ|^2=\BB_\QQ:\QQ-\ln Z_\QQ -\alpha |\QQ|^2.
\end{align*}
Indeed, note that for fixed second moment, the minimum is achieved by the distribution satisfying
$$\ln f - \BB:\mm\mm =\text{c,}$$  for some constant $c$ and trace-free symmetric matrix $\BB$, which implies
$f$ must take the form (\ref{relation:fQ}) with $\BB$ given by (\ref{relation:BQ}).

On the other hand, for any $\BB\in\mathbb{Q}$, let $\omega (\BB)=\ln \int_\BS \exp(\BB:\mm\mm)\ud\mm$ be a convex function for $\BB$, then
the entropy part $\BB_\QQ:\QQ-\ln Z_\QQ $ equals to
$$\max_{\BB\in\mathbb{Q}}\{\BB:\QQ-\omega(\BB)\},$$
which is the Legendre's transform of the function $\omega(\BB)$. We refer to \citeasnoun{LiWangZhang2015DCDSB} for detailed discussions.

To capture the inhomogeneity of the alignment of LC molecules, one has to take into consideration the elastic energy.  The formulation of the elastic energy should
be frame indifferent and usually it is assumed to be quadratic in $\nabla \QQ$. Some examples are $V_i=V_i(\QQ,\nabla\QQ)$:
\begin{align*}
V_1= Q_{ij,k}Q_{ij,k},~V_2= Q_{ij,j}Q_{ik,k},~ V_3=Q_{ij,k}Q_{ik,j},~V_4=Q_{ij}Q_{ik,l}Q_{jk,l}.
\end{align*}
The difference $V_2-V_3$ can be written as $\partial_k(Q_{ik} Q_{ij,j}-Q_{ij}Q_{ik,j})$ which is a null Lagrangian.
The following energy form is commonly used as the elastic part energy for NLCs:
 \begin{eqnarray}\label{energy:Landau-deGennes}
F^{(e)}[Q]=L_1|\nabla \QQ|^2  + L_2 Q_{ij,j}Q_{ik,k} + L_3Q_{ik,j}Q_{ij,k} + L_4 Q_{lk}Q_{ij,k}Q_{ij,l}.
\end{eqnarray}
The $L_2$--$L_4$ terms correspond to the anisotropic elasticity of LC materials.

Formally, if we brutally let $Q(\xx)$ minimizes the bulk energy in the LdG energy at each point $\xx$,
then $Q(\xx)=s(\nn(\xx)\nn(\xx)-\frac\II 3)$, and the full energy reduces to the Oseen-Frank energy (\ref{energy:Oseen-Frank})
with the coefficients given by
\begin{align}\label{relation:LdG-OF}
\begin{aligned}
{k_1} =2s^2(2L_1+L_2+L_3-2sL_4),\quad k_2=4s^2(L_1-sL_4),\\
 k_3=2s^2(2L_1+L_2+L_3+4sL_4),\quad k_4=2s^2(2L_1+L_3-sL_4).
\end{aligned}
\end{align}

When $L_4=0$, the energy  (\ref{energy:Landau-deGennes}) is coercive \cite{longa1987extension}, i.e.,
\begin{align*}
F^{(e)}[Q]\ge c_0|\nabla\QQ|^2,\quad \text{for some }c_0>0,
\end{align*}
provided that
\begin{align*}
  L_1>0,~-L_1<L_3<2L_1,~L_1+\frac{5}{3}L_2+\frac{1}{6}L_3>0.
\end{align*}
On the other hand, if $L_4\neq 0$, the energy is not bounded from below \cite{ball2010nematic}. However, if the  $L_4$ term is neglected,
due to the form (\ref{relation:LdG-OF}), we can only recover the Oseen-Frank energy with $k_1=k_3$.

Defects in $\QQ$-tensor theory is not characterized by singularities of $\QQ$. Indeed, for minimizers of the
LdG energy with  suitable boundary conditions, it is usually smooth everywhere, since the
corresponding Euler-Lagrange equation is a semilinear elliptic system. To observe defects, one has
to look at the uniaxial limit of the solution. This is reasonable, since the elastic constants are usually small, which
means that the bulk energy will force the $\QQ$-tensor to be in the minimal manifolds $\mathcal{M}$ defined in \eqref{set:minimal-manifold}.
Singularities of  the limit $\QQ$-tensor should be regarded as the set of defects.
There are a number of results studying  the uniaxial limit of minimizers to the LdG energy with certain boundary conditions. We skip to Section \ref{subsec:analysis-defects} for further discussions.

\subsubsection{Dynamic $\QQ$-tensor models}
The dynamical $\QQ$-tensor theories for LCs can be classified into two kinds. The first
one are derived directly from physical considerations such as variational principle. The
Beris-Edwards model \cite{BerisEdwards} and Qian-Sheng model \cite{QianSheng1998generalized} belong to this class.
Given the free energy $F(\QQ,\nabla \QQ)$, the variation is denoted by
$$\mu_{\QQ}=\frac{\delta{F(\QQ,\nabla \QQ)}}{\delta \QQ}.$$
Then the Beris-Edwards model and Qian-Sheng model  can be written in the form:
\begin{align}
\frac{\pa{{\QQ}}}{\pa{t}}+\vv\cdot\nabla{\QQ}&=D(\mu_\QQ)+S(\QQ,\DD)
+\BOm\cdot \QQ- \QQ\cdot\BOm,\label{eq:Q-general-intro}\\
\frac{\pa{\vv}}{\pa{t}}+\vv\cdot\nabla\vv&=-\nabla{p}+\nabla\cdot\big(\sigma^{dis}+\sigma^{s}+\sigma^a+\sigma^{d}\big),
\label{eq:v-general-intro}\\\nonumber
\nabla\cdot\vv=0,
\end{align}
where  $D(\mu_\QQ)$ is the  diffusion term,
$S(\QQ,\DD)$ is the velocity-induced term,  $\sigma^d$ is the distortion stress,
$\sigma^a$ is the anti-symmetric part of orientational-induced stress,
$\sigma^{s}=\gamma S(\QQ, \mu_\QQ)$ is the symmetric stress induced by the molecular alignments,
which conjugates to $S(\QQ, \DD)$ ($\gamma$ is a constant), and $\sigma^{dis}$
is the additional dissipation stress.

In both systems, module some constants, $\sigma^a$ and $\sigma^d$ are the same :
\begin{align*}
\sigma_{ij}^d=\frac{\partial E(\QQ,\nabla \QQ)}{\partial (\QQ_{kl,j})}\QQ_{kl,i},\quad
\sigma^a=\QQ\cdot\mu_\QQ-\mu_\QQ\cdot\QQ.
\end{align*}
In Beris-Edwards system, the other terms are given by:
\begin{align*}
&D_{BE}=-\Gamma \mu_\QQ,\quad \sigma_{BE}^{dis}=\beta\DD,\quad \sigma_{BE}^s=S_{BE}(\QQ,\mu_\QQ),\\
&S_{BE}(\QQ,A)=\xi\Big((\QQ+\frac13Id)\cdot A+A\cdot(\QQ+\frac13Id)-2(\QQ+\frac13Id)(A:\QQ)\Big).
\end{align*}
While in Qian-Sheng's system, they are given by:
\begin{align*}
&D_{QS}=-\Gamma \mu_\QQ,\quad
\sigma_{QS}^s=-\frac12\frac{\mu_2^2}{\mu_1}\mu_{\QQ},\quad
S_{QS}(\QQ,\DD)=-\frac12\frac{\mu_2}{\mu_1}\DD,\\
&\sigma^{dis}_{QS}=\beta_1 \QQ(\QQ:\DD)+\beta_2 \DD+ \beta_3(\QQ\cdot\DD+\DD\cdot \QQ).
\end{align*}
We remark that in Qian-Sheng's original formulation, the inertial effect is also considered.

Another kind of dynamical $\QQ$-tensor models are obtained
by various closure approximations from Doi's kinetic theory. The main idea is to derive the evolution equation for the second
momentum  $\QQ$ from the evolution equation for the orientation distribution function $f$. This is a natural way of model reduction and
the parameters can be calculated from the kinetic equations rather than being phenomenologically determined.
However, by a direct calculation, one can find that the evolution of $\QQ$ depends on the fourth momentum of $f$, which can not be determined by $\QQ$.
In order to ``close'' the equation at the level of the second moment tensor, one needs to represent the fourth momentum by $\QQ$ approximately.
Doi introduced a simplest approximation: $$\langle \mm\mm\mm\mm\rangle_f=\langle\mm\mm\rangle_f\langle\mm\mm\rangle_f.$$
Other various  closure methods have also been presented, such as  the HL1/HL2 closure
\cite{hinch1976constitutive} and the Bingham closure \cite{chaubal1998closure}.
We refer to \cite{feng1998which} for the summary and comparison between these closure methods.
All these models can capture many qualitative features of the LC dynamics
effectively. However, they do not obey the energy dissipation law.

\subsubsection{A systematic way to derive new $\QQ$-tensor models}
In \citeasnoun{han2015microscopic}, the authors proposed a systematic way to derive a $\QQ$-tensor model from the molecular theory.
The main idea can be explained as follows: we start from the nonlocal Onsager molecular energy functional (\ref{energy:nonlocal-onsager})
with suitable given interaction kernel $\mathcal{B}_\ve$, and then approximate the orientation distribution function by a
suitable function of its second moment. Then the energy can be entirely determined by
the second moment and thus is reduced to a $\QQ$-tensor type model.
This procedure can be applied not only for the nematics but also other phases for rod-like molecules and even other shapes \cite{xu2014microscopic,xu2018tensor,xu2018Calculating}.

For NLCs, we choose the Onsager's energy functional (\ref{energy:nonlocal-onsager})
with $\mathcal{B}_\ve$ being the excluded volume potential (\ref{potential:nonlocal-Onsager}).
Note that $\mathcal{B}_\ve$ is translational invariant, so we let $B(\xx-\xx',\mm,\mm')=\mathcal{B}_\ve(\xx,\mm; \xx',\mm')$ which is even in $\xx-\xx'$.

Make the Taylor expansion for the {orientational distribution function} $f(\textbf{x}',\textbf{m}')$
with respect to $\textbf{x}'$ at $\textbf{x}$ ($\mathbf{r}=\xx'-\xx$):
\begin{align}
f(\textbf{x}',\textbf{m}')&=f(\xx+\mathbf{r},\mm')\nonumber\\ \label{f-expan}
&=f(\textbf{x},\textbf{m}') + \nabla f(\textbf{x},\textbf{m}')
\cdot\mathbf{r} + \frac12 \nabla^2 f(\textbf{x},\textbf{m}'):\mathbf{r}^T\mathbf{r} + \cdots.
\end{align}
Then the energy can be expanded as:
\begin{align}\label{eq:FreeEnergy-mole-expan}
A[f]=&\int_{\Omega}\int_{\BS}\bigg\{f(\xx,\mm)(\ln{f(\xx,\mm)}-1)+\frac{1}{2}
\int_\BS M^{(0)}(\mm,\mm')f(\xx,\mm')f(\xx,\mm)\ud\mm'\bigg\}\ud\mm\ud\xx\nonumber\\
&\quad+\frac{1}{2}\int_\Omega\int_\BS\int_\BS f(\xx,\mm)M^{(1)}(\mm,\mm')\cdot\nabla f(\xx,\mm')\ud\mm'\ud\mm\ud\xx\nonumber\\
&\quad-\frac{1}{4}\int_\Omega\int_\BS\int_\BS M^{(2)}(\mm,\mm'):\nabla f(\xx,\mm')\nabla f(\xx,\mm)\ud\mm'\ud\mm\ud\xx+\cdots.
\end{align}
where for given kernel function $B(\mathbf{r};\mm,\mm')$,  the moments $M^{(k)}(\mm,\mm')$ $(k=0,1,2,\cdots)$ are defined by:
\begin{align*}
&M^{(0)}(\mm,\mm')=\int B(\mathbf{r},\mm,\mm')\ud\mathbf{r},\\
&M^{(k)}(\mm,\mm')=\int B(\mathbf{r},\mm,\mm')\underbrace{\mathbf{r}\otimes\cdots
\otimes\mathbf{r}}_{k \text{ times}}\ud\mathbf{r}, \qquad (k\ge1).
\end{align*}
These moments depend on the geometric shape of LC molecules. For nematic molecules, they commonly are treated as
 as ellipsoids or spherocylinders. In \citeasnoun{han2015microscopic}, the first three moments are explicitly calculated by considering the molecules as
 spherocylinders with length $L$ and diameter $D$:
\begin{align*}
M^{(0)}(\mm,\mm')&=2L^3\big(\eta|\mm\times\mm'| + \pi \eta^2 + \frac{2}{3}\pi \eta^3\big),\\
M^{(1)}(\mm,\mm')&=0,\\
M^{(2)}(\mm,\mm')&=B_1\II + B_2(\textbf{mm}+\textbf{m}'\textbf{m}') + B_3(\textbf{m}\textbf{m}'+\textbf{m}'\textbf{m})(\mm\cdot\mm'),
\end{align*}
where  $\eta=D/L$ and $B_i$ are functions of $\mm\cdot\mm'$ (see \cite[Appendix]{han2015microscopic} for details):
$$
\left\{
\begin{array}{ll}
B_1(\mm\cdot\mm')&=L^4D\Big(\frac{2|\mm\times\mm'|\eta^2}{3}
+\frac{\pi\eta^3}{2}+\frac{4\pi\eta^4}{15}\Big),\\
B_2(\mm\cdot\mm')&=L^4D\Big(\frac{|\mm\times\mm'|}{6}
+\frac{\pi\eta(1+\eta)}{3}+\frac{\pi\eta^3}{4}+\frac{2\eta^2}{3|\mm\times\mm'|}\Big),\\
B_3(\mm\cdot\mm')&=L^4D\eta^2\Big(\frac{2\arcsin(\mm\cdot\mm')}{3(\mm\cdot\mm')}
-\frac{2}{3|\mm\times\mm'|}\Big).
\end{array}
\right.
$$

The first line in (\ref{eq:FreeEnergy-mole-expan}) is independent of space variation of the probability distribution
function $f$, which gives the {bulk energy} part. The other terms, which depend on space variation of $f$, provide the elastic part of the free energy.

To derive tensor models from  molecular models, we need to use $\QQ(x)$ to express the total energy.
Since it is unrealistic to recover $f$ by finite number of moments, we need to make closure approximation.
We choose the Bingham closure
here, for the reasons that it keeps physical constraints
on the eigenvalues and  preserves the energy structure for dynamics.
We also take the densities variation into consideration, i. e.,
\begin{align}
  f(\xx,\mm)=c(\xx)f_\QQ(\mm),\quad \text{with }f_\QQ \text{ given by (\ref{relation:fQ}}).
\end{align}
The bulk energy is then approximated by
\begin{align}\label{energy:approx:bulk}
F_{\text{bulk}}=&\int_{\Omega}\bigg\{c(\xx)\ln c(\xx)+c(\xx)\int_{\BS}\Big(f_{\QQ}(\xx,\mm)(\ln{f_{\QQ}(\xx,\mm)}-1)\Big)\ud\mm\nonumber\\
&\quad+\frac{c^2(x)}{2}\int_\BS M^{(0)}(\mm,\mm')f_{\QQ}(\xx,\mm')f_{\QQ}(\xx,\mm)\ud\mm'\ud\mm\bigg\}\ud\xx.
\end{align}
Note that the above energy can be viewed as a functional of $c(\xx)$ and $\QQ(\xx)$. If the singular term $|\mm\times\mm'|$ in $M^{(0)}(\mm,\mm')$
is replaced by its smooth alternative $|\mm\times\mm'|^2$, we may arrive
\begin{align*}
F_{\text{bulk}} = \int_{\Omega} c(\xx)\Big(\ln c(\xx)+\QQ(\xx):\BB_\QQ(\xx)-\ln
Z_\QQ(\xx) -\alpha L^3 c(\xx)|\QQ(\xx)|^2\Big)\mathrm{d}\textbf{x},
\end{align*}
where $\alpha$ is a dimensionless constant.

To derive an  elastic energy convenient to use, we consider only finite terms
in  (\ref{eq:FreeEnergy-mole-expan}).
For the nematic phase, it is natural to neglect the terms
whose order of derivatives are greater than one
since the first order derivatives dominates the elastic energy part. If one would like to consider
the smectic phase, it seems enough to keep only the terms
whose order of derivatives are not greater than two.

 Now it is needed
to express
\begin{align}
\int_\BS\int_\BS f_\QQ(\xx,\mm)M^{(2)}(\mm,\mm')
f_\QQ(\xx,\mm')\ud\mm'\ud\mm'
\end{align}
in terms of tensor $\QQ$.
Thus, we have to separate the variables of $\mm$ and $\mm'$ in
$M^{(l)}(\mm,\mm')$. This can not be done precisely in general, since there are some terms like $|\mm\times\mm'|$.
One has to treat them as functions of $\mm\cdot\mm'$ and
use polynomial expansion, such as Taylor expansion or Legendre polynomial expansion,
up to a finite order to approximate them. We skip the details and just present the reduced elastic energy after approximating:
\begin{align}\label{energy:approx:elastic}
F_{\text{elastic}}
=&\frac12\int_\Omega\bigg\{J_1|\nabla{c}|^2+J_2|\nabla(c\QQ)|^2
+J_3|\nabla(c\QQ_4)|^2+J_4\partial_i(cQ_{ij})\partial_jc\nonumber\\
&+J_5\Big(\partial_i(cQ_{ik})\partial_j(cQ_{jk})+\partial_i(cQ_{jk})\partial_j(cQ_{ik})\Big)\nonumber\\
&+J_6\Big(\partial_i(cQ_{4iklm})\partial_j
(cQ_{4jklm})+\partial_i(cQ_{4jklm})\partial_j(cQ_{4iklm})\Big)\nonumber\\
&+J_7\partial_i(cQ_{4ijkl})\partial_j(cQ_{kl})\bigg\}\ud\xx.
\end{align}
Here $\QQ_4=\QQ_4[f_\QQ]$ is defined in (\ref{def:Q4}). The coefficients $J_i(1\le i\le 7)$
depend only on parameters $L, \eta$ and can be explicitly calculated.
Combining (\ref{energy:approx:bulk}) and (\ref{energy:approx:elastic}),
we obtain an energy in tensorial form derived from Onsager's molecular theory:
\begin{align*}
F^{Mol}(c, \QQ, \nabla c, \nabla\QQ)=F_{\text{bulk}}+F_{\text{elastic}}.
\end{align*}
The density variation can be neglected in most cases, especially when defects are absent. Then the energy can be further simplified into a
form which depends only $\QQ$ and $\nabla \QQ$.

If we further assume the uniaxiality of the $\QQ$ tensor, then we obtain a vector model with elastic coefficients given by molecular parameters.
We refer to \citeasnoun{han2015microscopic} for details. The procedure of reduction from microscopic molecular theories to macroscopic continuum theories
can be illustrated in Fig. \ref{fig:static}. The microscopic interpretation of elastic coefficients in the Oseen-Frank theory
has also been studied in \citeasnoun{gelbart1982molecular} and many other related works.

\begin{figure}
    \begin{center}
    \includegraphics[width=0.6\columnwidth]{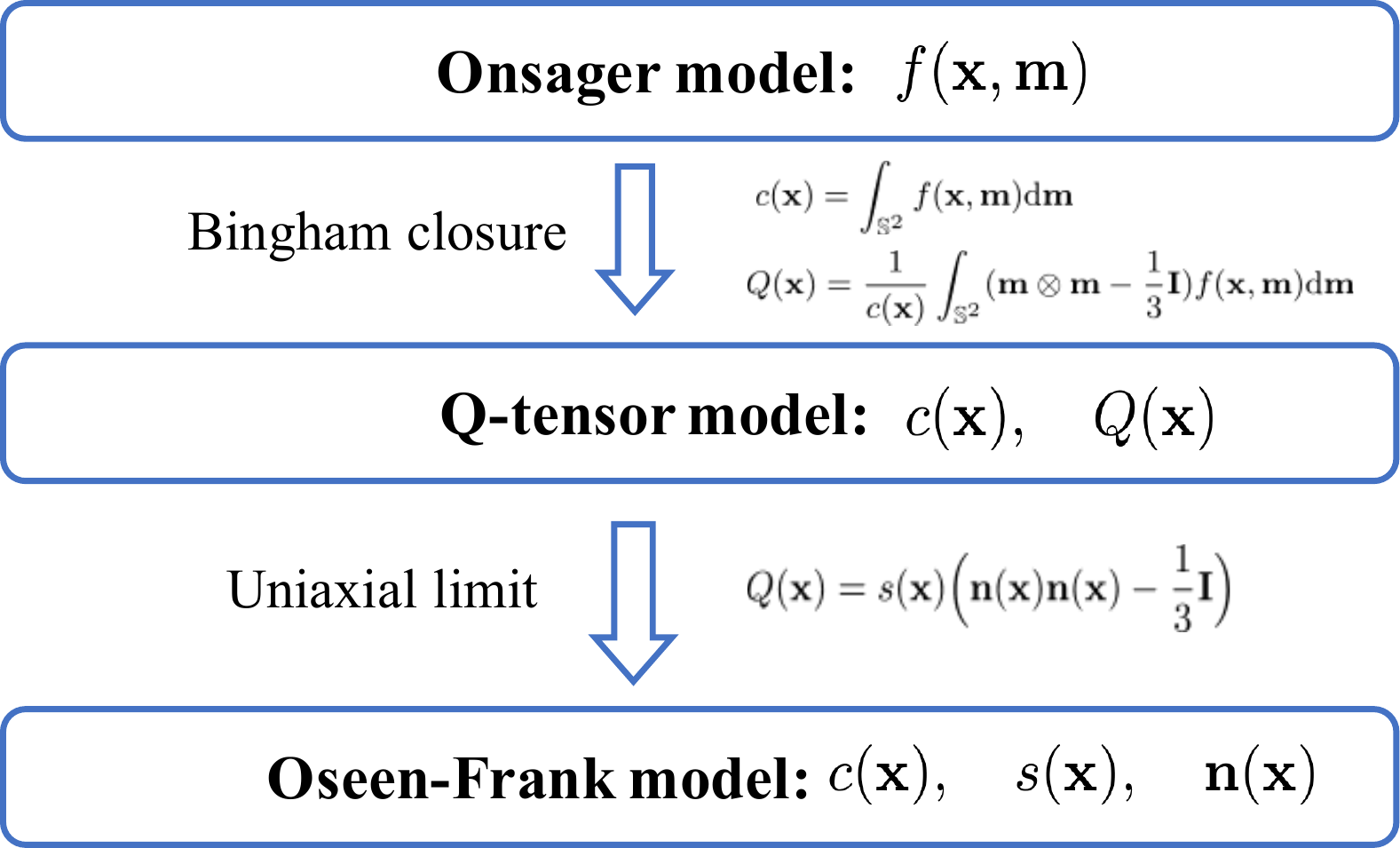}
    \caption{Reduction from microscopic molecular theories to macroscopic continuum theories.}
    \label{fig:static}
    \end{center}
\end{figure}

The similar procedure can be applied to derive the dynamical equation based on the dynamical Doi-Onsager equation.
The main step is to calculate the evolution of $\QQ$ from the evolution equation for $f$.
For simplicity of presentation, we assume the density is constant and take the energy being the simplest form:
\begin{align}\label{energy:nonlocal-Q}
F(\QQ)=\int \QQ(\xx) :\BB_\QQ(\xx)-\ln
Z_\QQ(\xx) -\frac{\alpha}{2} \QQ(\xx):\QQ^\ve(\xx) \ud\xx,
\end{align}
with
$$\QQ^\ve=\int_\BR g_\ve(\xx-\xx')\QQ(\xx')\ud\xx'.$$
This energy can be derived directly from the Onsager energy functional
(\ref{energy:nonlocal-onsager}) with nonlocal Maier-Saupe potential (\ref{potential:nonlocal-MS}).

 As higher order moments will be involved, we also replace
them by the corresponding ones of the approximated Bingham distribution. Namely, we denote
\begin{align*}
M^{(4)}_\QQ=\int_\BS\mm\mm\mm\mm{f}_\QQ\ud\mm,\quad
M^{(6)}_\QQ=\int_\BS\mm\mm\mm\mm\mm\mm{f}_\QQ\ud\mm,
\end{align*}
and
\begin{align*}
\mathcal{M}_{\QQ}(A)&=\frac13A+\QQ\cdot A-A:M_\QQ^{(4)},\\
\mathcal{N}_{\QQ}(A)_{\alpha\beta}&=\partial_i\bigg\{\Big[\gamma_\bot\Big(M^{(4)}_{\QQ\alpha\beta kl}\delta_{ij}
-\frac{\delta_{\alpha\beta}}{3}Q_{kl}\delta_{ij}\Big)
+(\gamma_{\|}-\gamma_\bot)\big(M^{6}_{\alpha\beta klij}
-\frac{\delta_{\alpha\beta}}{3}M^{(4)}_{klij}\big)\Big]\partial_jA_{kl}\bigg\}.
\end{align*}
We also let
\begin{align*}
\mu_\QQ:=\frac{\delta F(\QQ )}{\delta \QQ}=B_\QQ-\alpha \QQ^\ve.
\end{align*}
Noting that
\begin{align*}
3\QQ= 2\alpha \cM_\QQ(B_\QQ),
\end{align*}
we can obtain a closed $\QQ$-tensor system from the Doi-Onsager equation (\ref{eq:LCP-non}):
\begin{eqnarray}\nonumber
\frac{\pa{{\QQ}}}{\pa{t}}+\vv\cdot\nabla{\QQ}&=&\frac{\ve}{De}\mathcal{N}_\QQ(\mu_\QQ)-\frac{2}{De}
\big(\cM_\QQ+\cM_\QQ^T\big)(\mu_\QQ)+\big(\cM_\QQ+\cM_\QQ^T\big)(\kappa^T),\\\label{eq:molecular-Q}
\frac{\pa{\vv}}{\pa{t}}+\vv\cdot\nabla\vv&=&-\nabla{p}
+\frac{\gamma}{Re}\Delta\vv+\frac{1-\gamma}{2Re}\nabla\cdot(\DD:M_\QQ^{(4)})\\
&&\quad+\frac{1-\gamma}{DeRe}\Big(2\nabla\cdot \cM_\QQ(\mu_\QQ)+{\QQ}:\nabla\mu_\QQ\Big).\nonumber
\end{eqnarray}
The system (\ref{eq:molecular-Q}) obeys the following energy dissipation law
\begin{align*}
\begin{aligned}
&\frac{\ud}{\ud{t}}\Big\{\int_\BR\frac{1}{2}|\vv|^2\ud{x}+\frac{1-\gamma}{ReDe}F(Q)\Big\}\\
&=-\frac{1}{Re}\int_\BR\bigg\{\gamma|\DD|^2+\frac{1-\gamma}{2}M_\QQ^{(4)}:(\DD\otimes\DD)\\
&\qquad-\frac{\ve(1-\gamma)}{De^2}\mu_\QQ
:\mathcal{N}\mu_\QQ+\frac{4(1-\gamma)}{De^2}\mu_\QQ:\cM_\QQ\mu_\QQ\bigg\}\ud{x}.
\end{aligned}
\end{align*}
We can also replace the energy (\ref{energy:nonlocal-Q})
in (\ref{eq:molecular-Q})  in by a general local energy $F(\QQ, \nabla \QQ)$ to obtain
a dynamical system in local form. The energy law is kept regardless of the particular choice of $F(\QQ, \nabla \QQ)$.

Since the system (\ref{eq:molecular-Q}) is derived from the Doi-Onsager equation by the Bingham closure,
it keeps many important physical properties:
First, the system preserves the energy structure, which is violated by other closure models;
secondly, the parameters have definite physical meaning rather than be phenomenologically determined;
thirdly, the eigenvalues of $\QQ$ satisfy the physical constrain: $\lambda_i\in\{-1/3,~2/3\}(i=1,2,3)$ if they are satisfied initially;
Moreover, both translational and rotational diffusion can be kept in the formulation, and the translational diffusion can be anisotropic.

Finally, we make some comparisons between the system (\ref{eq:molecular-Q}) with the models of Beris-Edward and Qian-Sheng.
The system (\ref{eq:molecular-Q}) can be similarly written as
(\ref{eq:Q-general-intro})-(\ref{eq:v-general-intro}):
\begin{align}\label{eq:Q-general}
\begin{aligned}
\frac{\pa{{\QQ}}}{\pa{t}}+\vv\cdot\nabla{\QQ}&=D(\mu_\QQ)
+S(\QQ,\DD)
+\BOm\cdot \QQ- \QQ\cdot\BOm,\\
\frac{\pa{\vv}}{\pa{t}}+\vv\cdot\nabla\vv&=-\nabla{p}+\nabla\cdot\big(\sigma^{dis}+\sigma^{s}+\sigma^a+\sigma^{d}\big),
\\
\nabla\cdot\vv&=0.
\end{aligned}
\end{align}
Here
$$D(\mu_\QQ)=\frac{\ve}{De}\mathcal{N}(\mu_\QQ)+\frac{2}{De}(\cM_\QQ+\cM_\QQ^T)(\mu_\QQ),$$
where the first and the second term account for the translational and rotational
diffusion respectively, which is not different with the one in Beris-Edwards's and Qian-Sheng's models.
We remark that $\mu_\QQ:\nabla{\QQ}$ is equivalent to
 $\partial_j\big(\frac{\partial F}{\partial (Q_{kl,j})}Q_{kl,i}\big)$ module a pressure term, and
$$(\cM_\QQ-\cM_\QQ^T)(\mu_\QQ)=\QQ\cdot\mu_\QQ-\mu_\QQ\cdot \QQ,$$
so the stress terms $\sigma^a$ and $\sigma^d$  in our model, module some constants, are the same as in those two models:
\begin{align}
\sigma_{ij}^d=\frac{\partial F(\QQ,\nabla\QQ)}{\partial (Q_{kl,j})}Q_{kl,i},\quad
\sigma^a=\QQ\cdot\mu_\QQ-\mu_\QQ\cdot \QQ.
\end{align}
The dissipation stress is given by
$$\sigma^{dis}=\frac{2\gamma}{Re}\DD+\frac{1-\gamma}{2Re}\DD:M_\QQ^{(4)}.$$
The two conjugated terms $S(\QQ,\DD)$ and $\sigma^s=-\frac{1-\gamma}{Re De}S(\QQ,\mu_\QQ)$ are given by
$$S(\QQ,A)=(\cM_\QQ+\cM_\QQ^T)(A).$$
Note that if we apply Doi's closure
$\langle\mm\mm\mm\mm\rangle_f\sim \langle\mm\mm\rangle_f\langle\mm\mm\rangle_f$ in $M_\QQ^{(4)}$, then
$$(\cM_\QQ+\cM_\QQ^T)(A)\sim (\QQ+\frac13Id)\cdot A+A\cdot(\QQ+\frac13Id)-2(\QQ+\frac13Id)(A:\QQ),$$
which is indeed the corresponding term $S_{BE}(\QQ, A)$ in the Beris-Edwards system with $\xi=1$.

\section{Mathematical analysis for different liquid crystal models}
In this section, we review some analysis results on various LC models.
Since there are numerous progress in this active area, we do not intend to cover all topics but concentrate on the analysis of defects, well-posedness theory of dynamical theories and the connections between them.
Some of them have already been introduced in Section \ref{sec:2}, and will not be  presented again.

\subsection{Analysis on defects}\label{subsec:analysis-defects}
The LdG model and Oseen-Frank model,
have been widely used to study the properties of equilibrium configurations under various conditions.

In vector theories, the simplest configuration contains defect is
\begin{align}\label{config:hedgehog}
\nn(\xx)=\frac{\xx}{|\xx|},
\end{align}
known as the hedgehog. It is not hard to verify that the hedgehog is always a weak solution
of the Euler-Lagrange equation for any choice of  $k_i$-s. 
For the stability, it has been proved that the hedgehog is stable if $8(k_2-k_1)+k_3\ge 0$ \cite{cohen1990weak} and not stable if $8(k_2-k_1)+k_3 < 0$ \cite{helein1987minima}.
One can construct other typical configurations of point defects. For example, in the one-constant case, for any $\mathbf{R}\in O(3)$, $$\nn(\xx)=\mathbf{R}\frac{\xx}{|\xx|}$$ is always a
solution of the Euler-Lagrange equation. Moreover, it is always stable \cite{brezis1986harmonic}.
For the general case, the choice of $\mathbf{R}$ is limited. If
$k_1=k_3$, $\mathbf{R}$ has to be $\pm\mathbf{I}$ or any $180^\circ$ rotation.
When $k_i$ differs with each other, $\mathbf{R}$ must be $\pm\mathbf{I}$.
The corresponding stability is analyzed by \citeasnoun{kinderlehrer1993elementary}.

Moreover, by considering suitable surface energy, the free boundary problems have also been
analytically studied to explore optimal shapes of LC droplets \cite{LinPoon,ShenLiuCal}.
For more results on the analysis of the Oseen-Frank  and related models, we refer to the survey paper \cite{LinLiuSurv}.

The vector theory can provide the macroscopic information of the alignment for defects configuration.
While to study the fine structure of defect cores, one has to use $\QQ$-tensor models.
Given the boundary condition
$\QQ(\xx)=s_+\Big(\nn(\xx)\nn(\xx)-\frac13\II\Big)$ with $\nn(\xx)$ given by (\ref{config:hedgehog}), three kinds of equilibrium solutions
are found numerically \cite{mkaddem2000fine}, which are called radial hedgehog, ring disclination and split core respectively.
These solutions are illustrated in Fig. \ref{ball} (b-d).
Recently, in the low-temperature limit, \cite{yu2020disclinations} proved the existence of axially symmetric solutions describing the ring disclination and split core.

The radial hedgehog solution in the ball $B_R(0)$ can be represented in the form
\begin{align}\label{config:Q-hedgehog}
  \QQ(\xx)=h(|\xx|)\Big(\nn(\xx)\nn(\xx)-\frac13\II\Big),\quad \nn(\xx)=\frac{\xx}{|\xx|},
\end{align}
with $h$ satisfies the ODE
\begin{align}\label{config:heghog-h}
& h''(r)+\frac{2}{r}h'(r)-\frac{6}{r^2}h(r)
=ah(r)-\frac{b}{3}h^2+\frac{2c}{3}h^3,\\
&  h(0)=0,\quad h(R)=s_+.
\end{align}
\citeasnoun{gartland1999instability} proved the instability of  the radial hedgehog solution \eqref{config:Q-hedgehog} when the temperature $a$ is very low and the radii $R$ of the ball is large. {\citeasnoun{majumdar2012radial} proved that if $a$ is closed to zero or $R$ is small, the radial hedgehog solution is locally stable. For the whole space case, i. e., $R=\infty$,
it is shown in \citeasnoun{IgNguSlasZarn} that the radial hedgehog solution is locally stable for $a$ closed to zero and unstable for large $|a|$.
The monotonicity and uniqueness of the solution $h$ to (\ref{config:heghog-h}) is studied in \cite{Lamy}.}

As defects in $\QQ$-tensor theory are not the singularities of order tensor $\QQ$ but the
regions with rapid changes of $\QQ$, one may study the uniaxial limit of LdG model to analyze the properties of defect sets.
Most of these results concentrate in the case $L_2=L_3=L_4=0, L_1=L$.

Consider the strong anchoring boundary condition
\begin{align}\label{bc:ldg}
\QQ|_{\partial\Omega}=s_+\Big(\nn_b(\xx)\nn_b(\xx)-\frac13\II\Big),\quad \nn_b\in C^\infty(\partial\Omega, \mathbb{S}^2).
\end{align}
{Let $\QQ^{L}$ be the global minimizers of the Landau-de Gennes energy
\begin{align*}
  F_{L}(\QQ,\nabla\QQ)=\int_\Omega \frac{L}{2}|\nabla\QQ|^2 +\frac{a}{2}|\QQ|^2-\frac{b}{3}\tr \QQ^3+\frac{c}{4}|\QQ|^4 \ud\xx,
\end{align*}
with boundary condition  (\ref{bc:ldg}). 
\citeasnoun{majumdar2010landau} proved the following results.}
{\begin{theorem}
For a sequence $L_k\to 0$,  the minimizers $\QQ^{L_k}\to \QQ_*$ in the Sobolev space $H^1(\Omega)$, where $\QQ_*\in H^1(\Omega, \mathcal{M})$ is the minimizer of
$$\int_\Omega|\nabla\QQ|^2\ud\xx,\quad \text{with } \QQ\in\mathcal{M}\text{ and satisfying (\ref{bc:ldg})}.$$ 
In addition, the convergence is uniform away from the (possible) singularities of $\QQ_*$.
\end{theorem}
Moreover, away from the singular points of $\QQ_*$, the convergence is further refined by \citeasnoun{nguyen2013refined}.
Note that the fact the limit map $\QQ_*\in H^1(\Omega)$ implies the line defects are excluded in this case.
}
%

\citeasnoun{bauman2012analysis} investigated the uniaxial limit of the Landau-de
Gennes energy  with $L_2, L_3\neq 0$  on a 2D bounded domain.
By assuming that $\mathbf{e}_3$  is always an eigenvector of $\QQ$, they proved that, if the boundary data has nonzero degree,
then there will be a finite number of defects of degree $\pm1/2$. This problem models  the behavior of a thin LC material
with its top and bottom surfaces constrained to having a principal axis $\mathbf{e}_3$, and the limiting  defects correspond to vertical
disclination lines at those locations.
The line defects in the full three-dimensional domain is investigated in \citeasnoun{canevari2017line}, where it is proved that
the  minimizers $\QQ^{(L)}$ converge to a limit map with straight line segments singularities, by assuming the logarithemic bound of the energy:
\begin{align*}
  F\le CL|\ln L|.
\end{align*}
Some related results are also given in \cite{GolovatyMont,CaneESAIM}.

\subsection{The relation between the nonlocal Onsager energy and the Oseen-Frank energy}

%



The orientation distribution function $f$ contains detail configurational information of molecules. However, it is difficult to apply especially
for studying the macroscopic behavior or configurations, as it often leads to very high computational costs.
Thus the $\QQ$-tensor theory or Ericksen-Leslie theory are used more often in analysis and computational simulations.
Then it raises a natural question: are these models consistent to each other? This issue is fundamental but highly non-trivial in the analysis aspect.


We briefly show how to formally derive it.
We expand  the mean-field potential as:
\begin{equation} \CU_\ve
f=U_0[f]+\ve{U_1[f]}+\ve^2U_2[f]+\cdots,
\end{equation}
where $ U_0[f]=\CU f$ and for $k\ge 1$
\begin{eqnarray*}
U_k[f](\xx,\mm,t)=\frac{1}{(2k)!}\int_\BR\int_\BS\alpha|\mm\times\mm'|^2g(\yy)
(\yy\cdot\nabla)^{2k}f(\xx,\mm',t)\ud\mm'\ud\yy.
\end{eqnarray*}

Direct computation shows that
\begin{eqnarray*}
\Ue[f]-U[f]&=&\int_\BR\int_\BS\alpha|\mm\times\mm'|^2g_\ve(\xx-\xx')
(f(\xx',\mm',t)-f(\xx,\mm',t))\ud\mm'\ud\xx'\\
&=&\int_\BR\int_\BS\alpha|\mm\times\mm'|^2g(\yy)
(f(\xx+\sqe\yy,\mm',t)-f(\xx,\mm',t))\ud\mm'\ud\yy\\
&=&\int_\BR\int_\BS\alpha|\mm\times\mm'|^2g(\yy)
\Big(\sum_{k\ge1}\frac{\ve^{\frac{k}{2}}}{k!}(\yy\cdot\nabla)^kf(\xx,\mm',t)\Big)\ud\mm'\ud\yy\\
&=&\int_\BR\int_\BS\alpha|\mm\times\mm'|^2g(\yy)
\Big(\sum_{k\ge1}\frac{\ve^{{k}}}{(2k)!}(\yy\cdot\nabla)^{2k}f(\xx,\mm',t)\Big)\ud\mm'\ud\yy.
\end{eqnarray*}
We let ($U_0[f]=U[f]$):
\begin{eqnarray*}
U_k[f](\xx,\mm,t)&=&\frac{1}{(2k)!}\int_\BR\int_\BS\alpha|\mm\times\mm'|^2g(\yy)
(\yy\cdot\nabla)^{2k}f(\xx,\mm',t)\ud\mm'\ud\yy.
\end{eqnarray*}
Denote $\MM[f]=\int_\BS{{m_im_j}f(\xx,\mm)\ud\mm}$,
$G_{kl}=\int{g(\yy)y_ky_l}\ud{\yy}$. Then
\begin{eqnarray}\label{eq:U1f0}
U_1[f](\xx,\mm)&=&-\frac{\alpha}{2}G_{kl}\mm\mm:\partial_k\partial_l\MM[f],
\end{eqnarray}
Then the energy can be expanded as
\begin{align}\label{energy:onsager-expand}
A_\ve[f]=&\int_{\Omega}\int_{\BS}f(\xx,\mm)(\ln{f(\xx,\mm)}-1)+\frac{1}{2}f(\xx,\mm)U_0[f](\xx,\mm)\ud\mm\ud\xx\nonumber \\
&+\frac{\ve\alpha}{4} \int_\Omega G_{kl}\partial_k\MM[f]:\partial_l\MM[f] \ud\xx+O(\ve^2).
\end{align}
The leading order term is a local energy which can be written as
\begin{align}
\int_\Omega A_0[f(\xx,\cdot)]\ud\xx,
\end{align}
with
\begin{align}
A_0[f]=\int_{\BS}f(\mm)(\ln{f(\mm)}-1)+\frac{1}{2}f(\mm)U_0[f]\ud\mm.
\end{align}
Theorem \ref{thm:critical point} informs us that the global minimizers of the above energy are given by
\begin{align}\label{express:f0}
f_0(\mm)=h_{\eta, \nn}(\mm):=\frac1{Z}{e^{\eta (\mm\cdot\nn)^2}}, \quad Z =\int_\BS e^{\eta (\mm\cdot\nn)^2}\ud\mm,
\end{align}
where $\nn\in\BS$ is an arbitrary unit vector, and $\eta$ satisfies the relation:
\begin{align*}
  \eta=0,\quad \text{or} \quad \frac{\int_0^1e ^{\eta z^2}dz}{\int_0^1z^2(1-z^2)e ^{\eta z^2}dz}=\alpha.
\end{align*}
Substituting (\ref{express:f0}) into (\ref{energy:onsager-expand}), we can derive that for minimizers, the energy can be written as
\begin{align}\label{energy:Oseen-Frank-reduced}
A_\ve =\text{const} + \ve\frac{1}{2} G\alpha S_2^2 |\nabla\nn(\xx)|^2 +O(\ve^2)
\end{align}
with $S_2=\frac{\int_0^1(3z^2-1) e ^{\eta z^2}dz}{2\int_0^1e ^{\eta z^2}dz}$, and
$G=\frac{1}{3}\int{g(\yy)|\yy|^2}\ud\yy$, in which the leading nontrivial term is the one-constant
form of the Oseen-Frank energy. One may also derive the general Oseen-Frank energy for more
complicated (and realistic) interaction kernerl $\mathcal{B}_\ve(\xx; \mm, \mm')$. We refer to \citeasnoun{EZhang2006} for details.
In \citeasnoun{LiuWangARMA}, it is proved that the minimizers or critical points of the Onsager functional
converges to minimizers or critical points of the one-constant Oseen-Frank energy. In \citeasnoun{taylor2018oseen},
the $\Gamma$-convergence from the Onsager functional to a two-constant Oseen-Frank energy has been proved.


\subsection{Dynamics: analysis on various models}
The analysis on the models for LC flow is a hot topic in the analysis and PDE community during the past decades, not only because
these dynamical models provide typical and concrete examples of the hydrodynamical model for the general complex fluids, but also due to the
beautiful and complicate structures in these models and the corresponding deep challenges.
As the results on the wellposedness and long time behavior of dynamical LC models are numerous,
we only mention some of them here, which are far from complete.

The long-time behavior of the Smoluchowski equations without the hydrodynamics was studied by \citeasnoun{vukadinovic2009inertial}. For the full Doi-Onsager  hydrodynamical model, the local well-posedness of strong solutions is
established by \citeasnoun{ZhangZhang-SIAM}. However, the global existence of weak solutions remains open.
For recent work on the existence and properties of solutions to the dynamical $\QQ$-tensor models, we refer to some recent papers \cite{paicu2012energy,abels2014well,HuangDing,Wilk,liu2018initial} and the references therein.

The analysis of Ericksen-Leslie model is initiated by \citeasnoun{lin1991nematic}, in which a simplified Ericksen-Leslie
system (without Leslie's stress) is presented. Moreover, Lin-Liu proposed a regularized model based on the Ginzburg-Landau
approximation and proved the long-time asymptotic, existence and partial regularity \cite{lin1995nonparabolic,lin1996partial}. There are
a lot of works devoting to the existence of global weak solution to the simplified or more general Ericksen-Leslie system
in both $\mathbb{R}^2$  \cite{lin2010liquid,hong2011global,hong2012global,huang2014regularity}
and $\mathbb{R}^3$ \cite{lin2016global}. The local wellposedness of strong
solutions for the general Ericksen-Leslie system is proved in \citeasnoun{WangZZ2013ARMA} for whole space case and in \citeasnoun{HieberNPS}
on bounded domain with the Neumann boundary condition. Recently, the local wellposedness for the  Ericksen-Leslie system
which includes the inertial term was considered \cite{JiangLuo,CaiWang}.

Another important issue in the analysis part is the generation or movement of singularities of solutions to the Ericksen-Leslie equation, which characterizes the dynamical behavior of defects in LC flow.
For the simplified equation, \citeasnoun{huang2016finite} constructed  solutions in a 3-D bounded domain with Dirichlet boundary data  where the direction field blows up at finite time while the velocity field remains smooth. \citeasnoun{lai2019finite} proves that, for any given set of points in $\mathbb{R}^2$, one can construct solutions with smooth initial data which
blow up exactly at these points in a small time.

The above list of analysis results on  the Ericksen-Leslie system is far from complete.
For more analysis results on the Ericksen-Leslie system or its variant versions, we refer to the survey papers \cite{lin2014recent,hieber2016modeling} and the references therein.


\subsection{Dynamics: from Doi-Onsager to Ericksen-Leslie}\label{sec:Doi-EL}

The formal derivation of the Ericksen-Leslie equation from Doi's kinetic theory was first studied by  \citeasnoun{kuzuu1983constitutive}.
However, the Ericksen stress is missed since only the homogeneous case is considered.
This derivation was extended to the inhomogeneous case by  \citeasnoun{EZhang2006}, in which they found
that the Ericksen stress can be recovered from the body force which comes from the inhomogeneity of the chemical potential, see (\ref{eq:LCP-non}).

The derivations in \citeasnoun{kuzuu1983constitutive} and \citeasnoun{EZhang2006} are based on the Hilbert expansion
(also called the Chapman-Enskog expansion) of solutions  with respect to the small parameter
$\ve$:
\begin{align}\label{expan:fv}
\begin{aligned}
  f(t,\xx,\mm)&=f_0(t, \xx, \mm )+\ve f_1(t, \xx, \mm )+\ve^2 f_2(t, \xx, \mm )+\cdots,\\
  \vv(t,\xx)&=\vv_0(t, \xx )+\ve \vv_1(t, \xx )+\ve^2 \vv_2(t, \xx )+\cdots.
\end{aligned}
\end{align}
Substituting the above expansion to the system (\ref{eq:LCP-non}) and collecting the terms with the same order of $\ve$,
we can obtain a series of equations for $(f_0, \vv_0; f_1, \vv_1;\cdots)$.

The $O(\ve^{-1})$ equation gives that $f_0$ satisfies:
\begin{align*}
 \CR\cdot(\CR{f_0}+f_0\CR\CU f_0)=0,
\end{align*}
which means $f_0$ is a critical point of $A_0[f]$. So by Theorem \ref{thm:critical point}, we can let
\begin{align*}
  f_0(t,\mm,\xx)=h_{\eta, \nn(t,\xx)}(\mm), \quad \text{for some }\nn(t,\xx)\in \mathbb{S}^2.
\end{align*}
For the terms of order $O(\ve^0)$, it holds that
\begin{align}
\frac{\pa{f_0}}{\pa{t}}+\vv_0\cdot\nabla{f_0}=&~\CG_{f_0}f_1
+\CR\cdot(f_0\CR{U_1}f_0)-\CR\cdot(\mm\times(\nabla\vv_0)^T\cdot\mm{f_0}),\label{eq:Hilbert-L-10}\\
\frac{\pa{\vv_0}}{\pa{t}}+\vv_0\cdot\nabla\vv_0=&-\nabla{p_0}+\frac{\gamma}{Re}\Delta\vv_0
+\frac{1-\gamma}{2Re}\nabla\cdot(\DD_0:\langle\mm\mm\mm\mm\rangle_{f_0})\nonumber\\
&-\frac{1-\gamma}{Re}\Big\{\nabla\cdot\big\langle\mm\mm\times(\CR{f_1}
+\sum_{\scriptscriptstyle{i+j+k=1}}f_i{U}_jf_k)\big\rangle_1+\langle\nabla{U}_1f_0\rangle_{f_0}\Big\}.\label{eq:Hilbert-L-11}
\end{align}
Although the above system involves the next order term $f_1$ which is not known so far, it is a closed evolution system
for the direction field $\nn(t,\xx)$ and the velocity  $\vv_0(t,\xx)$. 
%
Indeed, it is equivalent to the Ericksen-Leslie system
for $(\nn(t,\xx), \vv_0(t,\xx))$ with coefficients determined by the parameters in the Doi-Onsager equation, 
which will be shown in Theorem \ref{thm:Hilbert-11} and \ref{thm:Hilbert-12} in the next subsections.

\subsubsection{Analysis of the linearized operator}
Analysis of the linearized operator
\begin{eqnarray}
\CG_hf\eqdefa \CR\cdot\big(\CR{f}+h\CR\CU{f}+f\CR\CU{h}\big).
\end{eqnarray}
around a critical point $h$ plays important roles in the reduction of the system (\ref{eq:Hilbert-L-10})-(\ref{eq:Hilbert-L-11}).
In \citeasnoun{kuzuu1983constitutive}, some properties of the null space and spectrums are assumed or presented. With the help of
 complete classifications of critical points of the Maier-Saupe energy, these properties can be rigorously proved \cite{WangZZCPAM}.

We introduce two operators $\CA_h$ and $\CH_h$ which are defined by
\begin{align*}
\CA_h{f}\eqdefa -\CR\cdot(h\CR f)+\CU{f},\quad \CH_h{f}\eqdefa\frac{f}{h}+\CU{f}.
\end{align*}
There holds the following
important relation:
\begin{align}\label{eq:relation}
 \CG_h f=-\CA_h\CH_h f.
\end{align}
Then  the null spaces of $\CG_h$, $\CH_h$ and $\CG_h^*$ (the conjugate of $\CG_h$) can be characterized by the following theorem.
\begin{Proposition}\label{prop:G-kernel}
Let $h_i=h_{\eta_i,\nn}, i=1,2$. For $\al>\alpha^*$, it holds that
\begin{itemize}

\item[1.] $\CG_{h_1}$ and $\CG_{h_2}$ with $\eta_2=0$ have no positive eigenvalues, whereas $\CG_{h_2}$ has at least one positive eigenvalue for $\eta_2\neq0$
;

\item[2.] If $\phi\in \mathrm{Ker}\CG_{h_1}$, then $\CH_{h_1}\phi=0$;

\item[3.] $\mathrm{Ker}\CG_{h_1}=\big\{\mathbf{{\Theta}}\cdot
\CR{h_1}; \mathbf{{\Theta}}\in \mathbb{R}^3\big\}$ is a 2D
space;

\item[4.]$\mathrm{Ker}\CG_{h_1}^*=\big\{\CA^{-1}_{h_1}(\mathbf{{\Theta}}\cdot
\CR{h_1}); \mathbf{{\Theta}}\in \mathbb{R}^3\big\}$.

\end{itemize}

\end{Proposition}


\subsubsection{Derivation of the angular momentum equation}
From the equation (\ref{eq:Hilbert-L-10}), we have that
\begin{align}\label{eq:evolu-f0}
\Big\langle\frac{\pa{f_0}}{\pa{t}}+\vv_0\cdot\nabla{f_0}+\CR\cdot(\mm\times(\nabla\vv_0)^T\cdot\mm{f_0})
-\CR\cdot(f_0\CR{U}_1f_0),
\psi\Big\rangle_{L^2(\BS)}=0,
\end{align}
for any $\psi\in \mathrm{Ker}\CG^*_{f_0}$. This equation provides the evolution equation for the vector field
$\nn(t,\xx)$. Actually, the following lemma tells us that $\nn(t,\xx)$ satisfies the evolution equation in the Ericksen-Leslie equation
with some constants $\lambda_1, \lambda_2$.

As $\mathrm{Ker}~\CG_\nn^*=\CA_\nn\mathrm{Ker}~\CG_\nn$,
$\psi_0\in\mathrm{Ker}~\CG_\nn^*$ if and only if there is a vector $\mathbf{\Theta}$ such that
\begin{equation}
-\CR\cdot(h_\nn\CR{\psi_0})=\Theta\cdot\CR{h_\nn}.
\end{equation}

\begin{theorem}\label{thm:Hilbert-11}
The equation (\ref{eq:evolu-f0}) holds if and only if $\nn(\xx, t)$ is a solution of
\begin{align}
\nn\times\big(\gamma_1(\partial_t\nn+\vv\cdot\nabla\nn+\BOm\cdot\nn)+\gamma_2\DD\cdot\nn-k_F\Delta \nn\big)=0,
\end{align}
 with $\vv=\vv_0$ and
some $\gamma_1, \gamma_2$ depending on $\alpha$ and $k_F$ depending on $\alpha$ and the interaction kernel function $g$.
\end{theorem}
\begin{proof}
Let $\WW$ be a matrix with components given by $W_{ij}=\int_\BS \CA^{-1}(\CR_i{f_0}) \CR_j{f_0}$. Then we have
\begin{itemize}
\item $\WW$ is symmetric;
\item $\WW\cdot\nn=0$;
\item For any unit vector $\pp$ satisfying $\pp\cdot\nn=0$,  $\pp^T\cdot \WW\cdot\pp$ is independent of $\pp$.
 \end{itemize}
Therefore, $\WW$ is a constant, denoted by $\gamma_1$, multiplying with $(\II-\nn\nn)$, and we have
\begin{align}
  \int_\BS \CA^{-1}(\vv\cdot\CR{f_0})\ww\cdot \CR_j{f_0}=\gamma_1\vv\cdot(\II-\nn\nn)\cdot\ww.
\end{align}

Now, assume that $\psi=\CA^{-1}(\Theta\cdot\CR{f_0})$ with $\Theta\in \mathbb{R}^3$.
First we have
$\frac{\ud f_0}{\ud{t}}=f_0'(\mm\cdot{\dot\nn})=(\dot\nn\times\nn)\cdot\CR f_0$.
Hence
\begin{eqnarray}\label{eq:first-order:1}
\langle\frac{\ud f_0}{\ud{t}}, \psi\rangle&=&\langle(\dot\nn\times\nn)\cdot\CR f_0,\CA^{-1}(\Theta\cdot\CR{f_0})\rangle
=-\gamma_1\Theta\cdot(\nn\times\dot\nn).
\end{eqnarray}
Next, we have
\begin{align}\label{eq:first-order:2}
&\langle\CR\cdot(\mm\times\kappa\cdot\mm{f_0}),
\psi\rangle=\langle\CR\cdot(\mm\times(\DD-\BOm)\cdot\mm{f_0}),\psi\rangle\nonumber\\\nonumber
&=-\langle\frac12\CA(\mm\mm:\DD)+(\nn\times(\BOm\cdot\nn))\cdot\CR{f_0}, \CA^{-1}\Theta\cdot\CR f_0\rangle\\\nonumber
&=\langle\mm\times(\DD\cdot\mm),\Theta  f_0\rangle-\langle(\nn\times(\BOm\cdot\nn))\cdot\CR{f_0}, \CA^{-1}\Theta\cdot\CR f_0\rangle\\
&=\Big(S_2\nn\times(\DD\cdot\nn)-\gamma_1\nn\times(\BOm\cdot\nn)\Big)\cdot\Theta.
\end{align}
From (\ref{eq:U1f0}), we have
\begin{eqnarray}
-\langle\CR\cdot({f_0}\CR{U}_1{f_0}),\psi\rangle&=&
\langle{U_1{f}_0}, \Theta\cdot\CR{f_0}\rangle\nonumber\\
&=&\int_\BS\alpha S_2 G
\mm\times(\Delta(\nn\nn)\cdot\mm) \cdot\Theta f_0\ud\mm\nonumber\\ \nonumber
&=&S_2^2\alpha{G}\Theta\cdot\big(\nn\times\Delta\nn\big).
\end{eqnarray}
Let $\hh=S_2^2\alpha G\Delta\nn$, then
\begin{eqnarray}\label{eq:first-order:3}
-\langle\CR\cdot({f_0}\CR{U}_1{f_0}),\psi\rangle
&=&\Theta\cdot(\nn\times\hh).
\end{eqnarray}

Let $\gamma_2=-S_2$. Then combining (\ref{eq:first-order:1}), (\ref{eq:first-order:2}) and (\ref{eq:first-order:3}), we obtain that for any $\Theta\in\BR$,
\begin{eqnarray*}\label{eq:first-order}
\Theta\cdot\Big(\nn\times\big(\gamma_1(\partial_t\nn+\vv_0\cdot\nabla\nn+\BOm\cdot\nn)+\gamma_2\DD\cdot\nn-\hh\big)\Big)=0,
\end{eqnarray*}
which closes our proof. Note that the induced Oseen-Frank energy is
\begin{align}\label{energy:Oseen-Frank-reduced-1}
E_{OF}=\frac{k_F}2 |\nabla\nn|^2,\qquad \text{with } k_F=S_2^2\alpha G.
\end{align}
which coincides the $O(\ve)$ part of (\ref{energy:Oseen-Frank-reduced}).
\end{proof}

\subsubsection{Derivation of the momentum equation}\label{sec:deriv-momentum}

Now we calculate the stress tensor in (\ref{eq:Hilbert-L-11}):
\begin{align*}
 \tau_1&=\int_\BS  \mm\mm\times (\CR{f_1}
+\sum_{\scriptscriptstyle{i+j+k=1}}f_i\CR{U}_j[f_k])\ud\mm,\\
 \tau_2&=\int_\BS f_0\nabla{U}_1[f_0] \ud\mm,\\
 \tau_3&=\int_\BS f_0 \mm\mm(\mm\mm:\DD) \ud\mm.
\end{align*}
From (\ref{eq:U1f0}), we have
\begin{align*}
\tau_2&=\int_\BS  -\frac{\alpha{S}_2}{2}G f_0(\mm\cdot\nn)\nabla\big( \mm\mm:\Delta(\nn\nn)\big) \ud\mm\\
&=\int_\BS \frac{\alpha{S}_2}{2}G \nabla f_0(\mm\cdot\nn) \mm\mm:\Delta(\nn\nn) \ud\mm+\text{gradient term}\\
&=\frac{\alpha{S}_2^2}{2}G \nabla( \nn \nn) :\Delta(\nn \nn )+\text{gradient term}\\
&= -{\alpha{S}_2^2}G \nabla  \nn_k  \Delta \nn_k +\text{gradient term}\\
&= -{\alpha{S}_2^2}G \nabla\cdot \big( \nabla \nn  \odot\nabla \nn \big) +\text{gradient term},
\end{align*}
which is actually the Ericksen tensor $\sigma^E$ up to a pressure term with the Oseen-Frank energy given by (\ref{energy:Oseen-Frank-reduced-1}).

For the stress of $\tau_1$, we can write
\begin{align*}
 \tau_1&=\Big\langle  \mm\mm\times \CR(\CH f_1+U_1[f_0])\Big\rangle_{f_0}.
 \end{align*}
It can be decomposed into two parts: $\tau_1^S$ and $\tau_1^A$, which are symmetric and anti-symmetric respectively. By
the identity (\ref{eq:stress1-sym}), the symmetric part
\begin{align*}
\tau_1^S
&=\frac12 \int_{\BS}(\mm\mm-\frac{1}{3}\II)\Big\{\frac{\pa{f_0}}{\pa{t}}+\vv_0\cdot\nabla{f_0}+\CR\cdot(\mm\times(\nabla\vv_0)^T\cdot\mm{f_0})\Big\}\ud\mm\\
&=\frac12\Big((\partial_t+\vv_0\cdot\nabla)\langle\mm\mm\rangle_{f_0}
+2\DD:\langle\mm\mm\mm\mm\rangle_{f_0}\\
&\qquad\quad-(\DD-\BOm)\cdot\langle\mm\mm\rangle_{f_0}-\langle\mm\mm\rangle_{f_0}\cdot(\DD+\BOm)\Big),
\end{align*}
where in the last equality, we have used the equation (\ref{identity-3}).

For any antisymmetric matrix $\AA$, one has
\begin{align*}
&\big\langle  \mm\mm\times f_0\CR (\CH f_1+U_1[f_0]) \big\rangle_{f_0}: \AA\\
=&\big\langle[ \mm\times(\AA\cdot\mm)]\cdot\CR (\CH f_1+U_1[f_0]) \big\rangle_{f_0} \\
=&-\big\langle[ \mm\times(\AA\cdot\mm)]\cdot\CR (\CH f_1+U_1[f_0]) \big\rangle_{f_0} \\
=&-\big\langle[ \nn\times(\AA\cdot\nn)]\cdot\CR (\CH f_1+U_1[f_0]) \big\rangle_{f_0} \\
=&-\big\langle[ \nn\times(\AA\cdot\nn)]\cdot\CR U_1[f_0] \big\rangle_{f_0} \\
=&-\Big[\nn\times(\AA\cdot\nn)\Big]\cdot(\nn\times\hh)\\
=&\frac12(\nn\hh-\hh\nn):\AA,
 \end{align*}
which implies 
\begin{align*}
  \tau_1^A=\frac12(\nn\hh-\hh\nn). 
\end{align*}

From the equation (\ref{M4}), we have
\begin{align*}
 \tau_3=&\int_\BS f_0 \mm\mm(\mm\mm:\DD) \ud\mm\\
=&~ S_4\nn\nn(\DD:\nn\nn)
+\frac{2(S_2-S_4)}{7}(\nn  \DD\cdot\nn+\DD\cdot \nn\nn)
+2\Big(\frac{S_4}{35}-\frac{2S_2}{21}
+\frac{1}{15}\Big)\DD.
\end{align*}

{Combining the above equations and noting that
\begin{align*}
\hh  =\gamma_1\NN+\gamma_2\DD\cdot\nn,
\end{align*}
we can obtain the following theorem.
\begin{theorem}\label{thm:Hilbert-12}
The equation (\ref{eq:Hilbert-L-11}) is equivalent to
  \begin{align}
 \partial_t\vv_0+\vv_0\cdot\nabla\vv_0=-\nabla{p}+\nabla\cdot(\sigma^E+\sigma^L),
   \end{align}
  where the stress terms are defined as
  \begin{align*}
    \sigma^E=&-k_F \nabla\cdot \big( \nabla \nn  \odot\nabla \nn \big),\\
    \sigma^L=&\frac{\gamma}{Re}\DD+\frac{1-\gamma}{Re}\Big(\alpha_1(\nn\nn:\DD)\nn\nn+\alpha_2\nn\NN+\alpha_3\NN\nn\\
    &\qquad+\alpha_4\DD+\alpha_5\nn\nn\cdot\DD+\alpha_6\DD\cdot\nn\nn\Big),
  \end{align*}
 with coefficients $\alpha_i(i=1,2,\cdots,6)$ given by 
  \begin{align}\label{Leslie cofficients}
\begin{aligned}
&\alpha_1=-\frac{S_4}{2},\quad\alpha_2=-\frac{1}{2}\big(S_2-\gamma_1\big),\quad\alpha_3=-\frac{1}{2}\big(S_2+\gamma_1\big),
\\
&\alpha_4=\frac{4}{15}-\frac{5}{21}S_2-\frac{1}{35}S_4,\quad
\alpha_5=\frac{1}{7}S_4+\frac{6}{7}S_2,\quad
\alpha_6=\frac{1}{7}S_4-\frac{1}{7}S_2.
\end{aligned}
\end{align}
\end{theorem}}


\subsubsection{The higher order terms and the control of the remainder terms}

We can also perform the higher order expansion in a similar way up to any order of $\ve$.
Then we arrive at a series of equations for $(f_i(t,\xx,\mm), \vv_i(t,\xx))$ ($i=0,1,2,\cdots$).
In principle, these high order terms provide more accurate corrections for approximate solutions.
However, it is not known so far whether the expansion in the right hand side of (\ref{expan:fv}) is convergent.
Moreover, even existences of corrector solutions $(f_i(t,\xx,\mm), \vv_i(t,\xx))$ are not clear, although the equations can be derived explicitly.

The rigorous confirmation on the validation of the expansion (\ref{expan:fv}) consists of the following part:
\begin{itemize}
  \item Existence of smooth solutions to $(f_i(t,\xx,\mm), \vv_i(t,\xx))$ for all $i\ge 0$;
  \item The bound of the difference between the true solution $(f_\ve, \vv_\ve)$ and the approximate solutions
\begin{align}\label{expan:fv-1}
\begin{aligned}
  f^{(k)}(t,\xx,\mm)&=\sum_{i=1}^k \ve ^i f_i(t, \xx, \mm ),\\
  \vv^{(k)}(t,\xx)&=\sum_{i=1}^k\ve^i \vv_i(t, \xx ).
\end{aligned}
\end{align}
Indeed, we need the bound tends to $0$ as $\ve\to 0$.
\end{itemize}

The proof of the first part relies  on the local existence of smooth solutions
to the Ericksen-Leslie equations, which is proved in \citeasnoun{WangZZ2013ARMA}.
Indeed, let $(\nn(t,\xx), \vv_0(t,\xx))$ be a solution on time interval $[0,T]$ satisfying
the Ericksen-Leslie system with parameters suitably given, then we can construct a corresponding
solution  $(f_0(t,\xx,\mm), \vv_0(t,\xx))=(h(\mm\cdot\nn(t,\xx)), \vv(t,\xx))$ to (\ref{eq:Hilbert-L-10})-(\ref{eq:Hilbert-L-11}).
The construction of $(f_1(t,\xx,\mm), \vv_1(t,\xx))$ is a little bit subtle, since the system is nonlinear and this implies that the existing time may be small than $T$ in general. However, by carefully examining the inherent structure of the system, one may find that
 $(f_1(t,\xx,\mm), \vv_1(t,\xx))$ could be constructed  by solving a linear system, and thus the existing time can be extended to $[0,T]$. The system for $(f_k(t,\xx,\mm), \vv_k(t,\xx))$ $(k\ge 2)$
 is straightforwardly linear and can be solved directly.

The second part, i.e., estimates for the difference, is much more difficult.  The key ingredients are
the spectral analysis for the linearized operator $\CG_h$ and a lower bound estimate for a bilinear
functional related to a modified nonlocal version of  $\CG_h$. We refer to \citeasnoun{WangZZCPAM} for details.


\subsection{Dynamics: other relations}

The connection between  the dynamic $\QQ$-tensor model and the Ericksen-Leslie model can be directly derived.
One may see \citeasnoun{BerisEdwards}
 or \citeasnoun{QianSheng1998generalized} for the derivation from the Beris-Edward model or Qian-Sheng model
 to the Ericksen-Leslie theory respectively. The correspondence
 of parameters can also be found.  For the rigorous justifications, one may apply the similar procedure discussed in Subsection \ref{sec:Doi-EL}.
  We refer to papers \cite{WangZZ2015SIMA,LiWangZhang2015DCDSB} for the uniaxial limit of the dynamic $\QQ$-tensor models such as
Beris-Edwards model,  the molecular based dynamic $\QQ$-tensor  model and the inertial Qian-Sheng's model respectively.
These results are summarized in Fig. \ref{fig:reduction-dynamic}.

\begin{figure}
    \begin{center}\label{fig:reduction-dynamic}
    \includegraphics[width=0.6\columnwidth]{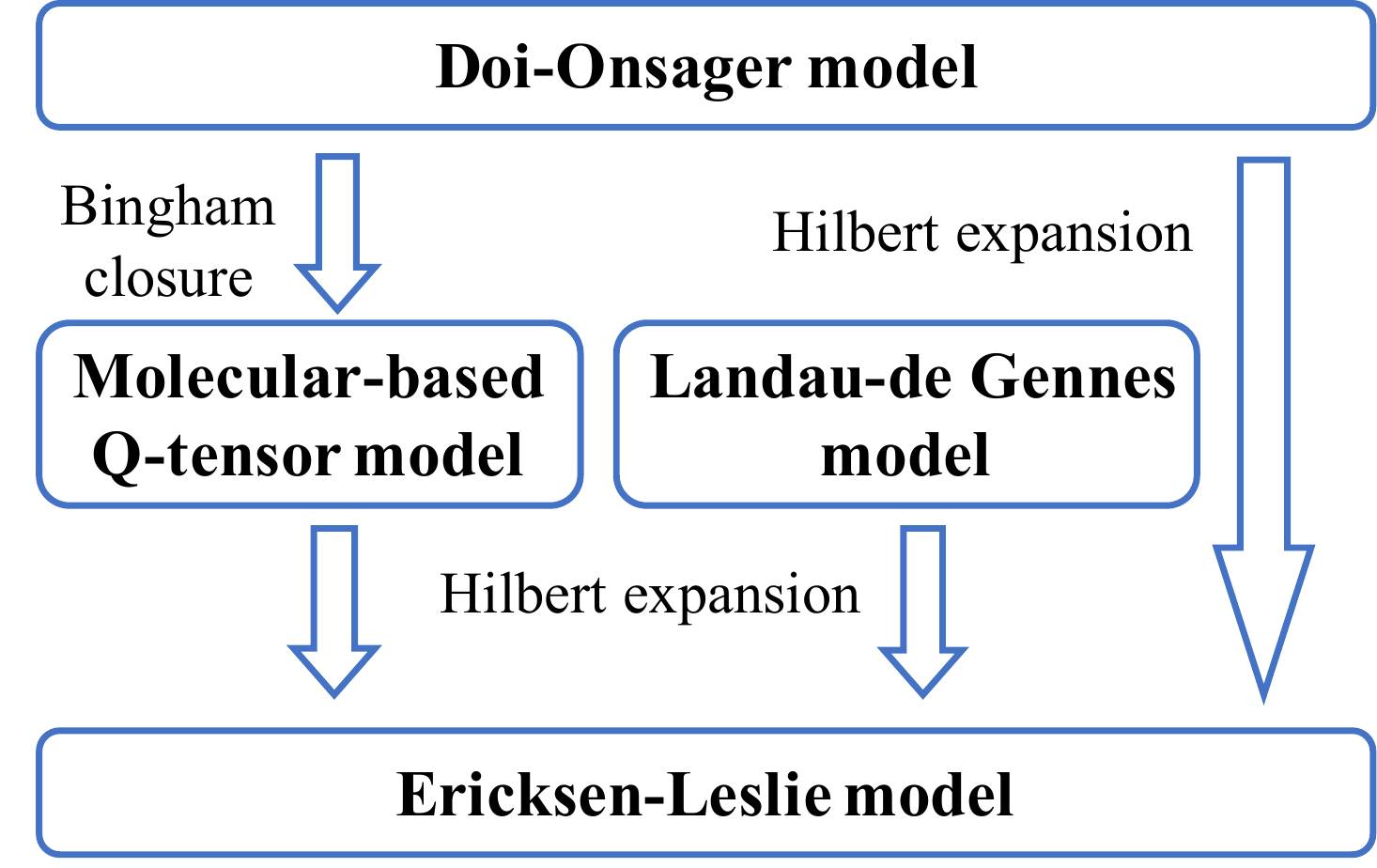}
    \caption{Reduction to Ericksen-Leslie theory from the Doi-Onsager theory or the dynamic $\QQ$-tensor theory}
    \end{center}
\end{figure}

The above discussions build a solid coincidence between the Ericksen-Leslie theory and the Doi-Onsager theory or the dynamic $\QQ$-tensor theory.
However, they are based on the framework of smooth solutions which relies on the existence
of smooth solution of the Ericksen-Leslie system. Thus it excludes the possibility of defects. To resolve this problem, one needs to study the
relations between weak solutions. \citeasnoun{LiuWangJFA} gives an attempt on this issue, where it is proved that, the solutions to the Doi-Onsager equation
without hydrodynamics converges to the weak solution of the harmonic map heat flow--the gradient flow of the one-constant Oseen-Frank energy.

\section{Numerical methods for computing stable defects of liquid crystals}
In this section, we review recent developments on the numerical methods for computing stable defect patterns of LCs. We will focus on the NLC systems again and apply the LdG theory as the model system to provide a sample of relatively new progress on the development of numerical algorithms that are applicable to general LC problems.
Two numerical approach will be reviewed in order to obtain a local minimum for a given energy functional, one is the energy-minimization based approach, which follows the idea of ``Discretize-then-Minimize'', and the other is to solve the gradient flow dynamics.

\subsection{Energy-Minimization Based Approach}

The energy-minimization based approach first discretizes the free energy by introducing a suitable spatial discretization to the order parameter of system, then adopts some optimization methods to compute local minimizers of the discrete free energy. From an optimization point of view, solving the gradient flow equation corresponds to the gradient decent method on the discrete free energy. An energy-minimization based approach enables us to apply some advanced optimization methods, such as Newton-type methods or quasi-Newton methods, which may be able to find local minimizers of the discrete free energy more efficiently.

The idea of computing defect structures in LC by an energy-minimization based approach
can be traced back to some early work on Oseen-Frank model in late 1980s.
\citeasnoun{cohen1987minimum} obtain several LC configuration by numerically minimizing the Oseen-Frank energy. However, due to the lack of convexity of the unit-length constraint,  the convergence of their algorithms is difficult to be established.  A significant improvement is made by \citeasnoun{alouges1997new}, who proposed an energy-decreasing algorithm for one constant Oseen-Frank model. The convergence of this algorithm is proved in a continuous setting.
Later, \citeasnoun{adler2015energy} proposed an energy-minimization finite-element approach to Oseen-Frank model by using Lagrangian multiplier and the penalty method, which can be applied to the cases with elastic anisotropy and electric/flexoelectric effects \cite{adler2015energy,adler2016constrained}.
A surface finite element method was developed for the surface Oseen-Frank problem to study the orientational ordering of NLCs on curved surfaces \cite{nitschke2018nematic,nestler2018orientational}.
For Ericksen model, \citeasnoun{nochetto2017finite} proposed a structure preserving discretization of the LC energy with piecewise linear
finite element and develop a quasi-gradient flow scheme for computing discrete equilibrium solutions that have the property of strictly monotone decreasing energy. They also prove $\Gamma$-convergence of discrete global minimizers to continuous ones as the mesh size goes to zero. Similar idea has also been applied to generalized Ericksen model with eight independent ``elastic'' constants \cite{walker2020finite} and an uniaxially constrained $\Qvec$-tensor model \cite{borthagaray2019structure}.


For the full $\Qvec$-tensor model, \citeasnoun{gartland1991numerical} constructed a numerical procedure that minimizes the LdG free energy model, which is based on a finite-element discretization to the tensor order parameter, and a direct minimization scheme based on Newton's method and successive over-relaxation. The corresponding analytical and numerical issues of this numerical procedure are addressed in \citeasnoun{davis1998finite}, in which the well-posedness of the discrete problem are proved. Besides more physically realistic, the full $\QQ$-tensor model has an advantage that the corresponding optimization problem is almost unconstrained (as the eigenvalue constraint is normally satisfied due to the boundary condition in a certain parameter region), although it might require more computational cost since $\Qvec$ has five degrees of freedom.

In recent years, we have incorporated the spectral method with the energy-minimization techniques and successfully applied our approach to different confined LC systems, including  three-dimensional spherical droplet \cite{hu2016disclination}, three-dimensional cylinder \cite{hu2016disclination,han2019transition}, nematic shell \cite{qu_wei_zhang_2017}, nematic well \cite{yin2020construction} and LC colloids \cite{wang2017topological,tong2017defects,wang2018Formation}.
From a computational perspective, as an efficient numerical method with high accuracy, spectral method makes an accurate free-energy calculation for 3D problems possible and enable us to determine the phase diagram of some complicated LC systems \cite{wang2017topological}.

To apply the spectral method to different confined LC systems, one can either identify an appropriate coordinates system, i.e., map the physical domain to a regular computational domain \cite{hu2016disclination,wang2017topological}, or phase-field type method \cite{wang2018Formation} to deal with the underlying geometry and the boundary conditions.
Then the free energy can be discretized by introducing a spectral discretization to the order parameter $\Qvec$. Since the order parameter $\Qvec$ is a traceless symmetric tensor, it can be written as
\begin{equation}
    \mathbf{Q} =
    \begin{bmatrix} q_1 & q_2 & q_3 \\ q_2 &  q_4 & q_5 \\ q_3 & q_5 & -q_1-q_4 \end{bmatrix}.
        \label{Qtensor}
\end{equation}
We can introduce a spectral approximation to $q_i$ separately.
The local minimizers of the resulting discrete free energy $\mathcal{F}_h({\bm \Xi})$ can be computed by some optimization methods, such as the Broyden-Fletcher-Goldfarb-Shanno (BFGS) method. Here ${\bm \Xi} \in \mathbb{R}^{5 \times K}$ consists of all expansion coefficients and $K$ is the number of basis functions.

The key step in the above numerical procedure is to compute the gradient of the discrete free energy $\mathcal{F}_h(\bm{\Xi})$ with respect to ${\bm{\Xi}}$. To illustrate the idea, we consider a simple system with a scalar order parameter $\varphi$ and a free energy
\begin{equation}\label{free_energy_example}
  \mathcal{F}[\varphi] = \int W(\varphi, \nabla \varphi) \dd \x.
 \end{equation}
We can discretize the free energy $\mathcal{F}(\varphi)$ by introducing a spectral approximation to
\begin{equation}
\varphi = \sum_{i = 1}^K \xi_i \phi_i,
\end{equation}
where $\phi_i$ is the basis function, and $\xi_i$ are the expansion coefficients that needed to be determined during the optimization procedure. The derivative of $\mathcal{F}_h(\bm \Xi)$ with respect to $\xi_i$ can be computed directly as
\begin{equation}
\frac{\partial}{\partial \xi_i} \mathcal{F}_h({\bm \Xi)} = \int \frac{\partial W}{\partial \varphi} \phi_i + \frac{\partial W}{\partial \nabla \varphi } \cdot \nabla \phi_i \dd \x,
\end{equation}
which is easy to compute by a numerical integration.
Noticed that the weak form of the Euler-Lagrangian equation of the free energy (\ref{free_energy_example}) is given by
\begin{equation}\label{discrete_derivative}
\left(\frac{\partial W}{\partial \varphi},  \psi \right) + \left( \frac{\partial W}{\partial \nabla \varphi},\nabla \psi \right) = 0,
\end{equation}
where $\psi$ is a test function and $(.,.)$ is the standard $L^2$-inner production. Hence, the stationary points of the discrete free energy $\mathcal{F}_h(\bm{\Xi})$ are exactly numerical solutions of the Euler-Lagrangian equation $\frac{\delta \mathcal{F}}{\delta \varphi} = 0$ obtained by a Galerkin method. The idea of discretization first can be extended to dynamics cases with variational structure, we refer the interesting readers to \citeasnoun{doi2019application} and \citeasnoun{liu2020lagrangian} for some recent developments. In particular, by using the strategy of ``discretize-then-variation'', \citeasnoun{liu2020variational} proposed a variational Lagrangian scheme for a phase-filed model, which has an advantage in capturing the thin diffuse interface with a small number of mesh points. Such a numerical methodology can be definitely applied to a liquid crystal system and will have potential advantages in computing the defect structures.

Next, we discuss the choice of basis function. For three-dimensional unit sphere, \citeasnoun{hu2016disclination} choose the basis function as Zernike polynomials $Z_{nlm}$ \cite{zernike1934diffraction,mahajan2007orthonormal}, defined by
\begin{equation}
Z_{nlm}(r,\theta,\phi)=R_n^{(l)}(r)Y_{lm}(\theta,\phi),
\end{equation}
where
\begin{equation*}
  \begin{aligned}
 & R_n^{(l)} =\left\lbrace
\begin{array}{ll}
\sum_{s=0}^{(n-l)/2} N_{nls}r^{n-2s} & 0\leq \frac{n-l}{2}\in \mathbb{Z},\\
0 & \text{otherwise}, \\
\end{array}
\right. \\
& N_{nls} = (-1)^s \sqrt{2n+3} \prod_{i=1}^{n-l}(n+l-2s+1+i)\prod_{i=1}^l\left( \frac{n-l}{2}-s+i \right) \frac{2^{l-n}}{s!(n-s)!}, \\
\end{aligned}
\end{equation*}
and  $Y_{lm}(\theta,\phi)=P_l^{|m|}(\cos\theta)X_m(\phi)$ are the spherical harmonic functions with
\begin{equation*}
X_m (\phi) =
    \left\lbrace
    {\begin{array}{ll}
        \frac{1}{\pi}\cos m \phi & m\geq0,\\
        \frac{1}{\pi}\sin |m|\phi & m<0,
    \end{array}}
    \right.
\end{equation*}
and $P_l^m(x) (m \geq 0)$ be the normalized associated Legendre polynomials. The Zernike polynomials $Z_{nlm}$ are orthogonal in three-dimensional unit sphere, which can simplify the computation in (\ref{discrete_derivative}). Similarly, for 2D disk and cylindrical domain \cite{hu2016disclination,han2019transition}, the authors choose 2D Zernike polynomials $Z_{nl}$ as basis functions.

The orthogonality of the basis functions is not required in above the numerical procedure. A disadvantage of using Zernike polynomials is that the value of $R_n^{(l)}(r)$ for given $r$ is not easy to compute in high accuracy. The standard basis function, such Legendre and Chebyshev polynomials, might be better choices in numerical calculation for more general problems. For example, for spherical shells and unbounded domain outside one or two spheres \cite{wang2017topological,tong2017defects,noh2020dynamic}, one can first map the physical domain to a computational domain
$$\Omega = \{(\zeta, \mu, \varphi) | -1 \leq \zeta \leq 1, 0 \leq \mu < \pi, 0 \leq \varphi < 2\pi \},$$
and choose  real spherical harmonics of $(\mu, \varphi)$ and Legendre polynomials of $\zeta$ as the basis functions.

To validate the algorithm, \citeasnoun{hu2016disclination} compared the numerical results of the radial hedgehog solution with its analytic form \cite{gartland1999instability,majumdar2012radial}.
The radial hedgehog solution is a radial symmetric solution, in which the $\mathbf{Q}$ is given by
\begin{equation*}
\mathbf{Q}_{RH}(\mathbf{r}) = h(|\mathbf{r}|)\left(\frac{\mathbf{r}}{|\mathbf{r}|}\otimes\frac{\mathbf{r}}{|\mathbf{r}|}-\frac{1}{3}\mathbf{I}\right),\ 0<|\mathbf{r}|\leq 1,
\end{equation*}
with $h(r)$ satisfies the second order ODE (\ref{config:heghog-h}), which can be solved accurately.
By increasing the number of the basis in the Zernike polynomials using $N=4k,L = 16, M = 4$, the numerical error in the total free-energy decreases to as low as $10^{-10}$ \cite{hu2016disclination}. For more complicated solutions (without radial symmetry), \citeasnoun{wang2017topological} shows the numerical error in free energy calculation for the dipole and Saturn ring defect around a spherical particle immersed in NLC. These numerical tests suggest that such numerical method is adequate for an accurate free energy calculation to determine the phase diagram of the system.

Finding minimizers of the LdG free energy is an unconstrained nonlinear optimization problem. The minimizers can be obtained by optimization methods such as gradient descent method and quasi-Newton methods.
A commonly used optimization method for such problem is the Limit-memory Broyden-Fletcher-Goldfarb-Shanno (L-BFGS) method
\cite{knyazev2001toward}.
The energy-minimization based numerical approach with L-BFGS usually converges to a local minimizer with a proper initial guess, but that is not necessarily guaranteed.
To check whether the solution is a local minimizer, we need to justify the stability of the solution on the energy landscape $F$ by computing the smallest eigenvalue $\lambda_1$ of its Hessian $\mathbb{H}(\mathbf{x})=\nabla^2 F(\mathbf{x})$,
the associated second variation of the reduced energy corresponding to $\mathbf{x}$ \cite{yin2019high,canevari_majumdar_wang_harris}:
\begin{equation}
\lambda_1 = \min_{\mathbf{v}\in\mathbb{R}^n} \quad  \frac{ \langle \mathbb{H}(\mathbf{x})\mathbf{v},\mathbf{v} \rangle}{ \langle \mathbf{v},\mathbf{v} \rangle} ,
\end{equation}
where $\langle\cdot,\cdot\rangle$ is the standard inner product in $\mathbb{R}^n$. The solution is locally stable or metastable if $\lambda_1>0$.
Practically, $\lambda_1$ can be computed by solving the gradient flow equation of $\mathbf{v}$
\begin{equation}
\frac{\partial\mathbf{v}}{\partial t} = -\frac{2\gamma}{\langle\mathbf{v},\mathbf{v}\rangle}\left(\mathbb{H}\mathbf{v}-\frac{\langle\mathbb{H}\mathbf{v},\mathbf{v}\rangle}{\langle\mathbf{v},\mathbf{v}\rangle}\right).
\label{eq:eigenvector}
\end{equation}
where $\gamma$ is a relaxation parameter and $\mathbb{H}(\mathbf{x})\mathbf{v}$ can be approximated by
\begin{equation}
\mathbb{H}(\mathbf{x})\mathbf{v} \approx -\frac{\nabla F(\mathbf{x}+l\mathbf{v})-\nabla F(\mathbf{x}-l\mathbf{v})}{2l},
\end{equation}
for some small constant $l$. We can choose $\gamma(t)$ properly to accelerate the convergence of the dynamic system \eqref{eq:eigenvector} \cite{yin2019high}.

\begin{figure}
    \begin{center}
    \includegraphics[width= 0.9 \columnwidth]{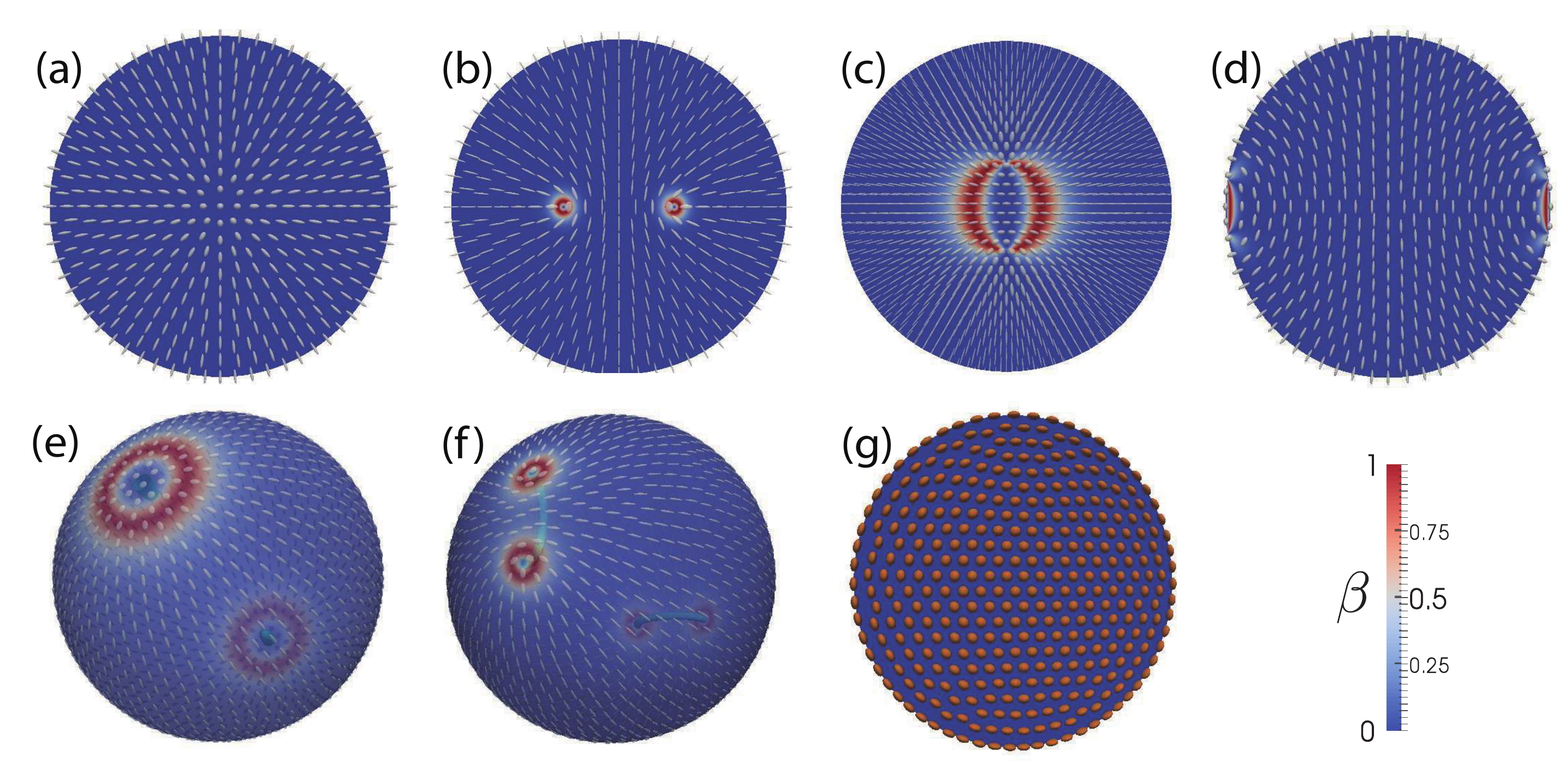}
    \caption{(Credit: \protect\cite{hu2016disclination}) Possible equilibria under various anchoring condition. (a-c) Three equilibria under the strong radial anchoring condition on the $xz$-plane: (a) radial hedgehog; (b) ring disclination and (c) split core (a zoom-in view).
    (d) Equilibrium under the relaxed radial anchoring condition.
    (e-f) Two equilibria for the planar anchoring condition.
    (g) A uniaxial solution for planar anchoring condition. $\beta = 0$ and $\mathbf{Q}$ is oblate everywhere.
	Colorbar for $\beta = 1-6\frac{(tr\mathbf{Q}^3)^2}{(tr\mathbf{Q}^2)^3}$ shown in (a)-(g), with red indicates biaxial and blue indicates uniaxial.
	$\beta$ (represented by color) and $\mathbf{Q}$-tensor (represented by ellipsoid glyph) from numerical simulation.
	The cyan tubes inside the ball in (e-f) are the iso-surfaces of $c_l = \lambda_1-\lambda_2$, the difference between two largest eigenvalues of $\mathbf{Q}$.}
    \label{ball}
    \end{center}
\end{figure}

First example of minimizing the LdG free energy is to calculate the possible equilibria of the NLCs confined in three-dimensional ball with different kind of anchoring condition \cite{hu2016disclination}. For the strong radial anchoring condition, 
three different configurations: the radial hedgehog, ring disclination and split core solutions (Figure \ref{ball}(a-c)) are obtained. In the radial hedgehog solution $\mathbf{Q}$ is uniaxial evergywhere with a central point defect, which is a rare example of pure uniaxial solution for Landau-de Gennes model \cite{lamy2015uniaxial,majumdar2018remarks}. For low temperature and large domain size, the point defect broadens into a disclination ring. The disclination ring solution is a symmetry breaking configuration. The split core solution contains a +1 disclination line in the center with two isotropic points at both ends, the existence of which is proved in \citeasnoun{yu2020disclinations} under the rotational symmetric assumption. The authors also consider relaxed radial anchoring condition, which allows $s = s(\mathbf{x})$ to be a free scalar function on $\partial\Omega$. Besides the radial hedgehog, disclination ring and split core configurations, one more solution is obtained for this boundary condition shown in Figure \ref{ball}(d). Two rings of isotropic points form on the sphere. Between the two rings, on the surface $\mathbf{Q}$ is uniaxial and oblate ($s<0$). Inside the ball, there is a strong biaxial region close to the surface. For a planar boundary condition, the authors find three stable solutions (Figure \ref{ball}(e-f)). In Figure \ref{ball}(e), two +1 point defects form at two poles. As temperature decreases, the point defect on the surface splits into two +1/2 point defects. In Figure \ref{ball}(f), the four +1/2 point defects on the sphere form the vertices of a tetrahedron. Inside the ball, there are disclination lines which intersect with the spherical surface on the above mentioned point defects. The configuration in Figure \ref{ball}(e) has two segments of +1 disclination lines with one isotropic end buried inside the ball and the other end connecting the surface. As temperature decreases, the +1 disclination splits into a +1/2 disclination with both ends open at the surface. Another metastable solution is shown in Figure \ref{ball}(g). In this configuration $\mathbf{Q}$ is uniaxial and everywhere and has radial symmetry. But unlike the radial hedgehog solution, $\mathbf{Q}$ is oblate ($s<0$) everywhere rather than prolate.

\begin{figure}
    \begin{center}
    \includegraphics[width=0.8\columnwidth]{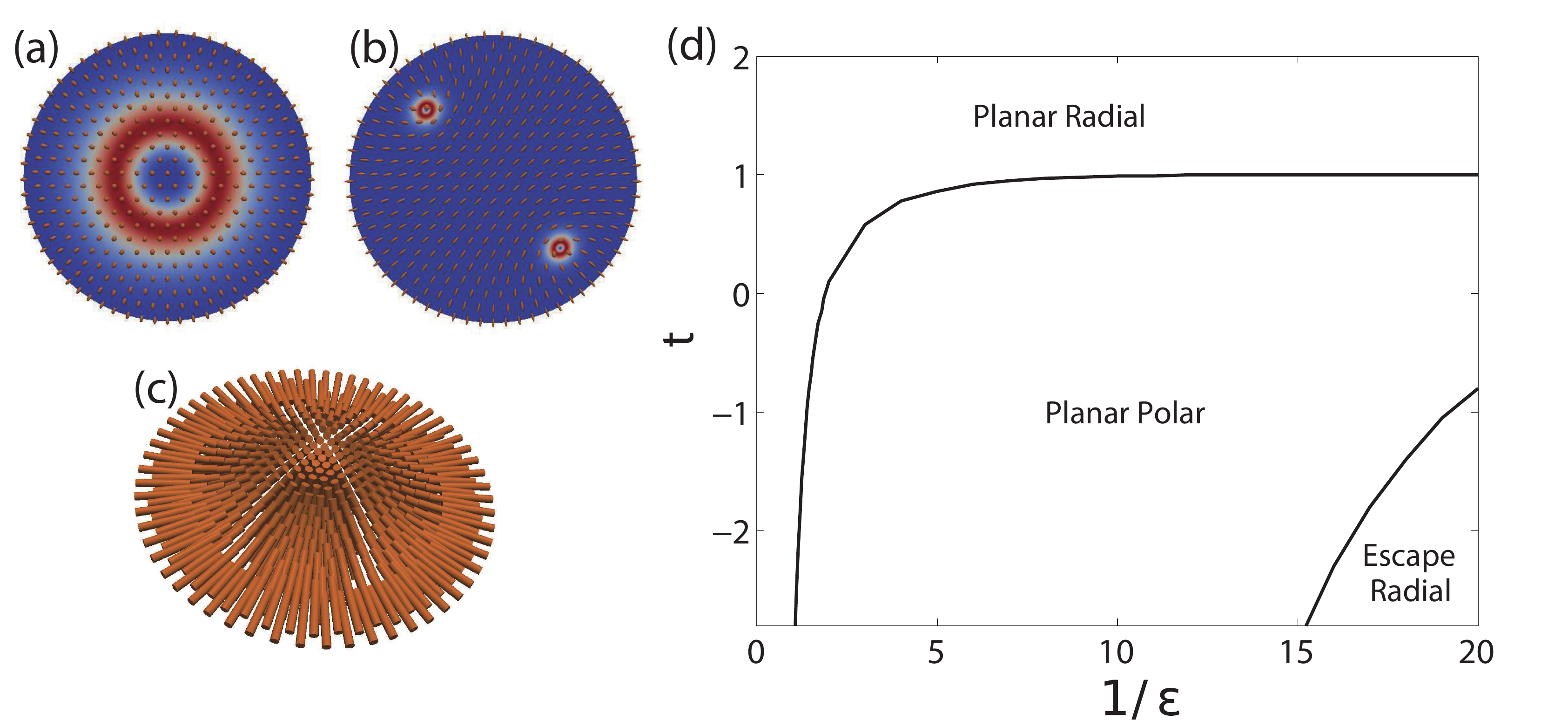}
    \caption{(Credit: \protect\cite{hu2016disclination}) Three equilibria: (a) planar radial (b) planar polar and (c) escape radial for radial anchoring condition on a 2D disc. $\beta$ is shown in color in (a-b) with red corresponds to biaxial and blue uniaxial. Ellipsoids represent the $\QQ$-tensor. Golden solid bars in (c) represent the eigendirection corresponding to the largest eigenvalue. (d) Phase diagram of the planar radial, planar polar and escape radial configurations. The partition is based on the lowest energy of the three.}
    \label{PR_PP_ER}
    \end{center}
\end{figure}

For 2D disk and 3D cylinder, \citeasnoun{hu2016disclination}
obtained minimizers on unit disc with boundary condition $\mathbf{Q}(\cos\phi, \sin\phi) = s_+(\mathbf{n}\otimes\mathbf{n}-\frac{\mathbf{I}}{3})$, where $\mathbf{n} = (\cos\frac{k}{2}\phi,\sin\frac{k}{2}\phi,0)$, $k = \pm1,\pm2,\cdots$, is kept in the $xy$-plane at the boundary. They found that the solutions are predictable. For large temperature and large domain, there is a semi-radial solution in which all the eigenvalues are rotational symmetric and the eigenvectors are invariant along the $r$-direction determined by the boundary condition. The semi-radial solution has a central defect with winding number determined by the boundary constraint. As temperature decreases and domain size increase, the semi-radial solution become unstable and the central defect quantize to defect points with +1/2 or -1/2 winding number. The number of 1/2 defects is determined by the conservation of the total winding number. When $k$ of the boundary condition is even, there is a non-singular harmonic map solution, a phenomena referred as ``escape into the third dimension" in \cite{sonnet1995alignment}. Both the escape solution and the quantized $\pm 1/2$ solutions are stable for low temperature. As the temperature decreases further, the free energy of the escape solution will be lower.
When $k = 2$ or with planar boundary condition, there are three known configurations: planar radial (PR), planar polar (PP) and escape radial (ER) as shown in Figure \ref{PR_PP_ER}. The planar radial has one +1 point defect at the center; the planar polar solution has two +1/2 point defect form at the opposite site of the disc; the escape radial has no defect in which $\mathbf{Q}$ is uniaxial everywhere with $s$ being constant and $\mathbf{n}$ being a harmonic map for the given boundary condition. A phase diagram for these three configurations is shown in Figure \ref{PR_PP_ER}.

\begin{figure}
    \begin{center}
    \includegraphics[width=\columnwidth]{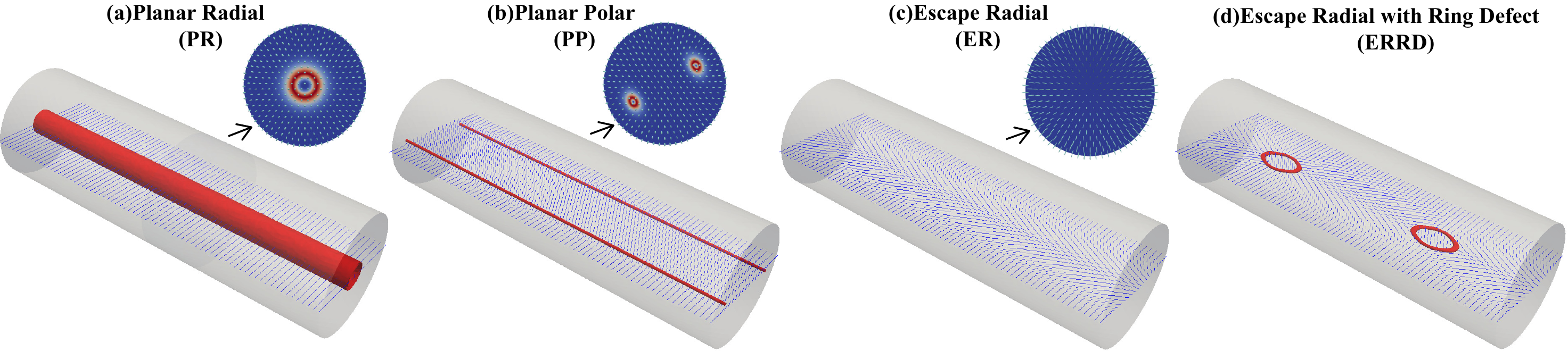}
    \caption{(Credit: \protect\cite{han2019transition}) Four equilibria of NLC confined in cylinder with homeotropic anchoring condition (a) Planar Radial; (b) Planar Polar; (c) Escape Radial and (d) Escape Radial with Ring Defect. Disclination lines are represented by the red isosurface of $c_l$. The disk in the upper-right corner is the transversal view of structure in (a)-(c). The ellipsoid on disk represents the $\mathbf{Q}$-tensor on the disc and the color corresponds to $\beta$, ranging from 0 (blue) to 1 (red).}
    \label{PR_PP_ER_ERRD}
    \end{center}
\end{figure}

With z-axial invariance, the PR, PP and ER (Figure \ref{PR_PP_ER_ERRD}(a-c)) are still stable equilibria of NLC confined in a cylinder with homeotropic boundary condition. Without axial invariance, the escape radial with ring defect (ERRD) configuration (Figure \ref{PR_PP_ER_ERRD}(d)) is also stable equilibria which has two disclination rings lying on a plane parallel to the axial direction of the cylinder. Two disclination rings can be considered as the broadening of two point defects with topological charge $+1$ and $-1$ \cite{Brada2003Annihilation}. The ERRD solution can be considered as a quenched metastable state formed by jointing ER configurations with opposite pointing directions, which means the free energy of ERRD is always higher than ER.

For the liquid crystal colloids (LCC) system, \citeasnoun{wang2017topological} presented a detailed numerical investigation to the LdG free-energy model under the one-constant approximation for systems of single and double spherical colloidal particles immersed in an otherwise uniformly aligned NLC. For the strong homeotropic anchoring with one spherical particle, two types of configurations, quadrupolar (also known as Saturn-ring structure) and dipolar states, illustrated in Figure \ref{single}, are obtained.
In the quadrupolar state,
a $-\frac{1}{2}$ disclination line forms a ring located at the spherical equator  and the entire $\mathbf{Q}$ tensor has an axisymmetry about the $z$ axis and reflection symmetry with respect to the $x-y$ plane through the spherical center.
The dipolar solution is an axisymmetric configuration and contains no reflection symmetry with respect to the $x-y$ plane. The defect ring in dipolar is near the spherical top. The phase diagram for the single-particle problem is shown in Figure \ref{single}. Below the critical temperature of isotropic-nematic phase transition, the dipolar state is stable for large particle systems and can only be found to the right of the dashed curve, which represents the stability limit of it. The quadrupolar pattern can be found as the stable or metastable state over the entire parameter space below the critical temperature.

\begin{figure}
    \begin{center}
    \includegraphics[width=\columnwidth]{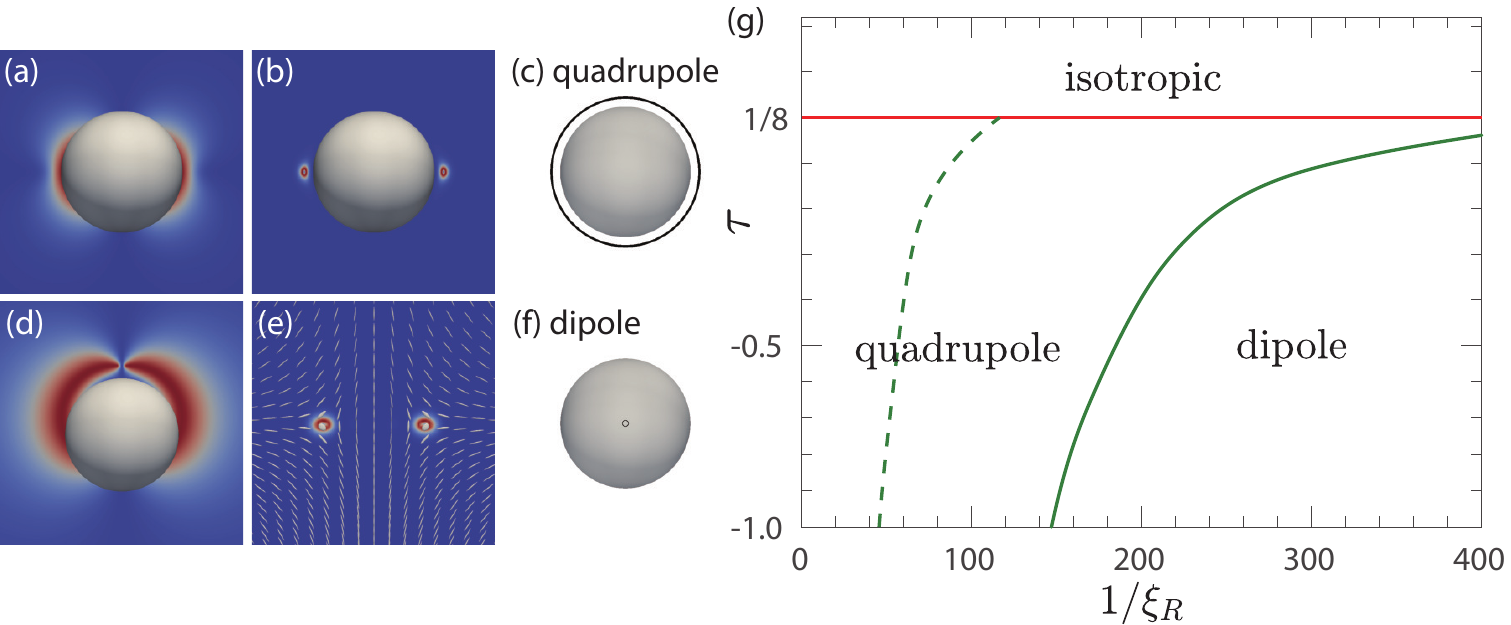}
    \caption{(Credit: \protect\cite{wang2017topological}) Two equilibria of NLC surrounding
a spherical particle: (a)-(c) quadrupolar state,
(d)-(f) dipolar state, as well as the corresponding regimes in the
phase diagram (g) where these states have the lowest free energies and the dashed line
represents the stability limit of the dipolar state.
Each state is illustrated by using three views: side views of the
$\QQ$-tensor element $Q_{11}$ [in (a) and (d)], the same side views of
the biaxiality coefficient $\beta$ [in (b) and (e)], and top view of the
defect location, indicated by the black circles produced from plotting
the isosurface of $c_l = 0.1$. (e) is an enlarged version of the defect
area on top of the sphere, where the tensor field (white ellipsoids)
indicates a $-1/2$ defect line.}
    \label{single}
    \end{center}
\end{figure}

To study the case with two spherical particles of equal radii are placed in a nematic fluid, \citeasnoun{wang2017topological}
introduced the bispherical coordinates $(\xi,\mu,\phi)$ ($-\xi_0\leq\xi\leq\xi_0$,$0\leq\mu<\pi$ and $0\leq\phi<2\pi$) \cite{fukuda2004interaction}, and choose real spherical harmonics of $(\mu, \varphi)$ and Legendre polynomials of $\zeta$ ($\zeta = \xi/\xi_0$) as basis functions.
The relation between bispherical coordinates $(\xi,\mu,\phi)$ and Cartesian coordinates $(x,y,z)$ is
\begin{equation*}
\begin{aligned}
  x &= a\frac{\sin \mu}{\cosh\xi - \cos\mu} \cos\phi,, y = a\frac{\sin \mu}{\cosh\xi - \cos\mu} \sin\phi, z = a\frac{\sinh\xi}{\cosh\xi - \cos\mu},
\end{aligned}
\end{equation*}
where $a = \sqrt{(D/2)^2-R^2}$, $R$ is radius of balls, $D$ is the distance between the centers of the two spherical particles. At a fixed $\phi$, $\xi=\mu=0$ corresponds to infinity, and the surface of constant $\xi$ represents a sphere given by
\begin{equation}
x^2+y^2+(z - \mathrm{acoth} \xi)^2=\frac{a^2}{\sinh^2\xi}.
\end{equation}
So the surfaces of the two spherical particles are represented by $\xi = \pm\xi_0$, and $\xi_0 = \cosh^{-1}(D/2R)$. Inspired by the configuration of single particle in NLC, \citeasnoun{wang2017topological} consider the dimer complex composed of dipole-dipole pair, quadrupole-quadrupole pair and dipole-quadrupole pair. They find three stable state configuration: entangled hyperbolic defect ($\rm H$), unentangled defect rings (${\rm U}_{\gamma}$) and parallel dipoles (${\rm D}_0$) in the dimer system after the free energy is minimized with respect to both the far-field nematic director (represented by $\gamma \in [0, \pi / 2]$) and inter-particle distance $D/R$. Both $H$ and $U_{\gamma}$ states are variations of the quadrupolar structure in Figure \ref{single}. $H$ state, in which $\gamma = \pi/2$, is only stable in the systems with extremely small particle size.
 Due to the two defect rings both have the same winding number -1/2, in $U_{\gamma}$ the portions of the defect rings near the sphere-sphere center locally repel each other. As these portions are twisted upwards and downwards, the original spherical axes tilt in order to accomodate a larger distance between these two repelling portions. The value of the optimal $\gamma$ deviates from $\gamma=\pi/2$ starting from the isotropic-nematic transition line and increases as the system moves to a smaller particle size state. In $D_0$, the far field tilt angle $\gamma = 0$.
Beyond the three free-energy ground states, six metastable states can be computed. We refer the interested reader to \citeasnoun{wang2017topological} for more detailed discussions.

\begin{figure}
    \begin{center}
    \includegraphics[width=\columnwidth]{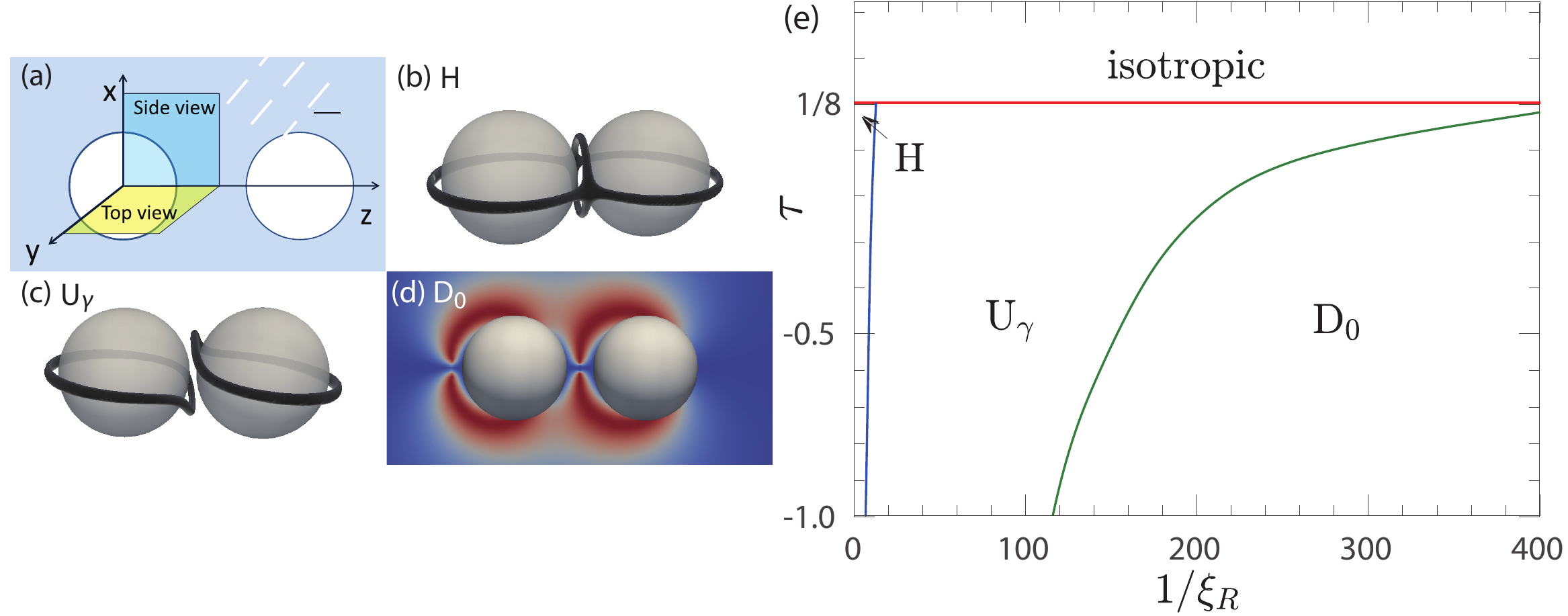}
    \caption{(Credit: \protect\cite{wang2017topological}) (a) Sketch of the coordinate system for a dimer problem;
(b) entangled hyperbolic defect (H) state where $\gamma = \pi/2$; (c) unentangled
defect rings (U$_{\gamma}$) where $\gamma= \pi/2$; (d) parallel dipole-dipole
state (D$_0$) where $\gamma = 0$; and (e) the phase diagram that describes the
regimes where these defect states have free-energy minima in terms
of a reduced temperature $\tau$ and reduced spherical radius $1/\xi_R$. The
phase diagram was determined based on the ground-state calculation,
after consideration of all other possibilities including minimization
with respect to the particle distance $D$ and tilt angle $\gamma$. (b) and
(c) are illustrations of the defect lines determined from isosurface of the largest eigenvalue of $\mathbf{Q}$, $\lambda_1 = 0.25$. (d) is the cross-section view of $Q_{11}$ .}
    \label{double}
    \end{center}
\end{figure}

 For more complex geometries, phase-field approaches or diffuse interface methods \cite{yue2004diffuse} can be incorporated into the above numerical framework. For example, combining the phase-field method with Fourier spectral method, \citeasnoun{wang2018Formation} investigated the formation of three-dimensional colloidal crystals in a nematic liquid crystal dispersed with large number of spherical particles. The corresponding defect structures in the space-filler nematic liquid crystal induced by the presence of the spherical surface of the colloids are computed. Multiple configurations are found for each given particle size and the most stable state is determined by a comparison of the free energies. Their numerical studies show that from large to small colloidal particles, a sequence of 3D colloidal structures, which range from quasi-one-dimensional (columnar), to quasi-two-dimensional (planar), and to truly three-dimensional, are found to exist.

 Besides the NLCs, similar numerical technique can also be applied to cholesteric LCs. \citeasnoun{tong2017defects} investigated the defect structures around a spherical colloidal particle in a cholesteric LC, i.e., $\Omega = \mathbb{R}^3 \backslash B(0, 1)$. In order to deal with the inhomogeneity of the cholesteric at infinity, they first identify the ground state $\Qvec_0$ and use spectral method to approximate $\Qvec - \Qvec_0$. Instead of using  classical orthogonal systems on the unbounded domain, such as Laguerre polynomials, they combine the exponential mapping and the truncation techniques to capture the property of $\Qvec - \Qvec_0$. The mapping between the computational domain $(\rho,\theta,\phi)$ and the physical domain $(r,\theta,\phi)$ are given by
 \begin{equation}
r = \sinh \left( \frac{\rho + 1}{2} L \right), \quad \rho \in [-1. 1],
  \end{equation}
 where $r$ is the radial distance in the spherical coordinates. Fig. \ref{Tong} shows two types of defect configurations obtained in a cholesteric LC by their numerical simulation for strong homeotropic anchoring condition : twisted Saturn ring (Figs. \ref{Tong}(a-c)) and cholesteric dipole (Figs. \ref{Tong}(d-f)).
\begin{figure}
    \begin{center}
    \includegraphics[width= 0.7 \columnwidth]{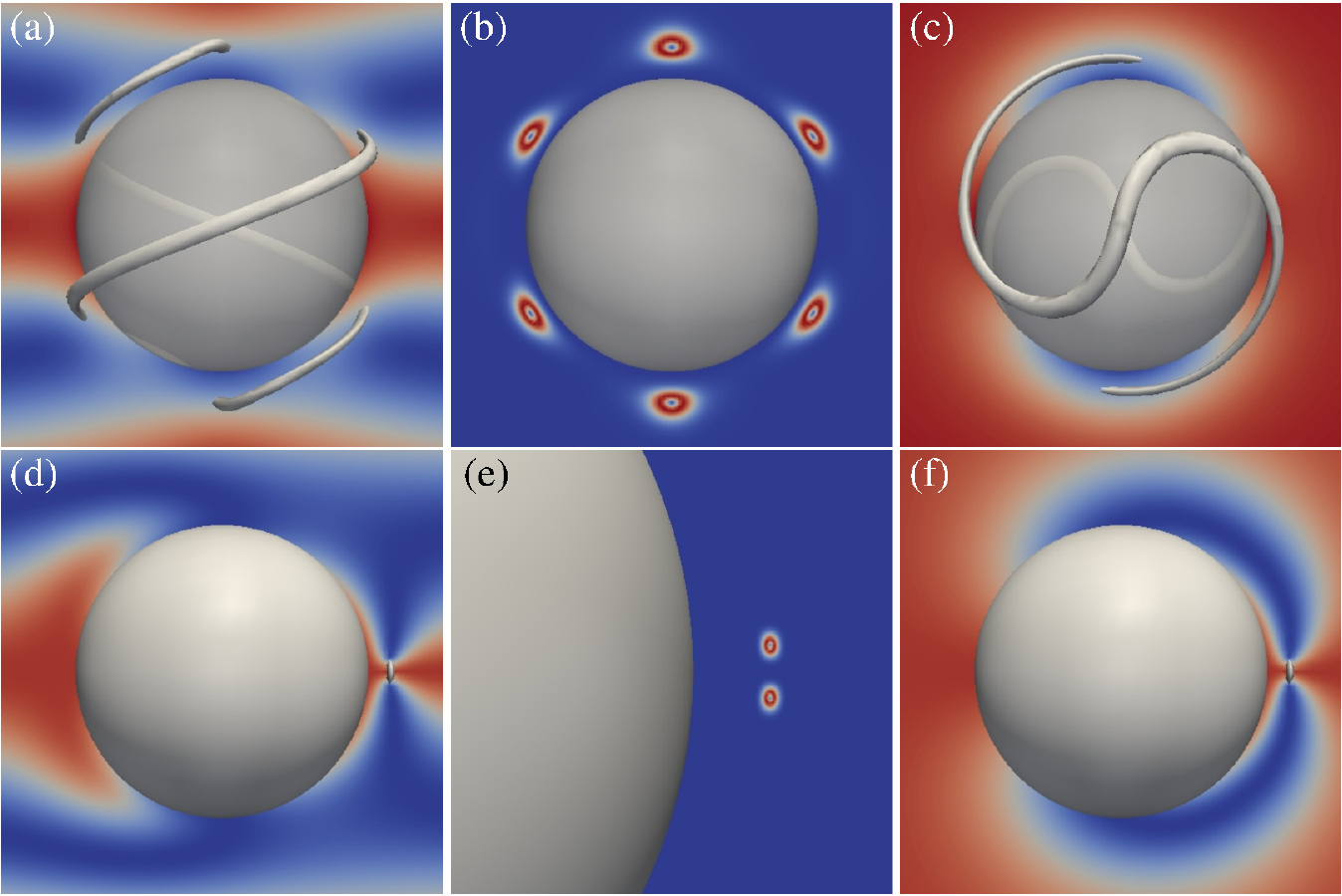}
    \caption{(Credit: \protect\cite{tong2017defects}) Two possible configurations around a spherical particle: (a-c) The defect structures in the twisted Saturn ring. (d-f) The defect structures in the cholesteric dipole.}
    \label{Tong}
    \end{center}
\end{figure}
Similar numerical method can be used to study defect structures in cholesteric LC and blue phase under the different geometric constraints \cite{fukuda2010novel,guo2016cholesteric,darmon2016topological}.

\subsection{Gradient Flow Approach}
Gradient flow is a dynamics driven by a free energy. There are quite a number of works devoted to obtain the defect patterns by solving the gradient flow equation corresponding to different LC systems \cite{fukuda2004interaction,ravnik2009landau,canevari2017order,wang2019order,macdonald2020moving}. For the LdG theory, the corresponding $L^2$-gradient flow equation can be written as
\begin{equation}
    \frac{\partial \Qvec}{\partial t} = - \gamma \frac{\delta \mathcal{F}}{\delta \Qvec},
\end{equation}
where $\gamma > 0$ is dissipative coefficient.
An advantage of gradient flow approach is that it also provides part information of dynamical evolution of defect structure.

On the numerical perspective, the main challenge in developing numerical schemes for gradient flow systems is to maintain the energy dissipation property at the discrete level. During the last a few decades, energy stable numerical schemes for gradient flow systems have been studied extensively, examples include full-implicit scheme \cite{xu2019stability,du2020phase}, convex splitting method \cite{Elliott1993global,eyre1998unconditionally}, stabilization method \cite{shen2010numerical,cai2017stable},
invariant energy quadratization \cite{Yang2016linear,zhao2017novel} and scalar auxiliary variable (SAV) \cite{shen2018scalar,shen2019new}.

In a recent work, \citeasnoun{shen2019new} developed a second-order unconditionally energy stable based on SAV and Crank-Nicolson for the LdG theory, in which the LdG free energy is given by
\begin{equation}
\mathcal{F}[\Qvec] = \mathcal{F}_{\rm b}[\Qvec] + \mathcal{F}_{\rm el} [\Qvec, \nabla \Qvec],
\end{equation}
where
\begin{equation*}
\begin{aligned}
  & \mathcal{F}_{\rm b} = \int f_{\rm b}(\Qvec) = \int \frac{A}{2} \tr \Qvec^2 - \frac{b}{3} \tr \Qvec^3 + \frac{C}{4} (\tr \Qvec^2)^2 \dd \x \\
  & \mathcal{F}_{\rm el} = \int \frac{L_1}{2} |\nabla \Qvec|^2 + \frac{L_2}{2} Q_{ik, i} Q_{jk,j} + \frac{L_3}{2} Q_{jk,i} Q_{ij, k} \dd \x. \\
\end{aligned}
\end{equation*}
Here $C > 0$, $L_1 > 0$ and $L_1 + L_2 + L_3 > 0$, so the total energy is bounded from below.
Due to quartic term $C (\tr \Qvec^2)^2$, there exists $a_1, C_0 > 0$ such that $f_b(\Qvec) - a_1 \tr \Qvec^2 /2 + C_0 > 0$, Let $r$ be the scalar auxiliary variable defined by
\begin{equation}
r(t) = \sqrt{\mathcal{F}_1} = \sqrt{\mathcal{F}_{\rm b} - \int_{\Omega} \frac{a_1}{2} \tr \Qvec^2 + C_0}
\end{equation}
and
\begin{equation}
\mathcal{L} = a_1 \Qvec + \frac{\delta \mathcal{F}_e}{\delta \Qvec},
\end{equation}
then the gradient flow equation can be rewritten as
\begin{equation}
  \begin{cases}
    & \frac{\partial \Qvec}{\partial t} = - \mu \\
    & \mu = \mathcal{L} \Qvec + \frac{r(t)}{\sqrt{\mathcal{F}_1}} \frac{\delta \mathcal{F}_1}{ \delta \Qvec} \\
    & r_t = \frac{1}{2 \sqrt{\mathcal{F}_1}} \int_{\Omega} \frac{\delta \mathcal{F}_1}{\delta \Qvec} : \Qvec_t \dd \x. \\
  \end{cases}
\end{equation}
and the corresponding numerical scheme is given by
\begin{equation} \label{GF_Q}
\begin{aligned}
  & \frac{\Qvec^{n+1}-\Qvec^n}{\Delta t}= -\mu^{n+1/2}, \\
  & \mu^{n+1/2}= \frac{1}{2}(\Qvec^{n+1} + \Qvec^n)
  +\frac{r^{n+1}+r^n}{2\sqrt{\mathcal{F}_1[\bar{\Qvec}^{n+1/2}]}}
  \frac{\delta \mathcal{F}_1}{\delta \Qvec}[\bar{\Qvec}^{n+1/2}], \\
  & r^{n+1}-r^n= \int_\Omega\frac{1}{2\sqrt{\mathcal{F}_1[\bar{\Qvec}^{n+1/2}]}}\left(\frac{\delta \mathcal{F}_1}{\delta \Qvec}[\bar{\Qvec}^{n+1/2}]\right)_{ij}
  (\Qvec_{ij}^{n+1}-\Qvec_{ij}^n) \dd  \bm{x}.
\end{aligned}
\end{equation}

For the LdG free energy with cubic term ($L_4 Q_{kl} Q_{ij, k} Q_{ij, l}$), although the total free energy is not bounded below \cite{ball2010nematic}, \citeasnoun{cai2017stable} constructed an unconditionally stable numerical scheme for 2D $\Qvec$-tensor by a stabilizing technique, They also established unique solvability and convergence of such a scheme. The convergence analysis leads to the well-posedness of the original PDE system for the 2D $\QQ$-tensor model. Several numerical examples are presented to validate and demonstrate the effectiveness of their scheme.

\subsection{Machine learning approach}
Over the last decade, machine learning has made a huge impact on the research areas of materials and soft matter, showing the highly powerful and effective performance by using the techniques of deep learning.
Here, we just take a recent work by \citeasnoun{walters2019machine} as an example of the LC system to demonstrate such trend.
In \citeasnoun{walters2019machine}, authors investigated a problem for identifying the topological defects of rod-like molecules confined in a square box from ``images''.
Unlike conventional images with correlated physical features where supervised machine learning has been successful, these images are coordinated files generated from an off-lattice sampling.
A single-line structure $[l, x_l, y_l, \theta_l]$ is given for each rod-like molecule, where $[x_l, y_l, \theta_l]$ specifies the location coordinates and angular orientation of the molecule labelled $l$, and the labels are not related .
The task is to identify which of the four defect patterns from the off-lattice data, and two types of machine-learning procedures, the feedforward neural network (FNN) and the recurrent neural network (RNN), are considered.
It is shown that FNN is not readily appropriate for studying defect types in this off-lattice model.
However, with a coarse-grained position sorting in the initial data input, referred to as \emph{presorting}, an effective learning can be realized by FNN.
On the contrary, RNN performs exceptionally well in identifying defect states in the absence of presorting.
Mort importantly, \citeasnoun{walters2019machine} pointed out that by dividing the whole image into small cells, an RNN approach with the data in each cell as an input can detect the types and positions of nematic defects in each image instead of naked eyes.

\section{Numerical methods for computing liquid crystal hydrodynamics}
In this section, we review some progress on numerical methods to study LC hydrodynamics, including vector models, tensor models, and molecular models.

\subsection{Vector models}
A NLC flow behaves like a regular liquid with molecules of similar size, and also displays anisotropic properties due to the molecule alignment, which is usually described by the local director field.
The Ericksen--Leslie equations have been applied to describe the flow of NLCs and attracted many theoretical and numerical researches.
According to the macroscopic hydrodynamic theory of NLCs established by \citeasnoun{ericksen1961conservation} and \citeasnoun{leslie1979theory}, \citeasnoun{lin1989nonlinear} proposed a simplied Ericksen--Leslie equations for a NLC flow,
\begin{equation}\label{eqn:eldyn}
\left\{
\begin{array}{l}
\partial_t \boldsymbol u + (\boldsymbol u\cdot \nabla)\boldsymbol u-\nu \Delta \boldsymbol u+\nabla p +\lambda\nabla\cdot((\nabla \boldsymbol d)^\top \nabla \boldsymbol d)= \boldsymbol0,\\
\nabla\cdot \boldsymbol u=0\\
\partial_t \boldsymbol d + (\boldsymbol u\cdot \nabla)\boldsymbol d-\gamma\Delta \boldsymbol d-\gamma |\nabla \boldsymbol d|^2 \boldsymbol d = \boldsymbol 0,\\
|\boldsymbol d|=1.
\end{array}
\right.
\end{equation}
Here $\boldsymbol u$ denotes the solenoidal velocity field, $p$ denotes the pressure, and $\boldsymbol d$ denotes the molecular alignment satisfying a sphere constraint almost everywhere.
The positive parameters $\nu,\lambda,\gamma$ are respectively a fluid viscosity constant, an elastic constant and a relaxation time constant.
The system \eqref{eqn:eldyn} consists of the Navier-Stokes equations coupled with an extra anisotropic stress tensor and a convective harmonic map equation.
Although simple, this system keeps the core of the mathematical structure, such as strong nonlinearities and constraints, as well as the physical structure, such as the anisotropic effect of elasticity on the velocity vector field $\boldsymbol u$, of the original Ericksen-Leslie system.
The first energy equality, which is established under certain boundary conditions, expresses the balance of energy in the system between the kinetic and elastic energies.

\citeasnoun{badia2011overview} provided an excellent overview of the numerical methods for (\ref{eqn:eldyn}).
Since an almost everywhere satisfaction of the sphere constraint restriction is not appropriate at a numerical level, two alternative approaches have been introduced: a penalty method and a saddle-point method.
With a Ginzburg--Landau penalty function $F_\varepsilon(\boldsymbol d)=(|\boldsymbol d|^2-1)^2/4\varepsilon^2$ to enforce the sphere constraint, a penalty formulation is obtained by weakening the constraint as
\begin{equation}\label{eqn:eldyn+gl}
\left\{
\begin{array}{l}
\partial_t \boldsymbol u +(\boldsymbol u\cdot \nabla)\boldsymbol u-\nu \Delta \boldsymbol u+\nabla p +\lambda\nabla\cdot((\nabla \boldsymbol d)^\top \nabla \boldsymbol d)= \boldsymbol0,\\
\nabla\cdot \boldsymbol u=0\\
\partial_t \boldsymbol d +(\boldsymbol u\cdot \nabla)\boldsymbol d+\gamma(f_\varepsilon(\boldsymbol d)-\Delta \boldsymbol d) = \boldsymbol 0,\\
|\boldsymbol d|\leqslant1,
\end{array}
\right.
\end{equation}
where $f_\varepsilon(\boldsymbol d)=\nabla_{\boldsymbol d}F_\varepsilon(\boldsymbol d)$, and the energy estimate for \eqref{eqn:eldyn+gl} was established by \citeasnoun{lin1995nonparabolic}.
Finite element methods of mixed types play an important role when designing numerical approximations for the penalty formulation in order to preserve the intrinsic energy estimate.
Alternatively, with a Lagrange multiplier $q$ to enforce the sphere constraint, a saddle-point formulation reads as
\begin{equation}\label{eqn:eldyn+lm}
\left\{
\begin{array}{l}
\partial_t \boldsymbol u +(\boldsymbol u\cdot \nabla)\boldsymbol u-\nu \Delta \boldsymbol u+\nabla p +\lambda\nabla\cdot((\nabla \boldsymbol d)^\top \nabla \boldsymbol d)= \boldsymbol0,\\
\nabla\cdot \boldsymbol u=0\\
\partial_t \boldsymbol d+ (\boldsymbol u\cdot \nabla)\boldsymbol d+\gamma(q \boldsymbol d-\Delta \boldsymbol d) = \boldsymbol 0,\\
|\boldsymbol d|=1,
\end{array}
\right.
\end{equation}
and the energy estimate for \eqref{eqn:eldyn+lm} was derived by \citeasnoun{badia2011finite}.
These approaches are suitable for their numerical approximation by finite elements, since a discrete version of the restriction is enough to prove the desired energy estimate.
The penalized Ginzburg--Landau problem \eqref{eqn:eldyn+gl} can be stated in a saddle-point framework as well,
\begin{equation}\label{eqn:eldyn+glsp}
\left\{
\begin{array}{l}
\partial_t \boldsymbol u +(\boldsymbol u\cdot \nabla)\boldsymbol u-\nu \Delta \boldsymbol u+\nabla p +\lambda\nabla\cdot((\nabla \boldsymbol d)^\top \nabla \boldsymbol d)= \boldsymbol0,\\
\nabla\cdot \boldsymbol u=0\\
\partial_t \boldsymbol d+ (\boldsymbol u\cdot \nabla)\boldsymbol d+\gamma(q\boldsymbol d-\Delta \boldsymbol d) = \boldsymbol 0,\\
|\boldsymbol d|^2-1=\varepsilon^2 q,
\end{array}
\right.
\end{equation}
which establishes a connection between \eqref{eqn:eldyn+gl} and \eqref{eqn:eldyn+lm}.

The numerical approximation of the Ericksen--Leslie equations is difficult and computationally expensive because of the coupling between the nonlinear terms and the constraint conditions.
\citeasnoun{liu2000approximation} dealt with the approximation of \eqref{eqn:eldyn+gl} for 2D domains, and convergence of finite element approximations is established under appropriate regularity hypotheses with dependence on the penalty parameter $\varepsilon$.
\citeasnoun{du2001fourier} studied a Fourier-spectral method for \eqref{eqn:eldyn+gl} by proving its convergence in a suitable sense and establishing the existence of a global weak solution of the original problem and its uniqueness in the 2D case.
The error estimates exhibit the spectral accuracy of the Fourier-spectral method, and a fully discrete scheme is constructed with a complete stability and error analysis.
\citeasnoun{feng2004analysis} proved a convergence result for a Ginzburg--Landau-type equation where the dependence of $\varepsilon$ is of polynomial order.
\citeasnoun{lin2006simulations} presented one of the simplest time-stepping schemes for the 2D Ginzburg--Landau problem \eqref{eqn:eldyn+gl}, where the space is discretized by $H^1$-conforming finite elements and time is discretized implicitly with respect to the linear terms and semi-implicitly with respect to the nonlinear terms, while the anisotropic stress tensor is fully explicit.
This scheme reduces the computational cost significantly and larger scale numerical simulations are allowed because only a sequence of two decoupled linear problems for the velocity-pressure pair and the director field need to be solved separately at each time step.
\citeasnoun{liu2007dynamics} presented an efficient and accurate numerical scheme for \eqref{eqn:eldyn+gl} in a cylinder domain.
The time discretization is based on a semi-implicit second-order rotational pressure-correction scheme and the Legendre--Galerkin method is used for the space variable.
Annihilation of a hedgehog-antihedgehog pair with different types of transport is simulated numerically.

By introducing an auxiliary variable $\boldsymbol w=\nabla \boldsymbol d$, \citeasnoun{liu2002mixed} avoided using Hermite finite elements for the approximation of the director equation in \eqref{eqn:eldyn+gl} and formulated the director equation in the framework of mixed methods.
They showed how a mixed method may be used to eliminate the need for Hermite finite elements and establish convergence of the method.
\citeasnoun{girault2011mixed} considered another auxiliary variable $\boldsymbol w=-\Delta \boldsymbol d$ to avoid the large number of extra degrees of freedom and the nonlinearity of the numerical schemes, and constructed a fully discrete mixed scheme which is linear, unconditionally stable and convergent towards \eqref{eqn:eldyn+gl}.
With an auxiliary variable $\boldsymbol w=-\Delta \boldsymbol d+f_\varepsilon(\boldsymbol d)$, two finite-element Euler time-stepping schemes that are implicit for the linear terms and semi-implicit for the nonlinear have been developed.
Specifically, \citeasnoun{becker2008finite} proposed a fully implicit approximation to deal with the time integration of $f_\varepsilon$, while \citeasnoun{guillen2013linear} suggested a fully explicit one.
The time step of the explicit scheme must be quite small if the size of $\varepsilon$ is proportional to the space parameter $h$, but their numerical experiences have demonstrated to be optimal.
\citeasnoun{lin2007energy} used convenient conformal $C^0$ finite elements in solving the problem in the weak form, and derived a discrete energy law for a modified midpoint time discretization scheme.
A fixed iterative method is used to solve the resulted nonlinear system so that a matrix-free time evolution may be achieved and velocity and director variables may be solved separately.

\citeasnoun{becker2008finite} presented a finite element scheme directly for the Ericksen--Leslie equations \eqref{eqn:eldyn}, which is a more difficult task than for the Ginzburg--Landau problem \eqref{eqn:eldyn+gl} because the sphere constraint $|\boldsymbol d| = 1$ is difficult to be fulfilled at the discrete level.
\citeasnoun{badia2011finite} developed a linear semi-implicit algorithm which is unconditionally stable for both the Ginzburg--Landau problem and the Ericksen--Leslie problem, and it does not involve nonlinear iterations.
\citeasnoun{guillen2015splitting} proposed a two sub-step viscosity-splitting time scheme for \eqref{eqn:eldyn+lm}, which is a fully decoupled linear scheme from the computational point of view.
The first sub-step couples diffusion and convection terms whereas the second one is concerned with diffusion terms and constraints.
Some numerical experiments in 2D domains are carried out by using only linear finite elements in space, confirming numerically the viability and the convergence of this scheme.

\citeasnoun{walkington2011numerical} considered numerical approximation of the flow of LCs governed by the Ericksen--Leslie equations.
Care should be taken to develop numerical schemes which inherit the Hamiltonian structure of these equations and associated stability properties.
\citeasnoun{zhao2016energy} developed a first-order and a second-order, coupled, energy stable numerical schemes for a modified Ericksen--Leslie hydrodynamic model, which can reduce to \eqref{eqn:eldyn} with the omission of some terms.
Two ways are presented to develop decoupled schemes for the model and the energy stability is shown as well.
Prediction comparisons of the modified model with a reduced model are made, demonstrating quite different but more realistic orientational dynamics in flows of NLCs.
\citeasnoun{chen2016kinematic} investigated the kinematic transports of the defects in the NLC system by numerical experiments with difference schemes and the semi-implicit scheme which is more efficient although it does not preserve discrete energy relation.
The presented development and interaction of the defects are partly consistent with the observation from the experiments.

Simulating the rise of Newtonian drops in a NLC is a computational challenge because of the numerical difficulties in handling moving and deforming interfaces as well as the complex rheology of the NLCs.
The anisotropic rheology of the LC can be represented by the Ericksen--Leslie theory, regularized to permit topological defects.
\citeasnoun{zhou2007rise} simulated the rise of Newtonian drops in a NLC parallel to the far-field molecular orientation by computing the moving interface in a diffuse interface framework.
The numerical results revealed interesting coupling between the flow field and the orientational field surrounding the drop, especially the defect configuration.
For example, drops with planar anchoring on the surface rise faster than those with homeotropic anchoring due to the viscous anisotropy of the nematic.
With both types of anchoring, the drag coefficient of the drop decreases with increasing Ericksen number as the flow-alignment of the
nematic orientation reduces the effective viscosity of the LC.

LC droplets immersed in another fluid matrix or another liquid droplet immersed in LC matrices have many  interesting technological applications \cite{yue2004diffuse}.
Based on an energetic variational approach, \citeasnoun{yue2004diffuse} derived a phase-field theory for immiscible mixtures of NLCs and viscous fluids.
A novel phase transition mechanism is implemented to couple the NLC phase with the viscous fluid phase to arrive at the dissipative hydrodynamic model for incompressible fluid mixtures.
In \citeasnoun{zhou2007rise}, a decoupled, linear scheme for a simplified version of the phase field model, as well as a coupled, nonlinear scheme for the full model, are developed and shown as unconditionally energy stable with consistent discrete dissipative energy laws.
The effectiveness of the presented numerical examples show the effectiveness of the new model and the developed numerical schemes.

A number of models for smectic phase LCs have been developed and studied during the last two decades, and we take one example of smectic LC hydrodynamics here.
\citeasnoun{chen2017second} considered the numerical approximations for solving a particular hydrodynamics coupled smectic-A model developed by \citeasnoun{weinan1997nonlinear}.
The model, which is derived from the variational approach of the modified Oseen--Frank energy, is a highly nonlinear system that couples the incompressible Navier--Stokes equations and a constitutive equation for the layer variable, but appears to be the minimal model of unknowns.
The director field is assumed to be strictly equal to the gradient of the layer and thus the total free energy is reduced to a simplified version with only one order parameter.
Rather than imposing nonconvex constraint directly on the gradient of the layer variable, the free energy is modified by adding a Ginzburg--Landau type penalization potential, which is a commonly used technique in LC theory.
Two linear, second order time marching schemes, which are unconditionally energy stable, are developed with the Invariant Energy Quadratization method for nonlinear terms in the constitutive equation, the projection method for the Navier--Stokes equations, and some subtle implicit-explicit treatments for the convective and stress terms.
Various numerical experiments are presented to demonstrate the stability and the accuracy of the numerical schemes in simulating the dynamics under shear flow and the magnetic field.

\subsection{Tensor models}
In recent years, more complex problems of modelling the interaction of flow and molecular orientation in a LCs are studied.
The hydrodynamic $\mathbf{Q}$-tensor model has been applied to the flows of LCs and LC polymers.
In particular, the nondimensionalized hydrodynamic $\mathbf{Q}$-tensor model of NLCs
\begin{equation}\label{eqn:qdyn}
\left\{
\begin{array}{l}
\partial_t \mathbf{Q}+\boldsymbol{u}\cdot \nabla \mathbf{Q}-\mathbf{S}(\nabla\boldsymbol{u},\mathbf{Q})=M \mathbf{H},\\
\nabla \cdot\boldsymbol{u}=0,\\
\partial_t \boldsymbol{u}+\boldsymbol{u}\cdot\nabla \boldsymbol{u}=-\nabla p +\eta \Delta \boldsymbol{u}+\nabla\cdot \Sigma-\mathbf{H}\nabla\mathbf{Q}.
\end{array}
\right.
\end{equation}
can be derived from a variational approach together with the generalized Onsager principle.
The molecular field $\mathbf{H}$, which provides the driving motion, is related to the derivative of the free energy.
More details are referred to \citeasnoun{BerisEdwards}, \citeasnoun{wang2002hydrodynamic}.
\citeasnoun{zhao2016semi} developed the first-order and second-order coupled energy stable numerical schemes for the $\mathbf{Q}$-tensor based hydrodynamic model of NLC flows.
The first-order coupled scheme is extended to a decoupled scheme, which is energy stable as well.
With the fully coupled schemes implemented in 2D space and time, defect dynamics for flow of NLCs in a channel is studied.
The developed methodology also provides a paradigm for developing energy-stable schemes for general hydrodynamic models of complex fluids with an energy dissipation law.
Furthermore, \citeasnoun{zhao2017novel} developed a second-order semi-discrete scheme in time, which is linear and unconditionally energy stable at the semi-discrete level, to solve the governing system of equations by following the novel `energy quadratization' strategy.
Several numerical examples are presented to demonstrate the usefulness of the model and the effectiveness of the numerical
scheme in simulating defect dynamics in flows of LCs.
\citeasnoun{denniston2001lattice} described a lattice Boltzmann algorithm to simulate LC hydrodynamics in terms of a tensor order parameter, including a molecular field that provides the driving motion.
A lattice Boltzmann algorithm was described in detail to simulate LC hydrodynamics.
Backflow effects and the hydrodynamics of topological defects are naturally included in the simulations, as are non-Newtonian flow properties such as shear thinning and shear banding.

\citeasnoun{ramage2016computational} considered the NLC in a spatially inhomogeneous flow with a second-rank alignment tensor, whose evolution is determined by two coupled equations: a generalized Navier--Stokes equation for the flow, and a convection-diffusion type equation for the alignments.
A specific model with three viscosity coefficients allows the contribution of the orientation to the viscous stress to be cast in the form of an orientation-dependent force, which effectively decouples the flow and orientation to circumvent the fully coupled problem.
A time-discretized strategy for solving the flow-orientation problem is illustrated using the Stokes flow in a lid-driven cavity.

Dispersing colloidal particles into LCs provides an approach to build a novel class of composite materials with potential
applications.
Many researches have been devoted to the hydrodynamic equations governing their dynamical evolution to understand the physics and dynamical properties of such colloid LC composites.
\citeasnoun{foffano2014dynamics} provided an overview of theoretical understanding of the hydrodynamic properties of such materials from computer simulations, including a single particle and two particles forming a dimer and dispersion. A number of open questions in the dynamics of colloid LC composites are also raised.

\subsection{Molecular models}
For the liquid crystalline polymer (LCP) model, the Doi--Hess theory by \citeasnoun{doi1988theory} and \citeasnoun{hess1976fokker} is the simplest and the most studied model.
Many interesting dynamics have been found in rod-like nematic LCPs in a shear flow.
\citeasnoun{li2004analysis} considered the stochastic model of concentrated LCPs in the plane Couette flow, where the flow is one dimensional while the configuration variable of the rod is restricted to the circle, and presented the local existence and uniqueness theorem for the solution to the coupled fluid-polymer system.
\citeasnoun{ji2008kinetic} considered the extended Doi model for nematic LCPs in the planar shear flow and studied the formation
of microstructure and the dynamics of defects.
The Fokker--Plank equation is discretized with the spherical harmonic spectral method and multiple flow modes are replicated in the simulations.
A comparison among complete closure models, Bingham closure models and kinetic models is given in \citeasnoun{ji2008kinetic} as well.

\citeasnoun{yu2007kinetic} presented a new kinetic hydrodynamic coupled model for the dilute LCP solution for inhomogeneous flow, which accounts for translational diffusion and density variation.
The coupled kinetic hydrodynamic model is a combination of an extended Doi kinetic theory for rigid rod-like molecules and the Navier--Stokes equation for incompressible flow.
They studied the microstructure formation and defects dynamics arising in LCPs in plane shear flow with mass conservation of LCPs in the whole flow region.
The LCP molecular director is restricted in the shear plane, and LCP molecules are ensured anchoring at the boundary by an additional boundary potential.
Plane Couette flow and Poiseuille flow were studied with a second-order difference scheme and the fourth-order Runge--Kutta method
for time-stepping.
Their numerical results in plane Couette flow predicted seven in-plane flow modes, four of which have been reported by \citeasnoun{rey1998recent}.
In plane Poiseuille flow, the micro-morph is quasi-periodic in time when flow viscosity and molecular elasticity are comparable.
Different local states, such as flow-aligning, tumbling or wagging, arise in different flow regions.
To describe the microstructures and defect dynamics of LCP solutions, \citeasnoun{yu2010nonhomogeneous} later purposed a general nonhomogeneous extension of the Doi's kinetic theory with translational diffusion and nonlocal potential.
A reduced second-order moment model for isotropic long-range elasticity was obtained as a decent tool for numerical simulations of defect dynamics and texture evolution.
Their numerical results of in-plane rotational case show that this model qualitatively predicts complicated nonhomogeneous director dynamics under moderate nematic potential strength, and the translational diffusion plays an important role in defect dynamics.

\section{Numerical methods for computing transition pathways and solution landscape of liquid crystals}
NLCs often exhibit a number of ordered phases. When an ordered phase is metastable, phase transition proceeds via nucleation and growth under thermal fluctuations or external perturbations. Such transition is often called as a rare event, which can be characterized by a long waiting period around one local metastable state and followed by a sudden jump over the transition state, to another stable state \cite{hanggi1990reaction}. Because the transition state is an unstable critical point of the LC free energy, i.e., an index-1 saddle point with a single unstable direction that connects two local minima, it is of great challenge to accurately compute the transition states and the transition pathways.

In this section, we will first review the numerical methods to compute the transition states and transition pathways. Then we will introduce a novel solution landscape, which describes a pathway map consisting of all stationary points (i.e. extrema and saddle points) and their connections.

\subsection{Transition pathways}

There are two popular approaches to find the saddles points and transition pathways. One is the class of surface-walking methods for finding saddle points starting from a single state, such as gentlest ascent dynamics \cite{weinan2011gentlest} and dimer-type method \cite{henkelman1999dimer,zhang2012shrinking}. The other approach is categorised as path-finding methods for computing minimal energy path (MEP) that involves two end states, including nudged elastic band method \cite{jonsson1998nudged} and string method\cite{weinan2002string}.

In the case when the object of interest is the most probable transition path between metastable/stable states of the smooth potential
energy, it is known that for overdamped Langevin dynamics the most probable path for the transition is the MEP.
The string method serves this purpose well, which was first proposed by E, Ren and Vanden-Eijnden \cite{weinan2002string} and has some variants \cite{weinan2007simplified,du2009constrained,ren2013climbing,samanta2013optimization}. The string method proceeds by evolving a string $\phi$ with intrinsic parametrization in the configuration space by using the steepest decent dynamics to the MEP:
\begin{equation}
    \phi_t=- \nabla F^{\perp}(\phi)+\lambda\mathbf{\tau},
        \label{string}
\end{equation}
where $\nabla F^{\perp}(\phi)$ is the component of $\nabla F$ normal to $\phi$, $\mathbf{\tau}$ denotes the unit tangent vector to $\phi$, and $\lambda$ is the Lagrange multiplier to impose the equal arc-length constraint.

In the numerical algorithm, the string $\phi$ is discretized to $N$ nodes $\{\mathbf{Q}_i^n, i = 1,..,N\}$, where $n$ represents the iteration step.
The string method adopts a time splitting scheme via the following two-step procedure:\\
{\it Step 1: string evolution:} The current path $\{\mathbf{Q}_i^n, i = 1,..,N\}$ is updated by following the gradient decent direction :
\begin{equation} \label{eq:step1}
    \bar{\mathbf{Q}}^{n+1}_i = \mathbf{Q}^{n}_i-\alpha_i\nabla F^{\perp}[\mathbf{Q}^{n}_i],\quad i=1,2, \cdots, N,
\end{equation}
where $\alpha_i$ is a stepsize that is chosen as a constant or determined by the linesearch.

{\it Step 2: string reparametrization:} Redistributes the nodes $\bar{\mathbf{Q}}_i^{n+1}$ to obtain a new path $\mathbf{Q}^{n+1}_i$ by enforcing the equal arc length parametrization. One can calculate the arc length corresponding to the current nodes $\bar{\mathbf{Q}}_i^{n+1}$  and then use a linear or cubic spline interpolation to find $\mathbf{Q}^{n+1}_i$ by the equal arc length.

Parametrization by equal arc length gives a good accuracy for the MEP, but not for the transition state. To achieve the better estimation of the transition state and the energy barrier, one way is to use energy-weighted arc length so that finer resolution can be achieved around the transition state \cite{weinan2007simplified}. If the final minima is far from the transition state and too many nodes are needed for the whole MEP, the free-end nudged elastic band method can be used to obtain subtle depiction around the transition state by allowing the last node not to be a minimum \cite{zhu2007interfacial}.

To improve the accuracy of both MEP and transition state without using too many nodes, in \cite{han2019transition}, the authors proposed a multi-scale scheme of the string method by applying a global coarse string and a local fine string. The global coarse string is used to obtain the whole MEP by connecting two minima, and then the local fine string is utilized to compute the local MEP by choosing two nodes with the highest energies on the global coarse string as the two ends for the case of single energy barrier. Both strings follow the same two-step procedure described above to make the implementation straightforward, and better estimation of the transition state can be achieved by the local fine string. They implement the multi-scale string method on numerical simulations of the transition pathway in three-dimensional cylinder between PR and PP with and without axial invariance, between ER and PP or ERRD for different values of temperature and radius of cylinder.

\begin{figure}
    \centering
        \includegraphics[width=0.9\columnwidth]{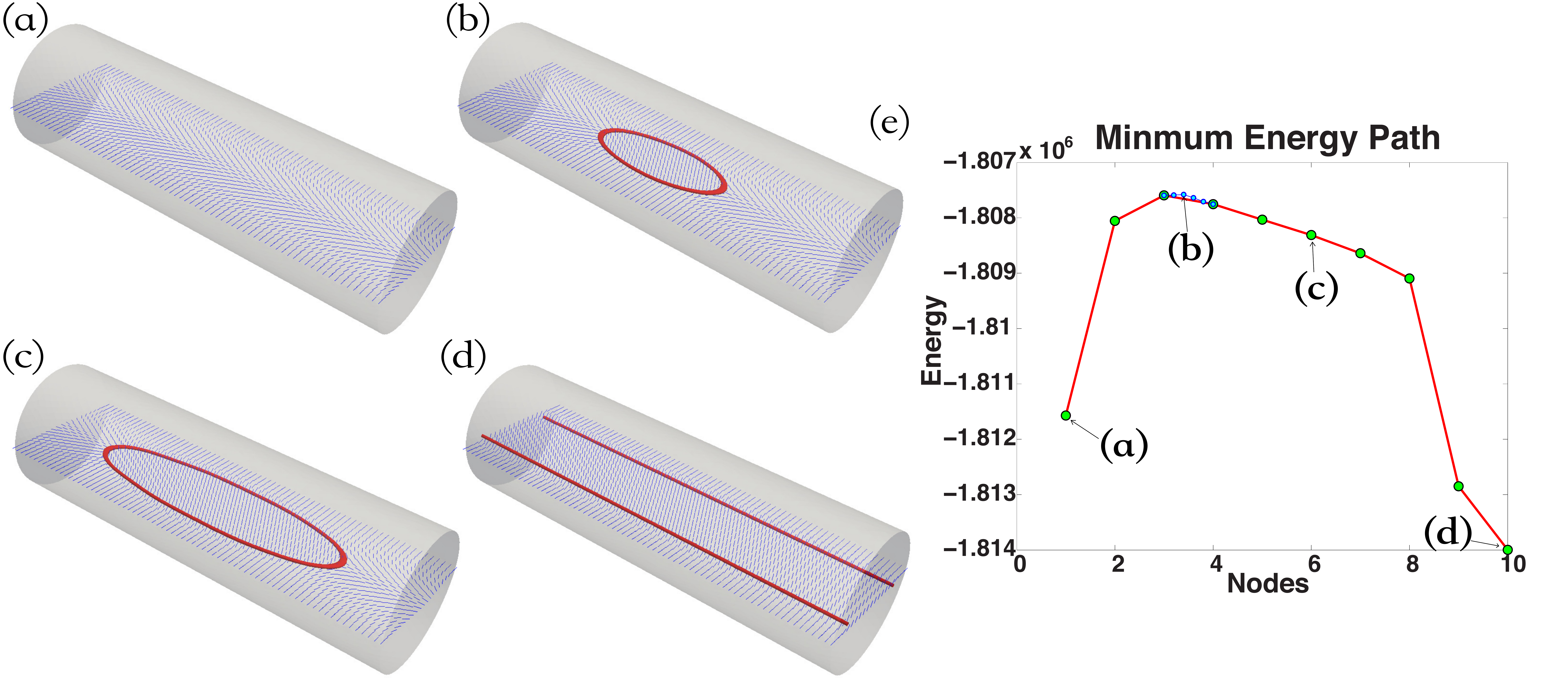}
    \caption{(Credit: \protect\cite{han2019transition}) Transition pathway from ER to PP. (a-d) four states along the MEP (e), corresponding to the ER (a), the transition state (b), an intermediate state (the 6th node) (c),  and the PP (d), respectively. In the MEP(e), the red line with green beads is the global coarse string and the blue line with blue beads is the local fine string.}
    \label{ERPP}
\end{figure}

For example, the transition pathway between ER and PP for fixed temperature and radius of cylinder is shown in Fig. \ref{ERPP}. A pair of +1/2 disclination lines starts to form in the middle of the cylinder. These two disclination lines appear to join their both ends, hence forming a disclination ring (Fig. \ref{ERPP}(b)). The width of the ring is increasing when less than the distance between defects of stable 2D PP layer. Finally, the ring is broken by the top and bottom boundary and the line defect straightens to PP line defect configuration (Fig. \ref{ERPP}(d)). The local fine string in Fig. \ref{ERPP} finds a more accurate transition state than the global coarse string.

\subsection{Saddle points}
The classical transition state theory describes a sufficiently accurate transition process for the systems with smooth
energy landscapes. On the energy landscape, the transition state is a saddle point with the lowest energy that connects two
neighbouring local minima, i.e., index-1 saddle point. As long as the transition state is located, one can calculate the MEP using gradient flow from the transition state perturbed along both sides of its unstable direction.

Extensive surface walking methods have been proposed to search saddle points. A key character of such methods is to
perform a systematic search for a saddle point starting from a given initial state, without knowing the final states. In this subsection, we focus
on a typical class of surface walking methods, the so-called minimum mode following methods, where only the lowest
eigenvalue and the corresponding eigenvector of the Hessian are needed and subsequently used together with the energy
gradient (often referred as the force) to compute transition states. Besides the eigenvector-following method by \citeasnoun{cerjan1981finding}, the trajectory following algorithm by \citeasnoun{vincent1992trajectory}, the activation-relaxation technique by \citeasnoun{mousseau1998traveling}, and the step and slide method by \citeasnoun{miron2001step} are also some samplers of surface-walking methods, and we refer to \citeasnoun{vanden2010transition} and \citeasnoun{zhang2016recent} as some excellent reviews.

In \citeasnoun{weinan2011gentlest}, the authors proposed the gentlest ascent dynamics, which is a dynamical formulation of the gentlest ascent method by \citeasnoun{crippen1971minimization}. It refers to the following dynamical system
\begin{equation}\label{eqn:gad}
\left\{
\begin{aligned}
\dot{x}&=-\nabla E(x)+2\dfrac{\langle\nabla E(x),x\rangle}{\langle v,v\rangle}v,\\
\dot{v}&=-\nabla^2 E(x)v+\dfrac{\langle v, \nabla^2 E(x) v\rangle}{\langle v,v\rangle}v,
\end{aligned}
\right.
\end{equation}
where $E(\cdot)$ repersents the energy functional.
Besides the position variable $x$, another variable $v$ represents the unstable direction.
The dynamics \eqref{eqn:gad} can drive $x$ climb out of the basin of attraction by following the ascent direction $v$, which is the eigenvector that corresponds to the smallest eigenvalue of its Hessian, to find the saddle point. It was proved in \cite{weinan2011gentlest} that the stable fixed points of this dynamical system are precisely the index-1 saddle points.

To avoid the calculation of the Hessian, a dimer method was developed using only first derivatives by \citeasnoun{henkelman1999dimer}.
Specifically, a dimer with two points is placed at the current position to compute the Hessian-vector multiplication with a finite difference scheme. The dimer method proceeds by alternately performing the rotation step for finding the lowest eigenmode and the translation step by using the modified force with either the steepest descent algorithm or the conjugate gradient method to search the saddle point.
Later, \citeasnoun{zhang2012shrinking} proposed a shrinking dimer dynamics by introducing a dynamic reduction of the dimer length, and proved linear stability and convergence results for such dynamics.
To accelerate the convergence and further improve the efficient of the dimer-type methods, \citeasnoun{zhang2016optimization} developed an optimization-based shrinking dimer (OSD) method by establishing a minimax optimization for the saddle searching problem and finding the unstable direction in an optimization framework. All these methods can be applied to the LC models for identifying the transition states accurately \cite{han2019transition}.

Besides the investigations of the minima and index-1 saddle points on the energy landscape of NLCs in confinement, there is substantial recent interest in the high-index saddle points of the LdG free energy, which are stationary points of the NLC free energy.
The Morse index of a saddle point is the maximal dimension of a subspace on which its Hessian is negative definite, which is equal to the number of negative eigenvalues of its Hessian \cite{milnor1969morse}.
For example, a stable stationary point has index $0$ and a transition state is an index-$1$ saddle point.
In recent years, powerful numerical algorithms have been developed to find multiple solutions of nonlinear equations, including the gradient-square-minimization method \cite{Angelani2010saddles}, the minimax method \cite{li2001minimax}, the deflation technique \cite{farrell2015deflation}, and the homotopy method \cite{mehta2011finding,hao2014bootstrapping}.
Despite substantial progress in this direction, it is still a numerical challenge to systematically find saddle point solutions, especially those with high indices, of nonlinear systems of partial differential equations such as the Euler-Lagrange equation.
In \citeasnoun{yin2019high}, the authors proposed a high-index saddle dynamics (HiSD) to efficiently compute the saddle points of any index (including minima) for the LdG energy \cite{han2020SL}.

For a non-degenerate index-$k$ saddle point $\hat{\mathbf{x}}$, the Hessian $\mathbb{H}(\mathbf{x})=\nabla^2 E(\mathbf{x})$ at $\hat{\mathbf{x}}$ has exactly $k$ negative eigenvalues $\hat{\lambda}_1\leqslant \cdots \leqslant\hat{\lambda}_k$ with corresponding unit eigenvectors $\hat{\mathbf{v}}_1,\ldots , \hat{\mathbf{v}}_k$ satisfying $\big\langle\hat{\mathbf{v}}_j, \hat{\mathbf{v}}_i\big\rangle = \delta_{ij}$, $1\leqslant i, j \leqslant k$.
Define a $k$-dimensional subspace $\hat{\mathcal{V}}=\mathrm{span}\big\{\hat{\mathbf{v}}_1,\ldots, \hat{\mathbf{v}}_k\big\}$, then $\hat {\mathbf{x}}$ is a local maximum on a $k$-dimensional linear manifold $\hat{\mathbf{x}}+\hat{\mathcal{V}}$ and a local minimum on $\hat{\mathbf{x}}+\hat{\mathcal{V}}^\perp$, where $\hat{\mathcal{V}}^\perp$ is the orthogonal complement space of $\hat{\mathcal{V}}$.

The HiSD dynamics for a $k$-saddle ($k$-HiSD) is given as follows:
\begin{equation}\label{dynamics}
\left\{
\begin{aligned}
    \beta^{-1}\dot{\mathbf{x}}   & =- \left(\mathbf{I}-2\sum_{j=1}^{k}\mathbf{v}_j \mathbf{v}_j^\top\right)\nabla E(\mathbf{x}), \\
    \gamma^{-1}\dot{\mathbf{v}}_i & = -\left(\mathbf{I}-\mathbf{v}_i\mathbf{v}_i^\top-2\sum_{j=1}^{i-1}\mathbf{v}_j \mathbf{v}_j^\top\right)\mathbb{H}(\mathbf{x})\mathbf{v}_i,\; i=1,\ldots,k,\\
\end{aligned}
\right.
 \end{equation}
where the state variable $\mathbf{x}$ and $k$ direction variables $\mathbf{v}_i$ are coupled, $\mathbf{I}$ is the identity operator and $\beta,\gamma>0$ are relaxation parameters.
The $k$-HiSD dynamics \eqref{dynamics} is coupled with an initial condition:
\begin{equation}\label{othonormal}
\mathbf{x}(0)= \mathbf{x}^0\in\mathbb{R}^n,\quad \mathbf{v}_i(0)=\mathbf{v}_i^0\in\mathbb{R}^n,  i=1, \ldots, k,
\end{equation}
where $\mathbf{v}_1^0,\ldots,\mathbf{v}_k^0$ satisfy the orthonormal condition $\left\langle\mathbf{v}_i^0,\mathbf{v}_j^0\right\rangle = \delta_{ij}$, $i,j=1,2,\ldots,k$.
The first equation in \eqref{dynamics} describes a transformed gradient flow, which allows $\mathbf{x}$ to move along ascent directions on the subspace $\hat{\mathcal{V}}$ and descent directions on the subspace  $\hat{\mathcal{V}}^\perp$.
The second equation in \eqref{dynamics} is used to search for an orthonormal basis of $\hat{\mathcal{V}}$. Because the Hessian $\mathbb{H}(\mathbf{x})$ is self-adjoint in a gradient system, we can simply take $\mathbf{v}_i$ as a unit eigenvector corresponding to the $i$th smallest eigenvalue of $\mathbb{H}(\mathbf{x})$, which can be obtained from a constrained optimization problem,
\begin{equation}\label{minvi}
\min_{\mathbf{v}_i\in\mathbb{R}^n} \quad\langle \mathbb{H}(\mathbf{x})\mathbf{v}_i, \mathbf{v}_i\rangle,\qquad
\mathrm{s.t.} \quad\langle \mathbf{v}_j,\mathbf{v}_i\rangle=\delta_{ij},\quad j=1,2,\ldots,i.
\end{equation}
Then we minimize the $k$ Rayleigh quotients \eqref{minvi} simultaneously by solving the second equation in \eqref{dynamics}.
Furthermore, one can apply the locally optimal block preconditioned conjugate gradient (LOBPCG) method \cite{knyazev2001toward} to calculate the smallest $k$ eigenvalues and $\mathbf{v}_1,...,\mathbf{v}_k$.

\subsection{Solution landscape of liquid crystals}
Solution landscape is a novel concept, which is defined as a pathway map consisting of all stationary points and their connections, which presents a hierarchy structure to show how minima are connected with index-1 saddle points, and how lower-index saddle points are connected to higher-index ones, finally to a parent state, i.e., the highest-index saddle point \cite{yin2020construction}. Solution landscape not only advances the understanding of the relationships between the minima and the transition states on the energy landscape, but also provides a full picture of the entire family of stationary points in both gradient systems and dynamical systems \cite{yu2020GHiOSD}.

To construct a solution landscape, one needs to apply HiSD and follow two algorithms,
downward search to find all connected low-index saddles from a high-index saddle and
upward search to find a connected high-index saddle, which drive the entire search to navigate up and down on the energy landscape \cite{yin2020construction}.

{\it Downward search algorithm}: Given an index-$m$ saddle point $\hat{\mathbf{x}}$ and $m$ unit eigenvectors $\hat{\mathbf{v}}_1,\ldots,\hat{\mathbf{v}}_m$ corresponding to the $m$ negative eigenvalues $\hat{\lambda}_1\leqslant \ldots\leqslant\hat{\lambda}_m$ of the Hessian $\mathbb{H}(\hat{\mathbf{x}})$ respectively, we search for a lower index-$k$ ($k<m$) saddle point using HiSD dynamics \eqref{dynamics}.
For the initial condition, we choose $\mathbf{x}(0) = \hat{\mathbf{x}}\pm\varepsilon \mathbf{u}$  for $\mathbf{x}$, where
we perturb $\hat{\mathbf{x}}$ along a direction $\mathbf{u}$ with a small $\varepsilon$.
The direction $\mathbf{u}$ can be chosen as $\hat{\mathbf{v}}_i$ or a linear combination of $(m-k)$ vectors in the set of unstable directions $\{\hat{\mathbf{v}}_{k+1},\ldots,\hat{\mathbf{v}}_{m}\}$, whose negative eigenvalues have the smallest magnitudes.
The other $k$ eigenvectors $\hat{\mathbf{v}}_1,\ldots,\hat{\mathbf{v}}_k$ are the initial unstable directions $\mathbf{v}_i(0)$.
A typical choice of initial conditions in a downward search is $(\hat{\mathbf{x}}\pm\varepsilon\hat{\mathbf{v}}_{k+1}, \hat{\mathbf{v}}_1,\ldots,\hat{\mathbf{v}}_k)$.
Normally, a pair of index-$k$ saddles can be found, corresponding to the $\pm$ sign of the initial searching direction.

{\it Upward search algorithm}: We can also search a high index-$k$ saddle from a low index-$m$ saddle $\hat{\mathbf{x}}$ ($m<k$) by using the HiSD method.
Starting at an index-$m$ saddle $\hat{\mathbf{x}}$, to search for a high index-$k$ saddle,
$(k-m)$ other unit eigenvectors $\hat{\mathbf{v}}_{m+1},\ldots,\hat{\mathbf{v}}_k$ corresponding to the smallest $k-m$ positive eigenvalues of the Hessian $\mathbb{H}(\hat{\mathbf{x}})$ are needed.
The initial state $\mathbf{x}(0)$ is set as $\hat{\mathbf{x}}\pm\varepsilon\mathbf{u}$ where $\mathbf{u}$ is a linear combination of $\{\hat{\mathbf{v}}_{m+1},\ldots,\hat{\mathbf{v}}_k\}$, and a typical initial condition for $k$-HiSD in an upward search is $(\hat{\mathbf{x}}\pm\varepsilon\hat{\mathbf{v}}_{k}, \hat{\mathbf{v}}_1,\ldots,\hat{\mathbf{v}}_k)$.

Each downward or upward search represents a pseudodynamics between a pair of saddle points, which presents valuable insights into transition pathways between stable and unstable solutions and the corresponding energy barriers.
By repeating downward search and upward search, we are able to systematically find saddle points of all indices and uncover the complex connectivity of the solution landscape.

In \citeasnoun{yin2020construction}, the authors demonstrated the success of the solution landscape and applied the LdG model to construct the pathway maps of 2D NLCs confined in a square domain with tangent boundary conditions.
One of the technical advantages of this method is to produce the entire family tree under a parent state.

\begin{figure}
\begin{center}
\includegraphics[width=\columnwidth]{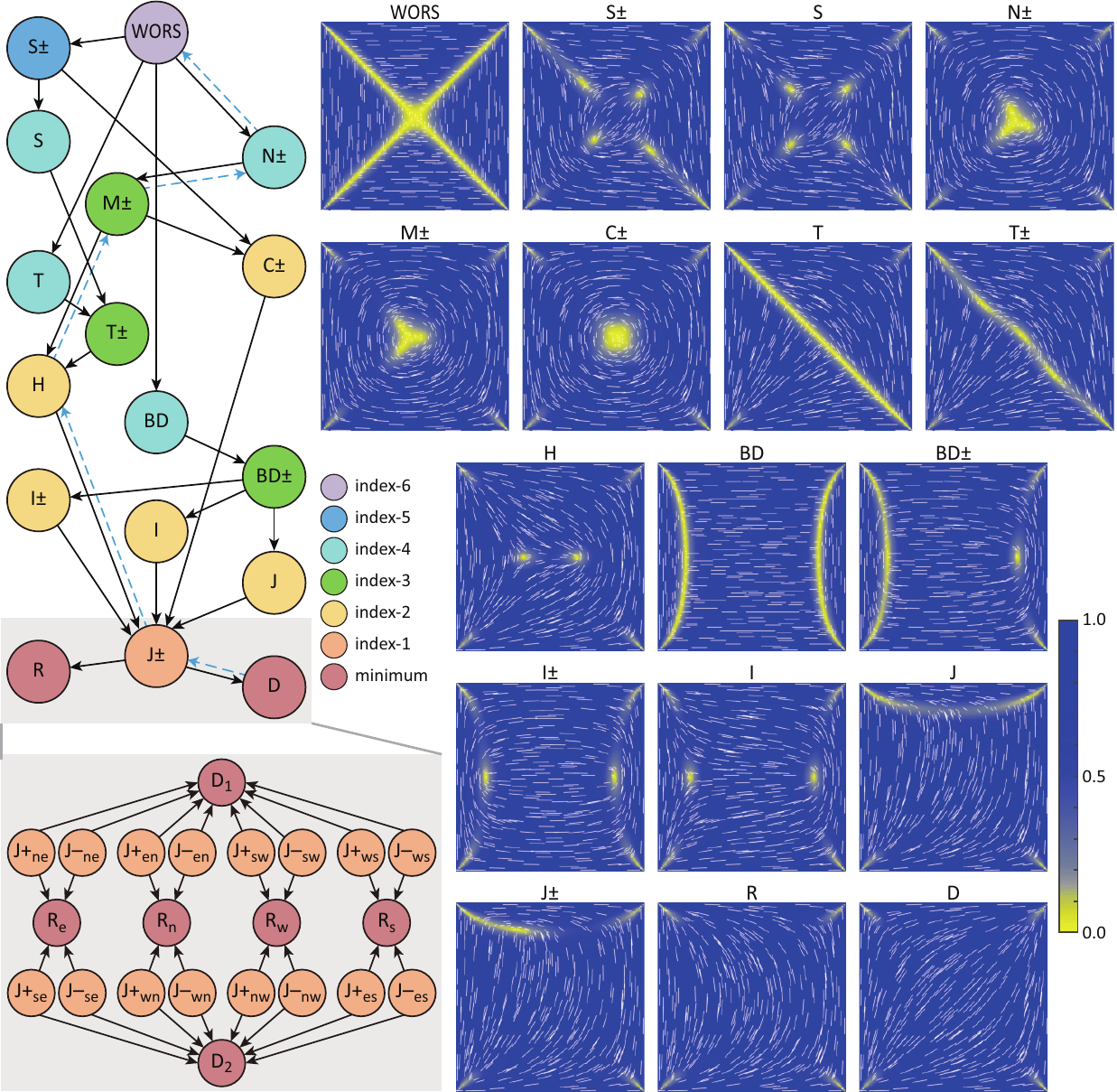}
\caption{(Credit: \protect\cite{yin2020construction}) Solution landscape of NLC confined on a square at the domain size $\lambda^2 = 50$.}
\label{WORS}
\end{center}
\end{figure}

For small nano-scale square domains, it is known that the well order reconstruction solution (WORS) is the unique solution with isotropic diagonals and is the global energy minimum in this asymptotic regime \cite{kralj2014order,canevari2017order}. Moreover, the existence of WORS can be proved in arbitrary sized nematic wells \cite{majumdar2016front}. By studying the second variation of LdG free energy numerically, \citeasnoun{wang2019order} shows that the WORS is a high-index saddle point in a large well.
In Figure \ref{comparison} (a), the Morse index of the WORS increases with the domain size and the WORS is always the parent state in the solution landscapes on a square domain.
Intuitively, this is because the length of the diagonal defect lines increases as the domain size increases, and thus the WORS has an increasing number of unstable directions and an increasing Morse index with the increasing square edge length.
Using WORS as the parent state (the highest-index saddle) and the HiSD method \cite{yin2019high}, \citeasnoun{yin2020construction} constructed a solution landscape on the square domain in Figure \ref{WORS} and reported novel saddle point solutions with multiple interior defects, such as $N\pm$, $M\pm$, $S\pm$, and $T\pm$, that were not reported in the existing literature . Besides these new solutions, solution landscape provides rich information of dynamical pathways. For example, two movies, a dynamic downward pathway sequence and an upward pathway sequence in Figure \ref{WORS}, are shown in \citeasnoun{yin2020construction} to demonstrate the hidden physical processes and transitions on the complicated energy landscape of NLCs.

\begin{figure}
\begin{center}
\includegraphics[width=\columnwidth]{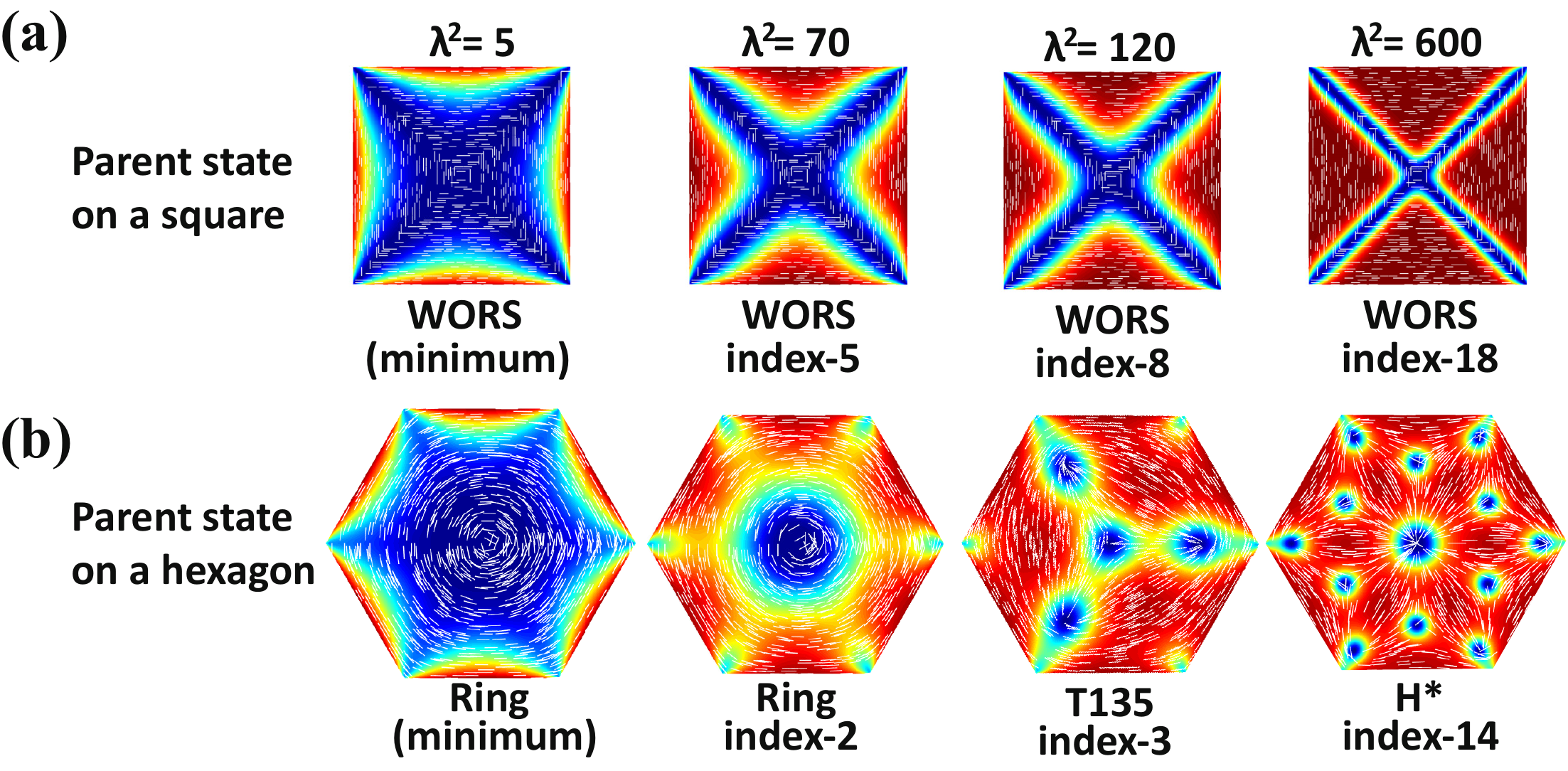}
\caption{(Credit: \protect\cite{han2020SL}) Comparison of the parent states of the solution landscapes on the square (a) and the hexagon (b).
The domain size $\lambda^2 = 5, 70, 120$, and $600$, respectively.}
\label{comparison}
\end{center}
\end{figure}

Solution landscape of the NLC system can be affected by many factors, such as material properties, external fields, temperature, boundary conditions, domain size and shape, etc.
Compared to a square domain, hexagon is a generic regular polygon with an even number of sides.
In \citeasnoun{han2020SL}, the authors investigated the solution landscapes of NLCs confined in a hexagon with tangent boundary conditions. The Ring solution, which is the analogue of the WORS, is a minimizer on a hexagon when the parameter of domain size $\lambda$ is small enough. However, the Ring solution becomes and remains as an index-$2$ saddle-point solution for large $\lambda$, i.e., the Morse index of the Ring solution does not increase with $\lambda$.
By using the HiSD method, the parent states in the solution landscapes of NLCs on a hexagon are found to be the Ring solution, index-$3$ T135 solution, and index-$14$ H* solution when $\lambda^2 = 70, 120$, and $600$ in Figure \ref{comparison} (b) respectively, where T135 and H* solutions emerge through saddle-node bifurcations.

\begin{figure}
\begin{center}
\includegraphics[width=\columnwidth]{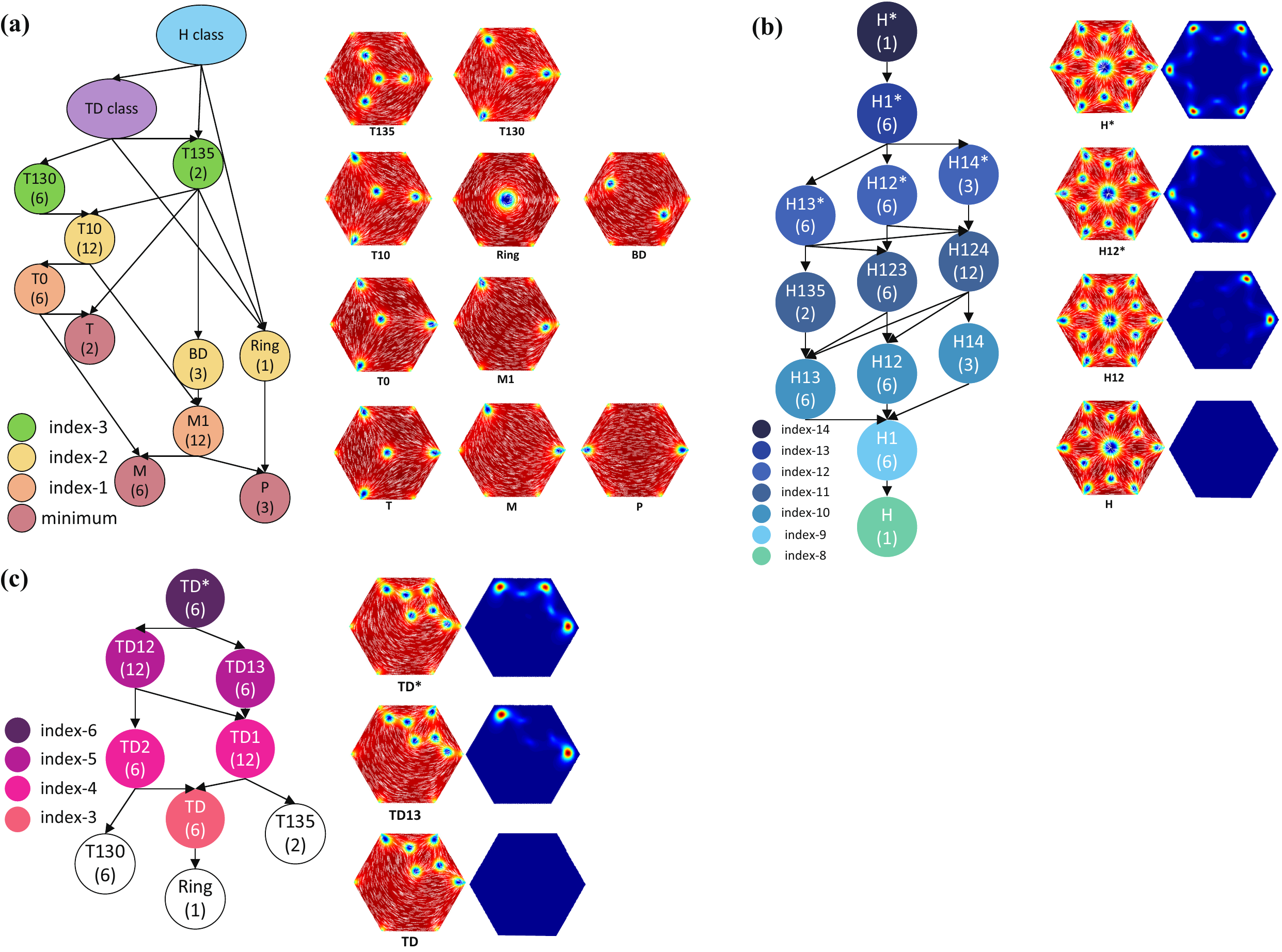}
\caption{(Credit: \protect\cite{han2020SL}) (a) Solution landscape on a hexagon at $\lambda^2 = 600$. The number in the parentheses indicates the number of solutions without taking symmetry into account.
The height of a node approximately corresponds to its energy.
(b) Solution landscape of the H class, and the configurations of sample solutions and the corresponding plots of $|\mathbf{Q}-\mathbf{Q}^H|$, where $\mathbf{Q}$ is the sample solution in H class, and $\mathbf{Q}^H$ is the index-$8$ H solution.
(c) Solution landscape of the TD class, and the configurations of sample solutions and the corresponding plots of $|\mathbf{P}-\mathbf{P}^{TD}|$, where $\mathbf{P}$ is the sample solution in TD class, and $\mathbf{P}^{TD}$ is the index-$3$ TD solution.}
\label{figure:H}
\end{center}
\end{figure}

In the NLC system confined on a hexagon, the solution landscape is very complicated at $\lambda^2=600$ in Figure \ref{figure:H}(a).
There are three notable numerical findings, including a new stable T solution with an interior $-1/2$ defect, new H and TD classes of saddle point solutions with high symmetry and high indices, and new saddle points with asymmetric defect locations.
Novel H-class solutions with interior point defects have Morse indices ranging from $8$ to $14$, and the connectivity of these solutions is shown in Figure \ref{figure:H}(b), as well as the corresponding configurations and their defect profiles.
The parent state is the index-$14$ H* saddle point solution connecting to the lowest index-$8$ saddle point solution, labelled as H.
The H-class solutions look very similar to each other at first glance and their subtle differences are illustrated by plotting $|\mathbf{Q}-\mathbf{Q}^H|$, where $\mathbf{Q}$ is a critical point solution in H-class and $\mathbf{Q}^H$ is the index-$8$ H solution.
The differences concentrate on the vertices with conspicuous red or white points in the dark blue background (Figure \ref{figure:H}(b)).
The vertex with a pinned +1/3 defect and a splay profile of direction is referred to as a \textit{splay-like vertex}.
The vertex with a $-1/6$ defect and a bend profile of direction is referred to as a \textit{bend-like vertex}.
These conspicuous points are localised near or at the bend-like vertices.
The TD-class solutions, in which TD is an abbreviation for ``triangle double", appear to be a superposition of two Ring solutions on a regular triangle, with two interior $-1/2$ point defects and an interior $+1/2$ point defect.
All saddle points in this class have three defective vertices, either bend-like or splay-like, and the connectivity of this class is shown in Figure \ref{figure:H}(c).
The lowest-index saddle point solution in this class is the index-$3$ TD solution with no bend-like vertices.

\begin{figure}
\begin{center}
\includegraphics[width=\columnwidth]{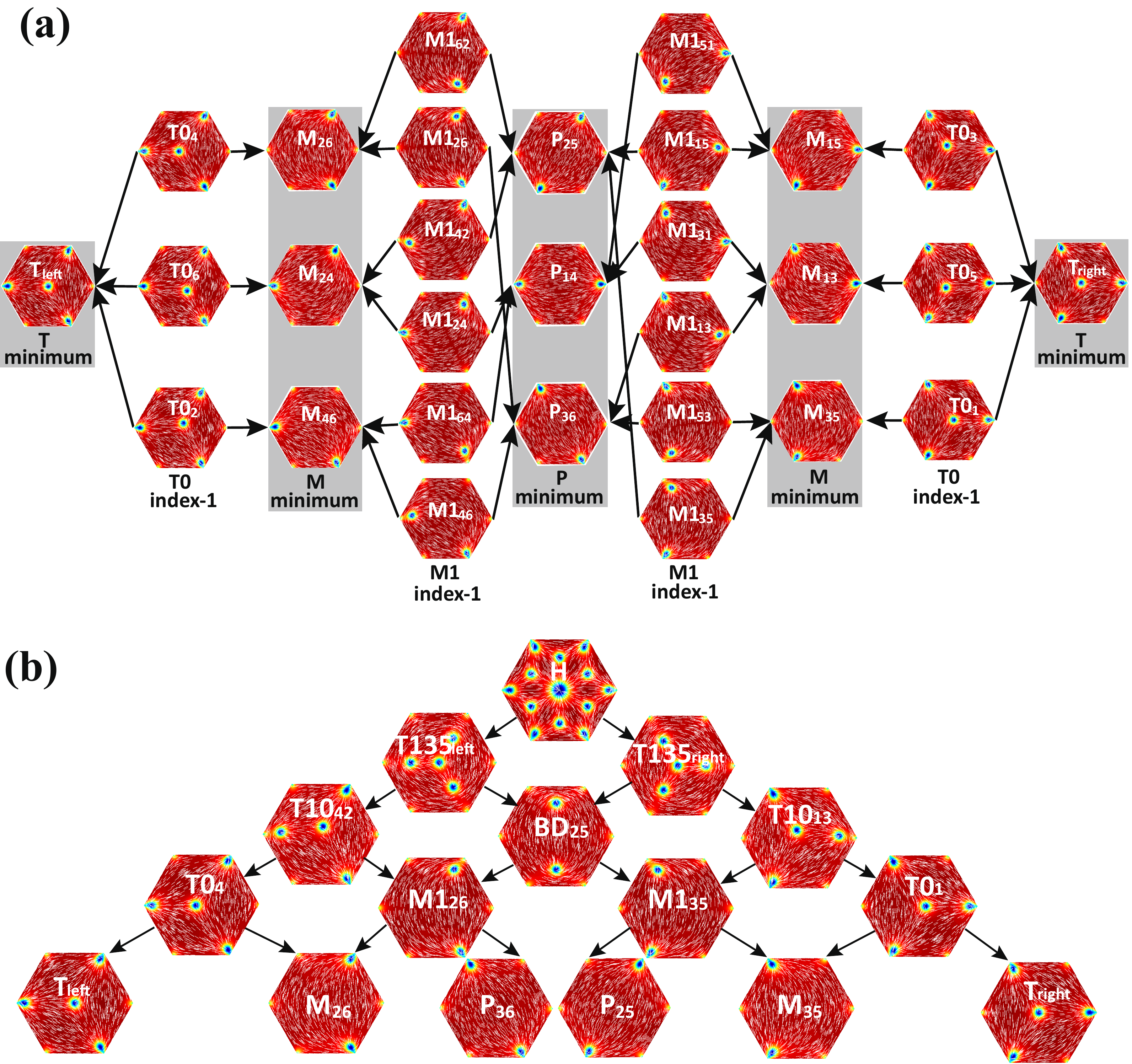}
\caption{(Credit: \protect\cite{han2020SL}) (a) The transition pathways between stable states including two T, six M and three P solutions at $\lambda^2 = 600$. (b) Solution landscape starting from the H solution.
All local minima such as T$_{\mathrm{left}}$, P$_{36}$, M$_{26}$, M$_{35}$, P$_{25}$, and T$_{\mathbf{right}}$ are connected by the index-$8$ H solution.}
\label{600_transition_pathway}
\end{center}
\end{figure}

A further innovative aspect of the solution landscape of NLCs is to provide rich information on dynamical pathways, which include transition pathways with a single transition state, multiple transition states, and dynamical pathways mediated by high-index saddle points. For example, in Figure \ref{600_transition_pathway}(a), the transition pathway between T$_{\mathrm{left}}$ and M$_{26}$ is connected by single transition state T0$_{4}$. On the other hand, one transition pathway between P$_{25}$ and P$_{36}$ can be P$_{25}\leftrightarrow$M1$_{35}\leftrightarrow$M$_{35}\leftrightarrow$M1$_{53}\leftrightarrow$P$_{36}$, connected by two transition states M1$_{35}$ and M1$_{53}$. The ``longest" transition pathway appears to be the pathway between the two T solutions: T$_{\mathrm{left}}$ and T$_{\mathrm{right}}$, despite T$_{\mathrm{left}}$ and T$_{\mathrm{right}}$ are two symmetric solutions related by a $60^{\circ}$ rotation. One switching pathway between T$_{\mathrm{left}}$ and T$_{\mathrm{right}}$ cannot be achieved by T$_{\mathrm{left}}\leftrightarrow$T0$_{4}\leftrightarrow$M$_{26}\leftrightarrow$M1$_{62}\leftrightarrow$P$_{25}\leftrightarrow$M1$_{15}\leftrightarrow$M$_{15}\leftrightarrow$T0$_{3}\leftrightarrow$T$_{\mathrm{right}}$, which indicates the transition between two energetically-close but configurationally-far T solutions has to overcome four energy barriers.

Figure \ref{600_transition_pathway}(b) shows how the different P, M, and T solutions are connected by high-index saddle points.
For example, the index-$1$ M1 connects stable M and P, and the index-$1$ T0 connects stable M and T. Contrarily, two M or two P are connected by the index-$2$ BD. Furthermore, T$_{\mathrm{left}}$ and T$_{\mathrm{right}}$, which are configurationally far away from each other, are connected by an index-$8$ H solution by following HiSD: T$_{\mathrm{left}}$$\leftarrow$T135$_{\mathrm{left}}$$\leftarrow$H$\rightarrow$T135$_{\mathrm{right}}$$\rightarrow$T$_{\mathrm{right}}$.
The index-$8$ H solution is the stationary point in the intersection of the smallest closures of on the energy landscape, which is able to connect every stable solution (two T, three P and six M solutions) through dynamical pathways in Figure \ref{600_transition_pathway}(b).
In \citeasnoun{han2020SL}, the authors deduce that index-$1$ saddle points are efficient for connecting configurationally-close stable solutions.
For configurationally-far stable states, they are generally connected by multiple transition states, or in another way, connected by a high-index saddle point.

\section{Conclusion and future directions}
LCs are fascinating examples of complex fluids that combine fluidity of fluids with the directionality of solids.
During the last decades, the study of LCs has grown tremendously due to the widespread applications in the industry, such as the display devices, the photonic devices and the biological sensors. Modeling, analysis, and computation are indispensable tools for describing, understanding, and predicting the physical phenomena related to the LCs. The topics of LCs land on a number of branches in physics, materials science, and mathematics, forming problems of fundamental importance.

Theoretical studies of LCs are primarily concerned with how topology and geometry of the elements (macromolecules or anisotropic molecules) produce and affect the mesoscopic structure, and how the structure determines the properties of the materials. It is crutial to understand the LC phenomena by building appropriate mathematical models and carry out both theoretical analysis and numerical simulation. Scientific values of LCs include two major aspects:

{\it Raise new mathematical problems}. The study of LCs involves calculus of variation, PDEs, harmonic analysis and computational methods. For example, Onsager model can be expressed as either a variation minimization problem, or a nonlinear nonlocal integral equation or PDE problem. In general, it is impossible to find all analytical solutions for such problem. It is because of the problem with special physical properties, which allows us to find all the solutions and prove that all solutions have axial symmetry. LC is prone to defects, which corresponds to the mathematical singularities. The main reason is because the nature of such problem is a vector field. Mathematicians often use the LCs as the background to study singularities of harmonic maps. 
The $\QQ$-tensor theory is a regularization model for the vector model, involving the mathematical problems such as eigenvalues degradation and multiscale analysis, which provides direction and guidance to the relevant mathematical study. Furthermore, because of multiscale phenomena in complex fluids, it is desirable to study modeling and numerical methods in computational and applied mathematics.

{\it Provide new methodologies}. 
The choice of order parameter in the macroscopic model of LC and is a key to the understanding of the connections between macroscopic theories and microscopic theories. For rod-like liquid crystal,  axially symmetric assumption is generally believed by chemists, physicists, and scientists in materials or mechanics. The strict mathematical proof can not only verify their intuition, but also achieve deep understanding of the mechanism of axial symmetry.
Continuum NLC theories, which assume that macroscopic quantities of interest vary slowly on the molecular length-scale, are typically defined in terms of a macroscopic order parameter and a free energy; the experimentally observable states correspond to energy minimizers. Mathematically, this reduces to analyzing and numerically computing solutions of the associated Euler Lagrange equations which are a system of highly nonlinear and coupled partial differential equations. Recent years have seen a boom in analytical and numerical efforts to compute the solution landscapes of NLC systems.

Although tremendous progress have been achieved in both theoretical and numerical study of LCs over the last four decades, there are still many challenges in modeling and computation of LCs. According to the contents of this review, we list some of perspectives as below.

{\it Symmetry of equilibrium configurations.} Symmetry is the centre theme of LC studies. The axially symmetry of equilibrium solutions of the Onsager's molecular energy plays fundamental role in molecular kinetic theory. However, it is only proved for the Maier-Saupe potential. To prove analogous results for the Onsager potential, even for the minimizers, is an important and difficult problem.
For the Landau-de Gennes energy with isotropic elasticity, it is believed that, the minimizing solutions
corresponding to the hedgehog boundary condition in a three dimensional ball is axially symmetric.
Furthermore, all minimizers should be one of the three types: radial hedgehog, ring disclination and split core.
However, the rigorous proofs are challenging. Similar problems can be asked for the 2D point defects, which are closely related to the disclinations in three dimensional space. If the anisotropic elastic constants are not zero, corresponding problems are more complicated and difficult.
These problems are related to fine structures of the point defects and disclinations, which are of fundamental importance.

{\it Analysis on the isotropic-nematic interface.} Study on profiles and structure of two phases interface is physically important and mathematically challenging. One of the simplest model is the
 energy functional with the order parameter being a scalar function,
which has been widely studied during past two decades. When the order parameter is a high dimensional vector, the problem is much more difficult. \citeasnoun{Lin2012} studied the asymptotical behavior of  minimizers for the high dimensional problem.
The phase transition interface between isotropic and nematic phases is even harder because it is not only high dimensional but also involves anisotropic elasticity. The problem on one dimensional line has been studied in some special cases \cite{Park2017,Chen2018}. The general problem, such as properties of profiles, the boundary condition on the interface and the limit behavior of the equilibrium configuration are needed to be further studied.

{\it Wellposedness, blow-up and long time behavior of solutions to different dynamical models.}
Although there are many works on the wellposedness of solutions to various dynamical models, some basic questions are still unsolved, for examples, the global existence of weak solutions to the three dimensional Ericksen-Leslie system and the inhomogenous Doi-Onsager equation. Constructing solutions which blow up at finite time are also of interest. In addition, the long time dynamics would be complicated, even under a given flow field. For examples, there are many periodic solutions which are called kayaking, wagging, logrolling, et. al. It would be interesting to rigorously justify the existence of these solutions.

{\it Dynamics of point defects, disclinations and two phase interfaces}. In the Ginzburg-Landau theory for superconductors, dynamics of vortices and filaments have been well understood.  The method can not be directly applied for similar problems in LCs, for example, the uniaxial limit of dynamical $\QQ$-tensor system, due to some essential differences: First, the elasticity is usually anisotropic which will cause complex phenomenons such as back flows; Second, the $\QQ$-tensor can escape to biaxial, and the minimizing manifold is isomorphic to the projective space $\mathbb{RP}^2$ which has different topological property with $\mathbb{S}^2$ or $\mathbb{S}^1$; Thirdly, the hydrodynamics will play important roles.
The singularity set of the limit system can be isolated points, lines and surfaces, which has co-dimensional three, two and one respectively.
The evolution of the singularity set is an interesting and challenging problems which may rely on
deeper understanding on the profiles near the singularity set. 

{\it Relationship between different dynamical models}. The rigorous derivation of the Ericksen-Leslie theory from the Doi-Onsager  theory
are only established for smooth solutions with isotropic elasticity ($k_1=k_2=k_3$). To generalize it to the anisotropic elasticity case would be
an important problem. Moreover, as the defect solutions can only be described by weak solutions in the vector theory, justification of convergence for weak solutions is an important problem. However, analogous results for weak solutions are very few.

{\it Solution landscape of confined liquid crystals}.
NLCs in confinement typically exhibit multiple stable and unstable states, which correspond to stationary solutions of a free energy. Solution landscape of confined NLCs, a hierarchy of connected solutions, is able to provide a rich and insightful information of the physical properties of such multisolution problems. The HiSD is an efficient numerical algorithm for computing saddle-point solutions for LC systems, and has been successfully applied to construct the solution landscapes of the NLCs confined in a square and a hexagon. Solution landscape reveals not only transition pathways between stable solutions connected by index-1 saddle points, but also innovative dynamical pathways mediated by high-index saddle points. Several challenging analytic and numerical questions remain for the solution landscape, for instance, the completion of solution landscape. Can we obtain bounds for the Morse index of the parent state in the solution landscape? How to estimate the number of the stationary solutions of a given free energy? Building theoretical framework of solution landscape is a big challenge as well as a great opportunity for applied mathematicians to make contributions.

Finally, we would like to emphasize again that we only focus on the simplest, NLC systems in the content of this review, and have not touched much upon the other LC systems. There is a crucial need for exhaustive theoretical and numerical studies of complex and unconventional LC systems, to complement thriving experimental work. We list a few samplers in the remaining part, in which the theory is relatively open, the energies are more complex, and the solutions are less regular with new analytic and numerical challenges.

{\it Phase transitions in multiphase liquid crystal systems}.
Recent experiments show NLC-cholesteric (by adding chiral particles), NLC-smectic (by lowering the temperature) phase transitions in LC systems \cite{siavashpouri2017molecular}. Compared with single phase LC, these heterogeneous systems are inherently more complex with multiple competing NLC, cholesteric/smectic and mixed solutions, novel solutions with phase separation, and interactions between phases \cite{gottarelli1983induction}. Multiple phases in the heterogeneous system need to be analyzed under the unified theory framework. For instance, the complex N-S or N-C systems can be analyzed in terms of LdG-type energies with multiple order parameters for coupled system with smectic/cholesteric terms. The heterogeneous systems can also be analyzed by molecular models. On the other hand, the accurate structures (e.g. interfaces between phases) and complex defects, have multiscale properties. The theory is relatively open with new analytic and numerical challenges.

{\it Ferroelectric and ferromagnetic liquid crystals}.
LC is sensitive to external fields such as electric field and magnetic field \cite{mertelj2013,mertelj2017,Chen14021,bisht2020tailored}. 
Due to the photoelectric effect, LCs are applied in display devices and have revolutionized the display industry. Besides normal LC, ferroelectric LC as a new LC material has attracted extensive attention. Ferroelectric LC materials have the characteristics of high response speed, high contrast, high resolution and large capacity information display. It has a good application prospect in the fields of display, optical interconnections and optical information processing. One future direction is to develop suitable LC models from molecular model to macroscopic models to study the photoelectric effect of ferroelectric LCs and the design of voltage waveform.


{\it Active liquid crystal systems}.
The non-equilibrium phenomena studied in active matter systems are widespread in nature, especially playing an important role in biological phenomena \cite{marchetti2013hydrodynamics,keber2014topology,prost2015active}. The active matter system is a typical non-equilibrium system, which is composed of a large number of active particles. These particles obtain energy for self-propulsion by converting other forms of energy into kinetic energy \cite{zhou2014living}. Due to the interaction between particles and particles, and particles and media, the entire system can exhibit extremely rich dynamic phenomena macroscopically. Active LC system is a paradigm of active systems for biologically inspired complex fluids with orientational alignment. In active LC, the rod-like constituents endow the fluid they are immersed in with active stresses. For instance, in cell motivity the motility mechanism arising by the interaction of myosin and actin. Their dynamic properties, phase transition process, and external field response are worthy of study and exploration.

\section*{Acknowledgement}
We would like to thank Dr. Yucen Han, Dr. Yiwei Wang, and Mr. Jianyuan Yin for providing great assistance. We also thank Professors Yucheng Hu, Yuning Liu, Apala Majumdar, Jinhae Park, and Zhifei Zhang for fruitful discussions and helpful comments. Wei Wang was partially supported by the National Natural Science Foundation of China No. 11922118 and 11931010. Lei Zhang was partially supported by the National Natural Science Foundation of China No. 11861130351 and the Royal Society Newton Advanced Fellowship, in partnership with Apala Majumdar from Strathclyde. Pingwen Zhang was partially supported by the National Natural Science Foundation of China No. 21790340, 11421101.

\section{Appendix}
\subsection{Elementary identities for the high order moment}

Let $\QQ_2[f]=\langle\mm\mm-\frac{1}{3}\II\rangle_f,$
and define $\QQ_4[f]$ as follows:
\begin{eqnarray}\label{def:Q4}
Q_{4\alpha\beta\gamma\mu}[f]\!\!\!\!\!&=&\!\!\!\!\Big\langle{m}_{\alpha}{m}_{\beta}{m}_{\gamma}{m}_{\mu}
-\frac{1}{7}(m_{\alpha}m_{\beta}\delta_{\gamma\mu}+
m_{\gamma}m_{\mu}\delta_{\alpha\beta}+m_{\alpha}m_{\gamma}\delta_{\beta\mu}
+m_{\beta}m_{\mu}\delta_{\alpha\gamma}\nonumber\\
&&+m_{\alpha}m_{\mu}\delta_{\beta\gamma}+
m_{\beta}m_{\gamma}\delta_{\alpha\mu})+\frac{1}{35}(\delta_{\alpha\beta}\delta_{\gamma\mu}
+\delta_{\alpha\gamma}\delta_{\beta\mu}+\delta_{\alpha\mu}\delta_{\beta\gamma})\Big\rangle_f.
\end{eqnarray}
Let $P_k(x)$ be the $k$-th Legendre polynomial and
\begin{align}
  S_2=\langle{P}_2(\mm\cdot\nn)
\rangle_{h_\nn},~S_4=\langle{P}_4(\mm\cdot\nn)\rangle_{h_\nn},
\end{align}
which only depend on $\alpha$. Firstly, we have a lemma:
\begin{lemma}\label{Lem:Q-tensor}
\begin{eqnarray*}
\QQ_2[h_\nn]&=&S_2(\nn\nn-\frac{1}{3}),\\
Q_{4\alpha\beta\gamma\mu}[h_\nn]&=&S_4\Big(n_{\alpha}n_{\beta}n_{\gamma}n_{\mu}
-\frac{1}{7}(n_{\alpha}n_{\beta}\delta_{\gamma\mu}+
n_{\gamma}n_{\mu}\delta_{\alpha\beta}+n_{\alpha}n_{\gamma}\delta_{\beta\mu}\nonumber\\
&&+n_{\beta}n_{\mu}\delta_{\alpha\gamma}+n_{\alpha}n_{\mu}\delta_{\beta\gamma}+
n_{\beta}n_{\gamma}\delta_{\alpha\mu})+\frac{1}{35}(\delta_{\alpha\beta}
\delta_{\gamma\mu}+\delta_{\alpha\gamma}\delta_{\beta\mu}
+\delta_{\alpha\mu}\delta_{\beta\gamma})\Big).~~~~~~~~
\end{eqnarray*}
\end{lemma}
\begin{corollary}\label{Cor:M-tensor}
Let $\MM=\langle\mm\mm\rangle_{h_\nn},~\MM^{(4)}=\langle\mm\mm\mm\mm\rangle_{h_\nn}$, there holds that
\begin{eqnarray}\label{M2}
\MM&=&S_2\nn\nn+\frac{1-S_2}{3}\II,\\
M^{(4)}_{\alpha\beta\gamma\mu}&=&S_4n_{\alpha}n_{\beta}n_{\gamma}n_{\mu}
+\frac{S_2-S_4}{7}(n_{\alpha}n_{\beta}\delta_{\gamma\mu}+
n_{\gamma}n_{\mu}\delta_{\alpha\beta}+n_{\alpha}n_{\gamma}\delta_{\beta\mu}
+n_{\beta}n_{\mu}\delta_{\alpha\gamma}\nonumber\\&&
+n_{\alpha}n_{\mu}\delta_{\beta\gamma}+
n_{\beta}n_{\gamma}\delta_{\alpha\mu})+\big(\frac{S_4}{35}-\frac{2S_2}{21}
+\frac{1}{15}\big)(\delta_{\alpha\beta}
\delta_{\gamma\mu}+\delta_{\alpha\gamma}\delta_{\beta\mu}
+\delta_{\alpha\mu}\delta_{\beta\gamma}).~~~~~~~~\label{M4}
\end{eqnarray}
\end{corollary}

\subsection{Some identities used in Section \ref{sec:deriv-momentum}.}
For any constant vector $\uu,~\ww\in\mathbb{R}^3$, and a vector field $\vv(\mm)$ defined on $\BS$,
\begin{eqnarray*}
&&\uu\cdot\Big(\int_{\BS}(\mm\mm-\frac{1}{3}\II)\CR\cdot(f\vv)\ud\mm\Big)\cdot\ww\\
&=&\int_{\BS}\Big((\mm\cdot\uu)(\mm\cdot\ww)-\frac{1}{3}\uu\cdot\ww\Big)\CR\cdot(f\vv)\ud\mm\\
&=&-\int_{\BS}\CR\Big((\mm\cdot\uu)(\mm\cdot\ww)\Big)\cdot(f\vv)\ud\mm\\
&=&-\int_{\BS}\mm\times\uu\cdot\vv(\mm\cdot\ww)f+(\mm\cdot\uu)(\mm\times\ww)\cdot\vv{f}\ud\mm\\
&=&\uu\cdot\langle\mm\times\vv\mm+\mm\mm\times\vv\rangle\cdot\ww.
\end{eqnarray*}
Hence, we have:
\begin{eqnarray}\label{eq:stress1-sym}
\int_{\BS}(\mm\mm-\frac{1}{3}\II)\CR\cdot(f\vv)\ud\mm=\langle\mm\times\vv\mm+\mm\mm\times\vv\rangle.
\end{eqnarray}
Let $\vv=\CR{\mu}$, we have
\begin{eqnarray}
\int_{\BS}(\mm\mm-\frac{1}{3}\II)\CR\cdot(f\CR\mu)\ud\mm=\langle\mm\times\CR\mu\mm+\mm\mm\times\CR\mu\rangle.
\end{eqnarray}
Let $\vv=\mm\times\kappa\cdot\mm$, we have
\begin{eqnarray}
&&\int_{\BS}(\mm\mm-\frac{1}{3}\II)\CR\cdot(f\mm\times\kappa\cdot\mm)\ud\mm\nonumber\\
&=&\langle\mm\times(\mm\times\kappa\cdot\mm)\mm+\mm\mm\times(\mm\times\kappa\cdot\mm)\rangle\nonumber\\
&=&\langle2(\mm\cdot\kappa\cdot\mm)\mm\mm-(\kappa\cdot\mm)\mm-\mm(\kappa\cdot\mm)\rangle\nonumber\\
&=&2\DD:\langle\mm\mm\mm\mm\rangle-\DD\cdot\langle\mm\mm\rangle
+\BOm\cdot\langle\mm\mm\rangle-\langle\mm\mm\rangle\cdot(\DD+\BOm).\label{identity-3}
\end{eqnarray}

%

\begin{Lemma}\label{Lem:antisymmetric}
For any antisymmetric constant matrix $\BOm$, we have
\begin{align*}
\CR\cdot\big(\mm\times(\BOm\cdot\mm){f}_0\big)&=(\nn\times(\BOm\cdot\nn))\cdot\CR{f_0},\\
\CR\cdot\big(\mm\times(\DD\cdot\mm){f}_0\big)&=\frac12\CR\cdot\big(f_0\CR (\mm\mm:\DD)\big).
\end{align*}
\end{Lemma}
\no{\bf Proof.}\,
The lemma is a direct consequence of the following identities
\begin{align*}
&\CR\cdot\big(\mm\times(\BOm\cdot\mm)\big)=\CR_i(\epsilon^{ijk}m_j\Omega_{kl}m_l)=(\II-3\mm\mm):\BOm=0,\\
&(\mm\times(\BOm\cdot\mm))\cdot\CR{f_0}=(\mm\times(\BOm\cdot\mm))\cdot(\mm\times\nn)f_0'\\
&\quad=(\nn\times(\BOm\cdot\nn))\cdot(\mm\times\nn)f_0'=(\nn\times(\BOm\cdot\nn))\cdot\CR{f_0}.
\end{align*}
This completes the proof. \ef

\bibliographystyle{actaagsm}
\bibliography{ActaNumerics0929}
\label{lastpage}
\end{document}